\documentclass[11pt]{article}

\usepackage{longtable}
\usepackage{tabularx}
\usepackage{tikz}
\usetikzlibrary{shapes}
\usetikzlibrary{decorations.pathreplacing}
\usepackage{enumerate}
\usepackage[shortlabels]{enumitem}
\usepackage{verbatim}
\usepackage[margin=1in]{geometry}
\usepackage{amssymb}
\usepackage{caption}
\usepackage[normalem]{ulem}
\usepackage{bbm}
\usepackage{amsthm}
\usepackage{enumitem}
\usepackage{authblk}
\usepackage{scalerel}    
\usepackage{stmaryrd}   

\usetikzlibrary{arrows.meta}

\DeclareFontFamily{U}{mathx}{\hyphenchar\font45}
\DeclareFontShape{U}{mathx}{m}{n}{
      <5> <6> <7> <8> <9> <10>
      <10.95> <12> <14.4> <17.28> <20.74> <24.88>
      mathx10
      }{}
\DeclareSymbolFont{mathx}{U}{mathx}{m}{n}
\DeclareFontSubstitution{U}{mathx}{m}{n}
\DeclareMathAccent{\widecheck}{0}{mathx}{"71}
\DeclareMathAccent{\wideparen}{0}{mathx}{"75}

\usepackage{rotating,centernot,cancel}    

\newcommand{\myfootnote}[1]{
    \renewcommand{\thefootnote}{}
    \footnotetext{\scriptsize#1}
    \renewcommand{\thefootnote}{\arabic{footnote}}
}

\usepackage{amsmath}
\usepackage{tocloft}
\usepackage{float}

\usepackage{tcolorbox}   
\PassOptionsToPackage{hyphens}{url}
\usepackage[hidelinks]{hyperref}
\usepackage{babel}
\theoremstyle{plain}

\input amssym.def
\input amssym.tex

\def\beqn{\begin{eqnarray}}
\def\eeqn{\end{eqnarray}}

\def\beq{\begin{equation}}
\def\eeq{\end{equation}}
\usepackage{mathtools}
\DeclarePairedDelimiter\floor{\lfloor}{\rfloor}

\newtheorem{theorem}{Theorem}[section]
\newtheorem{corollary}[theorem]{Corollary} 
\newtheorem{lemma}[theorem]{Lemma} 
\newtheorem*{lemma*}{Lemma}
\newtheorem{proposition}[theorem]{Proposition} 

\theoremstyle{remark}
\newtheorem{remark}[theorem]{Remark}

\theoremstyle{definition}

\numberwithin{figure}{section}
\numberwithin{equation}{section}

\def\R{\mathbb R}

\def\Z{\mathbb Z}

\def\to{\Bigarrow}
\def\dlim[#1][#2]{\lim_{#1 \to #2, #1 \neq #2}}

\def\Var{\textup{$\mathbb{V}$ar}}

\def\Cov{\textup{$\mathbb{C}$ov}}
\def\Corr{\textup{$\mathbb{C}$orr}}
\def\pb{\mathbb{P}}

\newcommand{\be}{\begin{equation}}
\newcommand{\ee}{\end{equation}}

\newcommand{\nn}{\nonumber}

\def\ind{\mathbf{1}}

\def\wt{\widetilde}

\newcounter{cnstcnt}
\newcommand{\newc}{%
\refstepcounter{cnstcnt}%
\ensuremath{C_{\thecnstcnt}}}
\newcommand{\oldc}[1]{\ensuremath{C_{\ref*{#1}}}}

\usepackage{xargs}
\usepackage[colorinlistoftodos,prependcaption,textsize=footnotesize]{todonotes}
\newcommandx{\addmath}[2][1=]{\todo[linecolor=red,backgroundcolor=red!25,bordercolor=red,#1]{#2}}
\newcommandx{\fixtext}[2][1=]{\todo[linecolor=blue,backgroundcolor=blue!25,bordercolor=blue,#1]{#2}}
\newcommandx{\note}[2][1=]{\todo[linecolor=yellow,backgroundcolor=yellow!25,bordercolor=yellow,#1]{#2}}

\def\barevent{\mathcal{B}_{\textup{bar}}}

\newcommand\bbullet{{{\scaleobj{0.6}{\bullet}}}} 

\title{Temporal correlation in the inverse-gamma polymer}
\date{}
\author{
  Riddhipratim Basu\thanks{\scriptsize{International Centre for Theoretical Sciences, Tata Institute of Fundamental Research, Bengaluru, INDIA.
 \texttt{rbasu@icts.res.in}}} \qquad Timo Sepp\"{a}l\"{a}inen\thanks{\scriptsize{Department of Mathematics, University of Wisconson-Madison, Wisconsin, USA. \texttt{seppalai@math.wisc.edu}}
} \qquad  Xiao Shen\thanks{\scriptsize{Department of Mathematics, University of Utah, Utah, USA. \texttt{xiao.shen@utah.edu}}}
  }

\begin{document}

\allowdisplaybreaks

\maketitle

\begin{abstract}
Understanding the decay of correlations in time for (1+1)-dimensional polymer models in the KPZ universality class has been a challenging topic. Following numerical studies by physicists, concrete conjectures were formulated by Ferrari and Spohn  \cite{Ferr-Spoh-2016} in the context of planar exponential last passage percolation. These have mostly been resolved by various authors. In the context of positive temperature lattice models, however, these  questions have remained open. We consider the time correlation problem for the exactly solvable inverse-gamma polymer in $\Z^2$. We establish, up to constant factors,  upper and lower bounds on the correlation between free energy functions for two polymers rooted at the origin (droplet initial condition) when the endpoints are either close together or far apart. We find   the same exponents as predicted in \cite{Ferr-Spoh-2016}. Our arguments rely on the understanding of stationary polymers, coupling, and random walk comparison. We use recently established moderate deviation estimates for the free energy. In particular, we do not require asymptotic analysis of complicated exact formulae. 
\end{abstract}

\myfootnote{Date: \today}
\myfootnote{2010 Mathematics Subject Classification. 60K35, 	60K37}
\myfootnote{Key words and phrases: correlation, coupling, directed polymer, Kardar-Parisi-Zhang, random growth model.}

\tableofcontents

\section{Introduction}

\subsection{Universality in stochastic growth} 
Random growth models have always been at the heart of probability theory. The simplest example of random growth, {occurring in zero spatial dimension plus one time dimension,} is a sum of independent and identically distributed (i.i.d.) random variables. Provided their second moment is finite, the large-scale behavior  of the centered sum  is independent of the distribution  of the summands, as described by the \textit{central limit theorem}. With the fluctuation exponent $1/2$ and the Gaussian distribution as the central limit scaling law, this model is a member of the \textit{Gaussian universality class}.

Following the seminal 1986 physics work of Kardar, Parisi and Zhang \cite{Kar-Par-Zha-86}, one major goal of recent  probability research has been to demonstrate that very  different  universal behavior arises in a wide class of stochastic models with spatial dependence. Extensive computer simulations, non-rigorous physical arguments, laboratory experiments, and rigorous mathematical results have all suggested that this \textit{Kardar-Parisi-Zhang  universality class} (KPZ) is rich. It includes interacting particle systems, percolation models, polymer models,  random tilings, certain stochastic PDEs and more. All known members of the KPZ class share universal fluctuation exponents and {many of their limiting distributions} are from random matrix theory \cite{introkpz1, introkpz2}. 

In the past twenty-five years,   many ground-breaking advances in understanding  KPZ universality   have come through the study of \textit{exactly solvable} or \textit{integrable} models. However, for the  vast majority of the conjectural members of the KPZ class,  these fine techniques of integrable probability, representation theory, and algebraic combinatorics do not apply. With the eventual goal of extending results beyond the integrable cases, a second line of  research uses in principle broadly applicable probabilistic techniques and geometric arguments to study the integrable models. This paper falls in the latter category.  We study the temporal correlation in the inverse-gamma polymer model,  originally introduced in  \cite{poly2}. 

In the remainder of this introduction, Section \ref{sec:methods} gives a brief overview of the presently used  mathematical methods in KPZ study, Section \ref{sec:corr} discusses the correlation problem studied in this paper, and Section \ref{sec:org} explains 
the organization of the rest of the paper.


\subsection{Methods in the  study of the KPZ class} 
\label{sec:methods} 

Several different approaches to studying the exactly solvable models of the KPZ class have emerged over the last 25 years.  We describe these methods briefly on a very general level, mainly  in the context of zero-temperature last-passage percolation (LPP) with exponential or geometric weights, where mathematical development is farthest along.  

\subsubsection{Integrable probability}
\label{s:int}
For exactly solvable LPP models based on RSK correspondence or similar remarkable bijections, it is possible to write down explicit formulas for one-point and multi-point distributions. Integrable probability estimates refer to estimates and asymptotics obtained by careful analysis of these formulas. Beginning with the seminal work of Baik, Deift and Johansson 
\cite{Bai-Dei-Joh-99}, 
which established the Tracy-Widom scaling limit for the longest increasing subsequence problem, this approach has brought much success. This includes process limits for the last-passage time profile started from particular initial conditions: {the droplet initial condition \cite{Airy_2}, the flat initial condition \cite{Airy_1, Airy_11}, and the stationary initial condition \cite{converge_BK}. 
More recently, the seminal work \cite{kpz_fix} constructed the KPZ fixed point which admits general initial conditions. 
Formulas for two-time distributions have also been obtained \cite{ Baik-Liu-19, two_time_LPP, Liao-22, Liu-22}.}

\subsubsection{Gibbsian line ensembles}
\label{s:gibbs}
A useful approach based on   resampling in line ensembles   was introduced by Corwin and Hammond \cite{bgibbs}. 
In the zero temperature model of Brownian last-passage percolation, where the corresponding line ensemble has the \emph{Brownian Gibbs property}, a detailed understanding of the passage time profile was obtained in a series of works \cite{cont_profile, quilt, disjoint_blpp, BR}. A similar approach exists for the positive temperature KPZ equation and has been used recently to great effect \cite{KPZ_line}. The Gibbsian line ensemble approach led to the construction of the \emph{directed landscape} (DL),  the space-time scaling limit of zero temperature models \cite{DL, DL_LPP}.  Subsequently, DL limits were established for the KPZ equation \cite{KPZ_DL_v, KPZ_DL}. 

\subsubsection{Percolation methods with integrable inputs}
\label{s:perc}
Yet another suite of techniques uses black-box integrable inputs together with probabilistic and geometric arguments that can in general be referred to as \emph{percolation arguments}. The inputs typically used in this line of work include (i)  uniform curvature of the limit shape,  (ii)  moderate deviation estimates of the passage time, and (iii) convergence of the   one point distribution  to the GUE Tracy-Widom law  that has a negative mean \cite{small_deviation_LPP, geowater}. Some cases  require more sophisticated inputs such as the Airy process limit of the full profile \cite{timecorrflat}. These inputs are typically obtained from the first approach above.  In cases like exponential and Brownian LPP, one can also exploit  random matrix connections to obtain similar, albeit usually a bit weaker, estimates \cite{led-rid-2021}. These estimates are then applied  to obtain fine information about the geodesic geometry, which in turn provides further information about the space-time profile of last-passage times. An axiomatic framework for these types of arguments has been developed and used in \cite{geowater}. 

\subsubsection{Coupling methods}
\label{s:coup}
The most probabilistic  approach that minimizes the role  of integrability utilizes couplings with stationary  growth processes. In zero temperature the seminal work was   \cite{Cat-Gro-06}, followed by \cite{cuberoot}, and \cite{poly2} began this development in positive temperature polymer models.  This effort has been  recently revolutionized by \cite{rtail1} that made possible certain quantitatively optimal bounds.  Presently this approach still relies on a special feature of the model, namely, that the stationary measure is explicitly known.   It has been most successful in the study of solvable models.  Its applicability does extend to some stochastic processes presently not known to be fully integrable, namely classes of zero-range processes and interacting diffusion 
\cite{Bal-Kom-Sep-12-jsp, kpz_id}.
Through comparisons with the stationary process, many results about the geodesics, parallel to those developed by the previous approach, have been proved  {\cite{balzs2019nonexistence, balzs2020local, uni_tree, seppcoal}}.  
Following the optimal bounds of  \cite{rtail1},  some of the integrable  inputs of  the percolation approach can now be supplied by coupling techniques, thereby reducing dependence on random matrix theory and integrable probability. 

\subsubsection{The approach of this paper}
The current paper uses a combination of the final two approaches discussed above to study the temporal decay of correlations   in the positive-temperature exactly solvable inverse-gamma polymer model. The major barrier to applying the percolation arguments from \cite{timecorriid, timecorrflat} has been the lack of one-point moderate deviation estimates. One advantage of the coupling  techniques is that they can be extended from  zero temperature to positive temperature \cite{Emr-Jan-Xie-22-, Xie2022}. In the context of the semi-discrete   O'Connell-Yor polymer, one-point estimates have recently been obtained under stationary initial conditions \cite{Lan-Sos-22-a-} and more recently for the point-to-point problem \cite{OC_tail}. These techniques carry over to the inverse-gamma polymer model as well.  This opens the door for proving versions of the lattice LPP results obtained through one-point estimates, now in the context of the positive temperature polymer models. {Our paper provides the first example in this vein (the fusion of percolation and coupling arguments), by proving the first bounds on the time 
correlation structure of a lattice polymer model.} 

We expect similar techniques to be applicable to a number of other related problems. {To emphasize, although our approach here is similar to the one described in \S~\ref{s:perc}, we only require the one point moderate deviation inputs, and these are provided by the recent advances in the coupling/stationary polymer approach of \S~\ref{s:coup}.  Therefore our work does not rely on the integrable methods described in \S~\ref{s:int}-\ref{s:gibbs}}.

\subsection{Time correlation problem} \label{sec:corr}
We turn to the precise problem we study, its history,  and our contributions.

A central object in KPZ models is the random  \textit{height function} $h:\R\times\R_{\ge0} \rightarrow\R$.   Depending on the situation studied,    $h(x, t)$ can be the height of a randomly  moving interface over spatial location $x$ at  time $t$, the passage time on the plane from the origin to $(x,t)$, or the free energy of point to point polymers between the origin and location $(x,t)$. 

The spatial statistics of $x\mapsto h(x,t_0)$ at a fixed time $t_0$ are  much better  understood than  the temporal  process $t \mapsto h(x_0, t)$.    Multi-time joint distributions of the height function have been obtained in several exactly solvable models \cite{Baik-Liu-19,  Joha-17, Joha-19, Joha-Rahm-21, Liao-22, Liu-22}. However, it has remained difficult to extract useful information from these impressive formulas.

Short of capturing the full distribution of the temporal evolution, a  natural object to   study   is   the two-time correlation function 
\begin{equation}\label{hcorr}
\Corr (h(0, t_1), h(0, t_2)),\end{equation}
where we have now singled out the origin $x_0=0$ as the spatial location.  
This correlation was first studied by physicists Takeuchi and Sano \cite{Take-Sano-2012}, who  measured the quantity \eqref{hcorr} from a turbulent liquid crystal experiment. Subsequently came  numerical simulations \cite{Take-2012} that predicted the behavior of \eqref{hcorr} by fixing $t_1$ and sending $t_2$ to infinity.

\subsubsection{Prior rigorous time correlation results}

Ferrari and Spohn \cite{Ferr-Spoh-2016} studied the large-time behavior of \eqref{hcorr} from various initial conditions in the exponential last-passage percolation, which is one of the most-studied zero-temperature KPZ last-passage growth model on the lattice. Taking   time   to infinity, they obtained  a variational formulation of the (rescaled) height function in terms of two independent \textit{Airy processes}. From the variational problem they derived an explicit formula for the limiting two-time covariance under the stationary initial distribution, as $t_1, t_2$ both tend to infinity. For the step and flat initial conditions  \cite{Ferr-Spoh-2016}  conjectured  asymptotics in the regimes $t_1/t_2\rightarrow0$   and $t_1/t_2\rightarrow1$. 



Following the conjectures of \cite{Ferr-Spoh-2016}, several  rigorous works studied this problem under  different initial conditions in the zero-temperature setting.

The time correlation problem for the droplet initial condition in  exponential LPP was solved in two parallel works.  Both  employed a combination of integrable inputs and a geometric study of geodesics. The results of \cite{ferocctimecorr}, which also utilizes comparison with stationary processes, are limiting in nature and also used the convergence of the passage time profile to the Airy$_2$ process. They also obtained an exact formula for the stationary case and identified universal behavior with respect to the initial condition when the two time points are close to one another. In contrast, \cite{timecorriid} used one point estimates, convergence to Tracy-Widom GUE distribution (and the negativity of its mean), together with geometric arguments, to obtain similar, but quantitatively weaker, results for the droplet initial condition, but valid  also in the pre-limit setting. When the two time points are far away, the case of the flat initial condition was dealt with in \cite{timecorrflat}. This work relied on strong Brownian comparison results for the Airy$_2$ process, in addition to convergence to it. 
{The time correlation problem in the half-space exponential LPP has also been recently studied in \cite{fer-occ-2022}, utilizing comparison arguments with its stationary version and the process limit obtained in \cite{half_process}. }

The Gibbsian line ensemble approach has also been useful in this context. In an unpublished work, Corwin and Hammond solved the time correlation problem in   Brownian LPP with this approach.  Subsequently, together with Ghosal, they extended their work to the positive temperature KPZ equation \cite{KPZcorr}. 

\subsubsection{Our work: temporal correlations in positive temperature on the lattice}
Prior to the present work,  there does not appear to be any mathematically rigorous work on this problem for positive temperature lattice models. The application of Gibbsian line ensemble techniques seems challenging for lattice models, {due to the absence of explicit calculations available for random walks compared to the Brownian motion}. The one-point convergence to Tracy-Widom GUE is known for the inverse-gamma polymer \cite{TW_polymer, kri-qua-2018, Bor-Cor-Rem-13}; { very recently, since our work was completed, the convergence of the free energy profile has also been shown in \cite{Airy_char}. Hence, the approach of \cite{ferocctimecorr} might also be feasible in the positive temperature case if one were to study the limiting regime,  but we are interested in the finite size estimates as well}.  

Our approach is inspired by \cite{timecorriid} and the recent progress in stationary techniques. One cannot directly apply the techniques of \cite{timecorriid} in the positive temperature set-up, as much of it refers to the fluctuations and coalescence of \emph{geodesics} which do not exist in our setting. We modify their definitions appropriately and construct events in terms of the free energy profile and restricted free energies, which can serve similar purposes.  Certain estimates are directly proved using stationary techniques.   The novel technical ingredients of our paper are developed in these directions. 

{In Section \ref{loc_fluc}}, we directly prove the locally diffusive behavior of the free energy profile instead of utilizing local fluctuations of geodesics as in \cite{timecorriid}. {For the lower bound in Theorem \ref{thm_r_small}}, we use the FKG inequality as in \cite{timecorriid}, but in the absence of geodesics, the resampling argument is significantly different, and, in fact, somewhat simpler. {In Section \ref{sec:nr}}, we give a direct proof of a lower bound of the difference between the expected free energy and its long-term value, at the standard deviation scale.  This way we avoid the  need for  the Tracy-Widom limit, and in fact,  this  provides a new proof  of the negativity of the mean of the Tracy-Widom distribution.  Our arguments carry over to the zero temperature setting as well, thus eliminating  the integrable probability inputs from  the LPP results of \cite{timecorriid}. 

To summarize, {in Theorem \ref{thm_r_large} and Theorem \ref{thm_r_small},} we establish the exponents that govern the decay of correlations in the time direction for the inverse-gamma polymer model. As expected on universality grounds,  the exponents are the same as in the zero-temperature case. Ours   is the first such result in a lattice polymer model in the KPZ class.   The only  special feature  of the model we use is  the explicit description of the stationary process. In particular, we do not use any weak convergence result (to the Tracy-Widom distribution, for example). Our techniques consist of   one-point estimates obtained through  stationary polymers, random walk comparisons, and  percolation arguments. Ours is the first instance where the stationary polymer techniques have been put together with the percolation arguments in a positive temperature setting.  This combination can be useful for  extending many  zero-temperature results to the inverse-gamma polymer. 

That 
  our approach does not rely on integrable inputs is not only  
 potentially useful for  future extensions, but also necessary in the current state of the subject. There are fewer  integrable tools available for our model in comparison with exponential LPP, Brownian LPP or the KPZ equation. There is no determinantal formula for the multi-point joint distribution of the free energy, and there is no corresponding Brownian Gibbs property in this discrete setting (unless one takes a certain limit of the model).
Lastly, the inverse-gamma polymer model sits higher up in the hierarchy of the KPZ models. This means that through appropriate transformations and limits,  LPP, BLPP, and the KPZ equation can be derived from the inverse-gamma polymer.  In consequence, our results should carry over to these other models and thereby {remove the inputs from integrable probability} utilized in previous works.

\subsection{Organization of the paper}\label{sec:org} 

The polymer model is defined and our main results on the correlation bounds, Theorems \ref{thm_r_large} and \ref{thm_r_small}, are stated in Section \ref{main_results}.  Theorem \ref{thm_r_large} is proved in Section \ref{pf1} and Theorem   \ref{thm_r_small} in Section  \ref{pf2}. 
Auxiliary results needed for the main proofs are collected in Sections \ref{s:aux1} and \ref{loc_fluc}.   We treat the proofs of these auxiliary results differently depending on their status. Those that require significant proof are verified in Sections \ref{sec:nr} and \ref{sec_tail}, while those based on existing work, such as analogous zero-temperature results, are in the appendices. 
Next, we explain the organization of the supporting results in more detail. 

Section \ref{notation} contains additional notation and conventions, in particular,  for various subsets of $\Z^2$ and  partition functions of restricted collections of paths.  
Section \ref{sec:reg} collects regularity properties of the shape function. Nothing beyond calculus is used here.  

Section \ref{free_est} covers various estimates for the free energy, organized into several subsections.
\begin{itemize} 
\item Section \ref{mod_est} gives   moderate deviation estimates for the point-to-point free energy. Two estimates for the left tail that appear here are used multiple times in the paper and  proved in Section \ref{sec_tail}. 

\item Sections \ref{sec:Zpath}--\ref{sec:Zvar}  contain a a variety of estimates.   These   are used   only for the lower bound of Theorem \ref{thm_r_small}. 
 Those that have previously appeared in the zero-temperature setting have their proofs in Appendix \ref{free fluc}. 

\item Section \ref{non_rand} gives a lower bound on the discrepancy between the asymptotic free energy and the finite-volume  expected free energy, sometimes called the \textit{non-random fluctuation}. It is proved in Section \ref{sec:nr} by comparison with  the increment-stationary polymer.   This result is used in the  proof of the lower bound of the left tail in Section \ref{mod_est} and the construction of the Barrier event $\barevent$ in Section 6.1. 
\end{itemize}

Section \ref{stat_poly} introduces the increment-stationary inverse-gamma polymer and discusses some of its properties. The proofs for these properties can be found in Section \ref{sec:nr}. Among  the results here are  upper and lower bounds on  the free energy difference between the stationary model and the i.i.d.~model.

Section \ref{sec_rw}  presents a random walk comparison of the free energy profile. Specifically, we establish upper and lower bounds on the free energy along a down-right path using two random walks. The proof of this comparison, which relies on the stationary polymer process,  can be found in Appendix \ref{rwrw_proof}.

 Section \ref{loc_fluc} is dedicated to  local fluctuations in the free energy profile. The proofs in this section rely on the moderate deviation estimates and the random walk comparison. The results obtained here are crucial for the proofs of  the main theorems.  
 
{
We end this section with an index of different partition functions that will appear in the paper. 
\begin{longtable}{ l l }
$\mathcal{L}_{\bf a}$ & the anti-diagonal line $\{\mathbf{a} + (j, -j) : j \in \mathbb{Z}\}$\\ [2mm]
$\mathcal{L}^k_{{\bf a}}$ & the anti-diagonal segment $\{{\bf x}\in \mathcal{L}_{{\bf a}} : |{\bf x}-{\bf a}|_\infty \leq k \}$ \\ [2mm]
$R_{{\bf a}, {\bf b}}^k$ & the parallelogram spanned by ${\bf a} \pm (-k,k)$ and ${\bf b} \pm (-k,k)$\\ [2mm]
$Z_{\mathbf{u}, \mathbf{v}}$ & the point-to-point partition function \\[2mm]
$Z_{\mathbf{u}, \mathcal{L}_\mathbf{v}}$ & the point-to-line partition function \\[2mm]
$Z_{\mathcal{L}_\mathbf{u}^c, \mathcal{L}_\mathbf{v}^d}$ & the segment-to-segment partition function \\[2mm]
$Z_{A, B}$ & the partition function from summing over all paths between $A$ and $B$, for $A, B \subset \mathbb{Z}^2$\\[2mm]
$Z^\textup{max}_{A, B}$ & the maximum $\max_{{\bf a}\in A, {\bf b}\in B} Z_{{\bf a}, {\bf b}}.$\\[2mm]
$Z_{A, B}^{\textup{in}, R^{h}_{{\bf c}, {\bf d}}}$ & the partition function with paths from $A$ to $B$ contained inside $R^{h}_{{\bf c}, {\bf d}}$\\ [2mm]
$Z_{A, B}^{\textup{exit}, R^{h}_{{\bf c}, {\bf d}}}$ & the partition function with paths from $A$ to $B$  that exit diagonal sides of  $R^{h}_{{\bf c}, {\bf d}}$\\ [2mm]
$Z_{\mathcal{L}_{{\bf a}}^{s_1}, \mathcal{L}_{{\bf b}}^{s_2}}^{\textup{in}, k}$ &  the abbreviation for $Z_{\mathcal{L}_{{\bf a}}^{s_1}, \mathcal{L}_{{\bf b}}^{s_2}}^{\textup{in}, R^{k}_{{\bf a}, {\bf b}}}$ \\[2mm]
$Z_{\mathcal{L}_{{\bf a}}^{s_1}, \mathcal{L}_{{\bf b}}^{s_2}}^{\textup{exit}, k}$ &  the abbreviation for $Z_{\mathcal{L}_{{\bf a}}^{s_1}, \mathcal{L}_{{\bf b}}^{s_2}}^{\textup{exit}, R^{k}_{{\bf a}, {\bf b}}}$ \\[2mm]
$Z_{m,n}$ & the abbreviation for $Z_{(m,m),(n,n)}$ \\[2mm]
$\wt Z_{\mathbf{u}, \mathbf{v}}$ & the point-to-point partition function, including the weight at $\mathbf{u}$ \\[2mm]
$Z^\rho_{\mathbf{u}, \mathbf{v}}$ & the partition function for the ratio-stationary polymer \\[2mm]
\end{longtable}
}

\noindent\textbf{Acknowledgements.}  
We would like to express our gratitude to the two anonymous referees for their careful reading and valuable suggestions, and to Pranay Agarwal for pointing out that the statement of Theorem \ref{rwrw} in the previous version was inadequate for later uses (a stronger version of this theorem is now appropriately stated and proved). RB was partially supported by a MATRICS grant (MTR/2021/000093) from SERB, Govt.~of India, DAE project no.~RTI4001 via ICTS, and the Infosys Foundation via the Infosys-Chandrasekharan Virtual Centre for Random Geometry of TIFR.  TS was partially supported by National Science Foundation grants  DMS-1854619 and DMS-2152362, by Simons Foundation grant 1019133, and by the Wisconsin Alumni Research Foundation. Part of this work was conducted at the International Centre for Theoretical Sciences (ICTS), Bengaluru, India during the program "First-passage percolation and related models" in July 2022 (code:~ICTS/fpp-2022/7), and the authors thank ICTS for their hospitality.


\section{Main results} \label{main_results}

Let $\{Y_{\bf z}\}_{{\bf z}\in \mathbb{Z}^2}$ be a collection of positive weights on the integer lattice $\mathbb{Z}^2$. Fix two points ${\bf u}, {\bf v} \in \mathbb{Z}^2$ and denote the collection of up-right paths between them by $\mathbb{X}_{{\bf u}, {\bf v}}$. An element  $\gamma \in \mathbb{X}_{{\bf u}, {\bf v}}$  is viewed as a sequence of vertices $\gamma = (\gamma_0, \gamma_1, \dots, \gamma_{|{\bf u}- {\bf v}|_1})$ such that $\gamma_0 = {\bf u}$, $\gamma_{|{\bf u}- {\bf v}|_1} ={\bf v}$ and  $\gamma_{i+1}-\gamma_i \in \{{\bf e}_1, {\bf e}_2\}$.
The \textit{point-to-point polymer partition function} between ${\bf u}$ and  ${\bf v}$ is defined by 
\begin{equation}\label{par_fun}
Z_{{\bf u},{\bf v}} = \sum_{\gamma \in \mathbb{X}_{{\bf u}, {\bf v}}} \prod_{i=1}^{|{\bf u} - {\bf v}|_1} Y_{\gamma_i},
\end{equation}
provided that ${\bf u} \neq {\bf v}$ and $\mathbb{X}_{{\bf u}, {\bf v}}$ is non-empty. Otherwise, we set $Z_{{\bf u},{\bf v}} = 0$. Note the convention here that the weight $Y_{\bf u}$ at the beginning of the path does not enter into the definition of the partition function, since the product starts with $i=1$. 

The \textit{free energy} is defined to be 
$\log Z_{{\bf u}, {\bf v}}$ 
and takes the value $-\infty$ if $Z_{{\bf u}, {\bf v}} = 0$.
Provided that $Z_{{\bf u},{\bf v}} >0$, the \textit{quenched polymer measure} is a probability measure on the set of paths $\mathbb{X}_{{\bf u}, {\bf v}}$   defined by 
$$Q_{{\bf u}, {\bf v}}\{\gamma\} = \frac{1}{Z_{{\bf u}, {\bf v}}} \prod_{i=1}^{|{\bf u}-{\bf v}|_1} Y_{\gamma_i} \qquad \textup{ for $\gamma \in \mathbb{X}_{{\bf u}, {\bf v}}$}.$$

In general, the positive weights $\{Y_{\bf z}\}_{{\bf z}\in \mathbb{Z}^2}$ can be chosen as a collection of i.i.d.~positive random variables on some probability space $(\Omega, \mathbb{P})$. 
Under a mild moment assumption such as
$$\mathbb{E}\big[|\log Y_{\bf z}|^p\big] < \infty \quad  \textup{ for some $p>2$},$$
a law of large numbers type result called the \textit{shape theorem}  holds for the free energy   (Section 2.3 of \cite{Jan-Ras-20-aop}):  
there exists a concave, positively homogeneous and deterministic continuous function $\Lambda: \mathbb{R}^2_{\geq 0} \rightarrow \mathbb{R}$ that satisfies 
$$\lim_{n\rightarrow \infty} \sup_{{\bf z}\in \Z^2_{\geq 0} : |{\bf z}|_1 \geq n} \frac{|\log Z_{(0,0), {\bf z}} - \Lambda({\bf z})|}{|{\bf z}|_1} = 0 \qquad \textup{$\mathbb{P}$-almost surely}.$$

For general i.i.d.~weights, regularity properties of $\Lambda$ such as strict concavity or differentiability, expected to hold at least for continuous  weights,  are unknown.
There is a special case, first observed in \cite{poly2}, that if the i.i.d.\ weights  have the inverse-gamma distribution, then $\Lambda$ can be computed explicitly. 
The density function of the inverse-gamma distribution is  
\be\label{ig4}  f_\mu(x) = {\displaystyle {\frac {1}{\Gamma (\mu )}}x^{-\mu -1}e^{-{\frac {1}{x}}}}
\quad \text{for } x>0. \ee 
The shape parameter $\mu\in(0,\infty)$  plays the role of temperature in this polymer model. 
We derive several properties for $\Lambda$ in Section \ref{sec:reg} which will be used in our proofs later on. 
In addition, for this inverse-gamma polymer, many more explicit estimates can be established, hence it is often referred to as an exactly solvable model.

As is standard, the correlation coefficient  of two random variables $\zeta$ and $\eta$ is defined by 
\[  \textup{$\mathbb{C}$orr}(\zeta, \eta)=
\frac{\Cov(\zeta, \eta)}{\Var(\zeta)^{1/2}\,\Var(\eta)^{1/2}}
=\frac{\mathbb E[\zeta \eta] - \mathbb E\zeta  \cdot  \mathbb E\eta}
{\mathbb E[ \,|\zeta-\mathbb E\zeta|^2\,]^{1/2}\,\mathbb E[ \,|\eta-\mathbb E\eta|^2\,]^{1/2}}.
\]
Our main result establishes the time correlation exponents $1/3$ and $2/3$ for two free energies based on the separation of their endpoints. 

The bounds in the next two theorems are valid under the assumption that the weights $\{Y_{\bf z}\}$ have the i.i.d.\  inverse-gamma distribution \eqref{ig4} for some choice of the parameter $\mu\in(0,\infty)$.

\begin{theorem}\label{thm_r_large}
There exist positive constants $\newc\label{thm_r_large_c1}, \newc\label{thm_r_large_c2}, c_0, N_0$ such that, whenever $N\geq N_0$ and $N/2 \le r \leq N-c_0$,
we have 
$$ 1-\oldc{thm_r_large_c1}\Big(\frac{N-r}{N}\Big)^{2/3} \leq \textup{$\mathbb{C}$orr}\big(\log Z_{(0,0), (r,r)}, \log Z_{(0,0), (N,N)}\big) \leq 1-\oldc{thm_r_large_c2}\Big(\frac{N-r}{N}\Big)^{2/3}.$$ 
\end{theorem}

\begin{theorem}\label{thm_r_small}  
There exist positive constants $\newc\label{thm_r_small_c1}, \newc\label{thm_r_small_c2}, c_0,  N_0$ such that, whenever $N\geq N_0$ and $c_0 \leq r \leq N/2$,  we have 
$$ \oldc{thm_r_small_c1}\Big(\frac{r}{N}\Big)^{1/3} \leq \textup{$\mathbb{C}$orr}\big(\log Z_{(0,0), (r,r)}, \log Z_{(0,0), (N,N)}\big) \leq \oldc{thm_r_small_c2}\Big(\frac{r}{N}\Big)^{1/3}.$$ 
\end{theorem}

{
We record the following two corollaries, which state the equivalent results but in terms of the covariance of the free energies.}

{
\begin{corollary}
There exist positive constants $\newc\label{cor_r_large_c1}, \newc\label{cor_r_large_c2}, c_0, N_0$ such that, whenever $N\geq N_0$ and $N/2 \le r \leq N-c_0$,
we have 
\begin{align*}\oldc{cor_r_large_c1}({N-r})^{2/3} \leq &\sqrt{\Var\big(\log Z_{(0,0), (r,r)}\big)} \sqrt{\Var\big( \log Z_{(0,0), (N,N)}\big)} \\
&\qquad \qquad \quad - \textup{$\mathbb{C}$ov}\big(\log Z_{(0,0), (r,r)}, \log Z_{(0,0), (N,N)}\big) \leq \oldc{cor_r_large_c2}({N-r})^{2/3}.
\end{align*}
\end{corollary}
\begin{corollary}
There exist positive constants $\newc\label{cor_r_small_c1}, \newc\label{cor_r_small_c2}, c_0,  N_0$ such that, whenever $N\geq N_0$ and $c_0 \leq r \leq N/2$,  we have 
$$ \oldc{cor_r_small_c1}r^{2/3} \leq \textup{$\mathbb{C}$ov}\big(\log Z_{(0,0), (r,r)}, \log Z_{(0,0), (N,N)}\big) \leq \oldc{cor_r_small_c2}r^{2/3}.$$ 
\end{corollary}
}

\section{Preliminaries on the polymer model} \label{s:aux1}

\subsection{Notation} \label{notation}

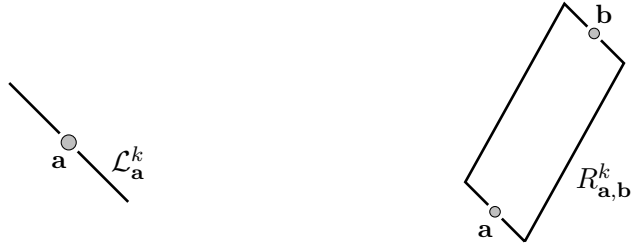
\begin{figure}[t]
\captionsetup{width=0.9\textwidth}
\begin{center}
\begin{tikzpicture}[x=0.75pt,y=0.75pt,yscale=1,xscale=1]
 
\draw [line width=1] (0,0+20)-- (-60,80);
\draw [line width=1] (200,0)--(200-30,0+30)--(200-30+50,0+30+90)--(200+50,0+90)--(200,0) ;

\fill[color=white] (200-15,15)circle(1.3mm);
\draw[ fill=lightgray](200-15,15)circle(0.7mm);

\node at (200-15-5,5) {\small ${\bf a}$};

\fill[color=white] (200-15+50,15+90)circle(1.3mm);
\draw[ fill=lightgray](200-15+50,15+90)circle(0.7mm);

\node at (200+40,25+90) {\small ${\bf b}$};

\node at (200+40,30) { ${R}_{{\bf a}, {\bf b}}^k$};

\fill[color=white] (-30,50)circle(1.7mm);
\draw[ fill=lightgray](-30,50)circle(1mm);

\node at (-35,40) {\small ${\bf a}$};

\node at (0,40) { $\mathcal{L}_{\bf a}^k$};

\end{tikzpicture}
\end{center}
\caption{\small Illustrations of the segment $\mathcal{L}_{\bf a}^k$ and the parallelogram $R_{{\bf a}, {\bf b} }^k$. All anti-diagonal segments have $\ell^\infty$-length $2k$.}
\label{LandR}
\end{figure}

Generic positive constants are denoted by $C, C'$ in the   proofs. They may change from line to line.  Other important positive constants   in the results are numbered in the form $C_{\textup{number}}$. 

As shown in Figure \ref{LandR}, for any  point ${\bf a}\in \mathbb{Z}^2$,  $\mathcal{L}_{{\bf a}}=\{{\bf a}+(j,-j): j\in\Z\}$ denotes the anti-diagonal line with slope $-1$ going through the point ${\bf a}$, and for any positive constant $k$, set
$$\mathcal{L}^k_{{\bf a}} = \{{\bf x}\in \mathcal{L}_{{\bf a}} : |{\bf x}-{\bf a}|_\infty \leq k  \}.$$
For ${\bf a}, {\bf b}\in \mathbb{Z}^{2}$ and  $k\in \mathbb{R}_{\geq 0}$,  
$R_{{\bf a}, {\bf b}}^k$ denotes the parallelogram spanned by the four corners ${\bf a} \pm (-k,k)$ and ${\bf b} \pm (-k,k)$.

For a collection of directed paths $\mathfrak{A}$, let   $Z(\mathfrak{A})$   be the partition function obtained by summing over all the paths in $\mathfrak{A}$
$$Z(\mathfrak{A})= \sum_{\gamma \in \mathfrak{A}}\, \prod_{\mathbf{z}\in \gamma} Y_{\mathbf{z}}.$$
For $A, B \subset \mathbb{R}^2$, let 
$Z_{A, B}$ denote the partition function obtained by summing over all directed paths starting from integer points 
$$A^\circ = \{{\bf a}\in \mathbb{Z}^2 : {\bf a}+[0,1)^2 \cap A \neq \emptyset\}$$ and ending in 
$$B^\circ = \{{\bf b}\in \mathbb{Z}^2 : {\bf b}+[0,1)^2 \cap B \neq \emptyset\}.$$ 
Furthermore, set 
$$Z^\textup{max}_{A, B} = \max_{{\bf a}\in A^\circ, {\bf b}\in B^\circ} Z_{{\bf a}, {\bf b}}.$$

For $A, B \subset \mathbb{R}^2$, ${\bf c},{\bf d} \in \mathbb{Z}^2$ and $h> 0$ we define two specific partition functions: 
\begin{align*}
   Z_{A, B}^{\textup{in}, R^{h}_{{\bf c}, {\bf d}}} &= \text{sum over directed paths from $A$ to $B$  contained inside the parallelogram  $R^{h}_{{\bf c}, {\bf d}}$, } \\
   Z_{A, B}^{\textup{exit}, R^{h}_{{\bf c}, {\bf d}}}&= \text{sum over directed paths from $A$ to $B$  that exit at least one of } \\
   &\quad\; \text{the sides of  $R^{h}_{{\bf c}, {\bf d}}$ parallel to ${\bf d}-{\bf c}$. }
\end{align*}
We simplify the notation when the starting and end places of the free energy match with the parallelogram, for example
$$Z_{\mathcal{L}_{{\bf a}}^{s_1}, \mathcal{L}_{{\bf b}}^{s_2}}^{\textup{in}, R^{k}_{{\bf a}, {\bf b}}} = Z_{\mathcal{L}_{{\bf a}}^{s_1}, \mathcal{L}_{{\bf b}}^{s_2}}^{\textup{in}, k} \qquad \textup{ and } \qquad Z_{\mathcal{L}_{{\bf a}}^{s_1}, \mathcal{L}_{{\bf b}}^{s_2}}^{\textup{exit}, R^{k}_{{\bf a}, {\bf b}}} = Z_{\mathcal{L}_{{\bf a}}^{s_1}, \mathcal{L}_{{\bf b}}^{s_2}}^{\textup{exit}, k} .$$

Integer points on the diagonal are abbreviated as $a=(a,a)\in\Z^2$. Common occurrences of this include  $Z_{r,N}=Z_{(r,r), (N,N)}$,  $Z_{{\bf p}, N}=Z_{{\bf p}, (N,N)}$,  $\mathcal{L}_a^k=\mathcal{L}_{(a,a)}^k$ and $R_{a, b}^k = R^k(a, b)=R_{(a,a), (b,b)}^k$.

The standard gamma function is $\Gamma(s)=\int_0^\infty x^{s-1} e^{-x}\,dx$ and the polygamma functions are $\Psi_k(s) = \frac{d^{k+1}}{ds^{k+1}}\log\Gamma(s)$ for $k=0,1,2,\dotsc$.  

Finally, we point out  two conventions. First, we drop the integer floor function   to simplify  notation. For example, if we divide the line segment from $(0,0)$ to $(N, N)$ in $5$  equal pieces, we denote the free energy of the first segment by  $\log Z_{0, N/5}$ even if  $N/5$ is not an integer. The second one is about the dependence of constants on parameters. A  statement of the type 
``there exists a positive $\theta_0$ such that for each $0< \theta < \theta_0$, there exist positive constants $C_0, N_0, t_0$ such that ..." means that $C_0, N_0$ and $t_0$ can (and often necessarily do) depend on $\theta$. 

\subsection{Regularity of the shape function and the characteristic direction}\label{sec:reg}

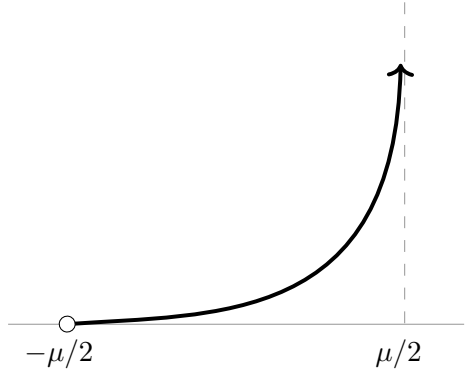
\begin{figure}[t]
\captionsetup{width=0.8\textwidth}
\begin{center}
\begin{tikzpicture}[x=0.75pt,y=0.75pt,yscale=-1,xscale=1]

\draw [color={rgb, 255:red, 155; green, 155; blue, 155 }  ,draw opacity=1 ]   (430,207) -- (200,207) ;
\draw [line width=1.5, ->]    (229,207) .. controls (305,201.5) and (395,208.5) .. (398,75.5) ;
\draw  [fill={rgb, 255:red, 255; green, 255; blue, 255 }  ,fill opacity=1 ] (225.94,205.86) .. controls (226.53,203.73) and (228.73,202.47) .. (230.86,203.06) .. controls (232.99,203.64) and (234.24,205.84) .. (233.66,207.97) .. controls (233.08,210.1) and (230.87,211.36) .. (228.74,210.77) .. controls (226.61,210.19) and (225.36,207.99) .. (225.94,205.86) -- cycle ;
\draw [color={rgb, 255:red, 155; green, 155; blue, 155 }  ,draw opacity=1 ] [dash pattern={on 4.5pt off 4.5pt}]  (400,206.5) -- (400,41.5) ;

\draw (207,213.4) node [anchor=north west][inner sep=0.75pt]    {$-\mu /2$};
\draw (384,213.4) node [anchor=north west][inner sep=0.75pt]    {$\mu /2$};
\end{tikzpicture}
\end{center}
\caption{\small The graph of $m_{\mu/2}(z)$ for $z\in (-\mu/2, \mu/2)$. The function $m$ is smooth and has a non-vanishing derivative on $(-\mu/2, \mu/2)$. The image of $m_{\mu/2}(z)$ is $(0, \infty)$ which corresponds the slopes of the points inside $]{\bf e}_1, {\bf e}_2[$.}
\label{m(z)}
\end{figure}

Henceforth fix the shape parameter $\mu\in(0,\infty)$ and assume that the weights $\{Y_{\bf z}\}$ have the i.i.d.\ inverse-gamma distribution \eqref{ig4}. 
Recall from Section \ref{notation} that $\Psi_1$ is the trigamma function, define the \textit{characteristic direction} as a function of $\rho \in (0, \mu)$
\begin{equation}\label{char_dir}
{\boldsymbol\xi}[\rho] = \Big(\frac{\Psi_1(\rho)}{\Psi_1(\rho) + \Psi_1(\mu -\rho)}\,,\, \frac{\Psi_1(\mu -\rho)}{\Psi_1(\rho) + \Psi_1(\mu -\rho)}\Big). 
\end{equation}
  The term characteristic direction becomes  meaningful  when we define the stationary inverse-gamma polymer in Section \ref{stat_poly}.
$\Psi_1$ is strictly decreasing and $C^\infty$ on $\mathbb{R}_{>0}$.  Thus  ${\boldsymbol\xi}[\rho]$ is a continuous bijection between $\rho \in (0, \mu)$ and vectors  (or directions) on the open line segment between ${\bf e}_1$ and  ${\bf e}_2$.   
Denote the slope of the vector ${\boldsymbol\xi}[\rho +z]$ by 
$$m_\rho( z) = \frac{{\boldsymbol\xi}[\rho + z] \cdot {\bf e}_2}{{\boldsymbol\xi}[\rho + z] \cdot {\bf e}_1} =\frac{\Psi_1(\mu - \rho-z)}{\Psi_1(\rho+z)}. $$
It  is $C^\infty$  with non-vanishing derivative on the interval $ z \in (-\rho, \mu-\rho)$. 
Its inverse function  $z_\rho(m)$  is $C^\infty$ and has a non-vanishing positive derivative for $m\in (0,\infty)$. 
The graph of $m_{\mu/2}(z)$ is illustrated in Figure \ref{m(z)}.
Taylor expansion for $m_\rho (z)$ around $z= 0$ gives this estimate:  
\begin{proposition}[Lemma 3.1 of \cite{Bus-Sep-22}]\label{slope1}
There exist positive constants $\newc\label{slope1_c1}, \newc\label{slope1_c2}, \epsilon$ such that for each $z\in [-\epsilon, \epsilon]$ and each $\rho \in [\epsilon, \mu-\epsilon]$, we have
$$
\big|m_\rho(z) - (m_\rho(0) + \oldc{slope1_c1} z)\big| \leq \oldc{slope1_c2} z^2.
$$
\end{proposition}

The next few results  specialize to the diagonal direction $\rho = \mu/2$. We drop the subscript and write $m = m_{\mu/2}$ and $z = z_{\mu/2}$.
Taylor expansion of $z(m)$ around $m = 1$ gives this estimate: 
\begin{proposition}\label{slope_z}
There exist positive constants $\newc\label{slope_z_c1}, \newc\label{slope_z_c2}, \epsilon$ such that for each $m\in [1-\epsilon, 1+\epsilon]$, we have
$$
\big|z(m) - \oldc{slope_z_c1}(m-1)\big| \leq\oldc{slope_z_c2} (m-1)^2.
$$
\end{proposition}

Next, we quantify the dependence of   the shape function on $\rho$.  {Recall the \textit{shape function} $\Lambda$ is a positively homogeneous, nonrandom continuous function $\Lambda: \mathbb{R}^2_{\geq 0} \rightarrow \mathbb{R}$ that satisfies the \textit{shape theorem} (see \cite[Section 2.3]{Jan-Ras-20-aop}):
\be\label{La8} \lim_{n\rightarrow \infty} \sup_{\mathbf z\in \Z^2_{\geq 0} : |\mathbf z|_1 \geq n} \frac{|\log Z_{0, \mathbf z} - \Lambda(\mathbf z)|}{|\mathbf z|_1} = 0 \qquad \text{$\mathbb{P}$-almost surely}.\ee}
Let $f(\rho)$ denote the shape function $\Lambda$ evaluated at the vector ${\boldsymbol\xi}[\rho]$, and recall from \cite{poly2} that  \begin{equation}\label{define_f}
f(\rho) = \Lambda({\boldsymbol\xi}[\rho]) = -\tfrac{\Psi_1({{\rho}})}{\Psi_1({{\rho}}) + \Psi_1({{\mu - \rho}})}\cdot  \Psi_0( \mu-\rho)  -  \tfrac{\Psi_1({{\mu-\rho}})}{\Psi_1({{\rho}}) + \Psi_1({{\mu - \rho}})}  \Psi_0(\rho)
\end{equation}
where $\Psi_0$ and $\Psi_1$ are the digamma and trigamma function.
Let $f_d=f(\mu/2)$ denote  the shape function in the  diagonal direction. From concavity and symmetry, we get this inequality: 
\begin{proposition}\label{time_const1}
For each $\mu>0$ and  each $z\in (-\mu/2, \mu/2)$, $f(\mu/2) \geq f(\mu/2 + z)$.
\end{proposition}

The next  bound   captures the curvature of the shape function. 
\begin{proposition}\label{time_const}
There exist positive constants $\newc\label{time_const_c1}, \newc\label{time_const_c2}, \epsilon$ 
such that for each $z\in [-\epsilon, \epsilon]$, we have
$$\Big|(f(\mu/2 + z) - f(\mu/2))-(-\oldc{time_const_c1} z^2)  \Big| \leq \oldc{time_const_c2} z^4.$$
\end{proposition}

\begin{proof}
From \eqref{define_f},   
\begin{align}
& f(\mu/2 + z) - f(\mu/2) \nonumber\\
&= \Big[- \Big(\tfrac{\Psi_1(\mu/2+z)}{\Psi_1(\mu/2+z) + \Psi_1(\mu/2-z)}\Psi_0(\mu/2- z) + \tfrac{\Psi_1(\mu/2-z)}{\Psi_1(\mu/2+z) + \Psi_1(\mu/2-z)}\Psi_0(\mu/2+ z)\Big)\Big]  \label{taylor_1}\\
&\qquad\qquad\qquad\qquad  -\Big[-\Big(\tfrac{\Psi_1(\mu/2)}{\Psi_1(\mu/2) + \Psi_1(\mu/2)}\Psi_0(\mu/2) + \tfrac{\Psi_1(\mu/2)}{\Psi_1(\mu/2) + \Psi_1(\mu/2)}\Psi_0(\mu/2)\Big)\Big]. \label{taylor_2}
\end{align}
Taylor expand \eqref{taylor_1} around $z = 0$. The ``zeroth" derivative terms and \eqref{taylor_2} cancel each other. The coefficients of   $z$, $z^3$, $z^5$ are zero. The coefficient of $z^2$ is $\frac{1}{2}\Psi_2(\mu/2)<0$. 
\end{proof}

Our next proposition controls the variation of the shape function on a segment $\mathcal{L}_{N}^{hN^{2/3}}$.

\begin{proposition}\label{reg_shape}
There exist positive constants $\newc\label{reg_shape_c1}, N_0, \epsilon_0$ such that for each $N\geq N_0$ , $h\leq \epsilon_0N^{1/3}$ and each ${\bf p}\in \mathcal{L}_{N}^{h N^{2/3}}$, we have 
$$\big|\Lambda({\bf p}) - 2N{f_d}\big| \leq \oldc{reg_shape_c1} h^2 N^{1/3}.$$
\end{proposition}
\begin{proof}
Since each ${\bf p} \in \mathcal{L}_{N}^{hN^{2/3}}$ has the same $\ell^1$-norm $2N$, let us rewrite ${\bf p} = 2N{\boldsymbol\xi}[\mu/2 + z_{\bf p}]$ for some real number $z_{\bf p}$. 
Then, by our definition $\Lambda({\bf p}) = 2Nf(\mu/2 + z_{\bf p}).$

Since the perpendicular $\ell^\infty$-distance from ${\bf p}$ to the diagonal is at most $hN^{2/3}$,   the slope of the characteristic vector ${\boldsymbol\xi}[\mu/2 + z_{\bf p}]$ satisfies 
$$|m(z_{\bf p})-1|\leq 2hN^{-1/3}.$$ 
Fix $\epsilon_0$ sufficiently small, by Proposition \ref{slope_z}, we obtain
$$|z_{\bf p}| \leq C h N^{-1/3}.$$ 
Finally, applying Proposition \ref{time_const}, we obtain that 
$$|f(\mu/2 + z_{\bf p}) - {f_d}| \leq Ch^2 N^{-2/3}$$
which directly implies the result of our proposition after multiplying by $2N$ on both sides.
\end{proof}

\subsection{Free energy estimates}\label{free_est}
In this section, we collect a number of estimates used later in the proofs, organized thematically into subsections.  Some results are merely quoted, some proved later,  and in cases  where the result has already  appeared in the zero temperature setting  the positive temperature proofs are in Appendix 
\ref{appen}.

\subsubsection{Moderate deviation estimates for the free energy} \label{mod_est}

  There are four moderate deviation estimates: upper and lower bounds for both left and right tails.

The first theorem gives the upper bound on the right tail of the free energy. This result for the inverse-gamma polymer was first proved as a combination of the moderate deviation estimate from \cite{TW_polymer}, which used integrable techniques, and the large deviation estimate from \cite{Geo-Sep-13}. 
The same moderate deviation upper bound was also recently proven in \cite{OC_tail} for the O'Connell-Yor polymer using the coupling method, which was based on the seminal work \cite{rtail1} in the zero temperature setting. 
With a similar coupling approach, the forthcoming work \cite{Emr-Jan-Xie-22-} proves this bound and obtains the sharp leading order term $\tfrac{4}{3}t^{3/2}$ in the exponent for $t\leq C N^{2/3}$. A version of this bound can be found in the Ph.D. thesis of one of the authors of \cite{Emr-Jan-Xie-22-}, as Theorem 4.3.1 in \cite{Xie2022}. The right tail estimate for the KPZ equation with the sharp leading order term was also recently obtained in \cite{KPZ_uptail}.

\begin{proposition}\label{up_ub}
Let $\epsilon \in (0, \mu/2)$. There exist positive constants $\newc\label{up_ub_c1}, N_0$ depending on $\epsilon$ such that for each $N\geq N_0$, $t \geq 1$, and each $\rho \in [\epsilon, \mu-\epsilon]$, we have
$$\mathbb{P}(\log Z_{0,  2N{\boldsymbol\xi}[\rho]} - 2Nf(\rho) \geq tN^{1/3}) \leq e^{-\oldc{up_ub_c1} \min\{t^{3/2}\!,\, tN^{1/3}\}}.$$
\end{proposition}

The next theorem is the corresponding lower bound for the right tail, restricted to the diagonal direction. 
This was recently proved for the O'Connell-Yor polymer in \cite{OC_tail} in the diagonal direction. The proof uses the subadditivity of the free energy and the Tracy-Widom limit of the free energy. Since using integrable techniques, the  Tracy-Widom limit of the inverse-gamma polymer is also known \cite{TW_polymer}, the proof for the O'Connell-Yor polymer in Section 9 of \cite{OC_tail} can be repeated verbatim for the inverse-gamma polymer. 
A similar argument in the zero-temperature setting appeared earlier in \cite{timecorriid, bootstrap}.
Without this input from integrable probability, a lower bound with the correct leading order $\tfrac{4}{3} t^{3/2}$ for $t\leq CN^{2/3}$ over all directions in a compact interval away from ${\bf e}_1$ and ${\bf e}_2$  will appear in \cite{Emr-Jan-Xie-22-}. 
\begin{proposition}\label{up_lb}
There exist positive constants $\newc\label{up_lb_c1}, N_0, t_0, \epsilon_0$ such that for each $N\geq N_0$, $t_0 \leq t \leq \epsilon_0 N^{2/3}$, we have
$$\mathbb{P}(\log Z_{0, N} - 2N{f_d} \geq tN^{1/3}) \geq e^{-\oldc{up_lb_c1} t^{3/2}}.$$
\end{proposition}

The next theorem is the upper bound for the left tail. A similar result was stated as Proposition 3.4 in \cite{OC_tail} for the O'Connell-Yor polymer. We prove this estimate for the inverse-gamma polymer in Section \ref{sec_tail}. Our proof is similar to \cite{OC_tail}, based on ideas from the zero-temperature work  \cite{cgm_low_up}. 

\begin{proposition}\label{low_ub}
Let $\epsilon \in (0, \mu/2)$. There exist positive constants $\newc\label{low_ub_c1}, N_0 $ depending on $\epsilon$ such that for each $N\geq N_0$, $t\geq 1$ and each $\rho\in[\epsilon, \mu-\epsilon]$, we have 
$$\mathbb{P}(\log Z_{0,  2N{\boldsymbol\xi}[\rho]} - 2Nf(\rho) \leq -tN^{1/3}) \leq e^{-\oldc{up_ub_c1} \min\{t^{3/2}, tN^{1/3}\}}.$$
\end{proposition}
\begin{remark}
The correct order of the left tail should be $e^{-C\min\{t^3, tN^{4/3}\}}$ for all $t\geq t_0$. This is different from the zero-temperature model where the left tail behaves as $e^{-Ct^3}$. For the O'Connell-Yor polymer, the authors in \cite{OC_tail} also proved an upper bound $e^{-Ct^3}$ when $t_0 \leq t \leq N^{2/3}(\log N)^{-1}$. This is done by adapting the bootstrapping argument from the zero-temperature work  \cite{bootstrap}. We do not pursue
 this here but expect the same result.
\end{remark}

Finally, we have the lower bound on the left tail, which we prove in Section \ref{sec_tail}. The same lower bound was proved in \cite{OC_tail} for the O'Connell-Yor polymer. The idea of the proof follows the zero-temperature work  \cite{bootstrap}.  

\begin{proposition}\label{ptp_low}
There exist positive constants $\newc\label{ptp_low_c1}, N_0, t_0, \epsilon_0$ such that for each $N \geq N_0$ and each $t_0 \leq t \leq \epsilon_0 N^{2/3}/(\log N)^2$, we have 
$$\mathbb{P}(\log Z_{0, N}  - 2N{f_d} \leq -t N^{1/3}) \geq e^{-\oldc{ptp_low_c1}t^3}.$$
\end{proposition}

\subsubsection{Free energy and path fluctuations}
\label{sec:Zpath} 
In this section we  state the estimates which capture the loss of free energy when the paths have high fluctuations.

\begin{proposition}\label{trans_fluc_loss}
There exist positive constants $\newc\label{trans_fluc_loss_c1}, \newc\label{trans_fluc_loss_c2}, N_0$ such that for each $N\geq N_0$, $h\in \mathbb{Z}$ and $t\geq 0$ we have
$$\mathbb{P}\Big(\log Z_{\mathcal{L}_0^{N^{2/3}} , \mathcal{L}_{(N-2hN^{2/3}, N+ 2hN^{2/3})}^{N^{2/3}}}  - 2N{f_d} \geq (-\oldc{trans_fluc_loss_c1}h^2+t)N^{1/3} \Big) \leq e^{-\oldc{trans_fluc_loss_c2}(|h|^3 + \min\{t^{3/2}, tN^{1/3}\})}.$$
\end{proposition}

Then, by essentially a union bound, we obtain the following proposition.

\begin{proposition}\label{trans_fluc_loss2}
There exist positive constants $\newc\label{trans_fluc_loss2_c1}, \newc\label{trans_fluc_loss2_c2},  N_0$ such that for each $N\geq N_0$, $t\geq 1$ and $s\geq 0$, we have
$$\mathbb{P}\Big(\log Z_{\mathcal{L}^{sN^{2/3}}_0, \mathcal{L}_{N}\setminus\mathcal{L}_{N}^{(s+t)N^{2/3}}}   - 2N{f_d} \geq -\oldc{trans_fluc_loss2_c1}t^2N^{1/3}\Big) \leq e^{-\oldc{trans_fluc_loss2_c2}t^3}.$$
\end{proposition}
Following this, we have the next proposition which states that paths with high fluctuation tend to have much small free energy.
\begin{theorem}\label{trans_fluc_loss3}
There exist positive constants $\newc\label{trans_fluc_loss3_c1}, \newc\label{trans_fluc_loss3_c2},   N_0$ such that for each $N\geq N_0$, $1 \leq t \leq  N^{1/3}$ and $0< s < e^{t}$,  we have
$$\mathbb{P}\Big(\log Z^{\textup{exit}, (s+t)N^{2/3}}_{\mathcal{L}_{0}^{sN^{2/3}}, \mathcal{L}_N^{sN^{2/3}}} - 2N{f_d} \geq -\oldc{trans_fluc_loss3_c1}t^2N^{1/3}\Big) \leq e^{-\oldc{trans_fluc_loss3_c2}t^{3}}.$$
\end{theorem}

From this, we have the following corollary which is a similar bound for point-to-point free energy that is slightly off the diagonal direction.
\begin{corollary}\label{trans_fluc_loss4}
There exist positive constants $\newc\label{trans_fluc_loss4_c1}, \newc\label{trans_fluc_loss4_c2},  N_0$ such that for each $N\geq N_0$, $1 \leq t \leq  N^{1/3}$ and $0< s < t/10$,  we have
$$\mathbb{P}\Big(\log Z^{\textup{exit}, {tN^{2/3}}}_{(-sN^{2/3}, sN^{2/3}), N}  - 2N{f_d} \geq -\oldc{trans_fluc_loss4_c1}t^2N^{1/3}\Big) \leq e^{-\oldc{trans_fluc_loss4_c2}t^{3}}.$$
\end{corollary}

\subsubsection{Interval-to-line free energy}
In our work, we will also need an upper bound for the right tail of the interval-to-line free energy. 
\begin{theorem}\label{ptl_upper}
There exist positive constants $\newc\label{ptl_upper_c1}, \newc\label{ptl_upper_c2}, N_0$ such that for each $N \geq N_0$, $t\geq 1$ and $1\leq h \leq e^{\oldc{ptl_upper_c1}{\min\{t^{3/2}, tN^{1/3}\}}}$, we have 
$$\mathbb{P}\Big(\log Z_{\mathcal{L}_0^{hN^{2/3}},\mathcal{L}_{N}}  - 2N{f_d} \geq tN^{1/3}\Big) \leq 
e^{-\oldc{ptl_upper_c2}\min\{t^{3/2}, tN^{1/3}\}}.  
$$
\end{theorem}

\subsubsection{Estimates for the constrained free energy} \label{con_bound}
When we constrain the paths, the free energy decreases because we are summing over a smaller collection of paths in the partition function.
The first theorem captures that the point-to-point free energy can not be too small if we constrain the paths to a fixed  rectangle of size order $N\times N^{2/3}$ which obeys the KPZ transversal fluctuation scale. 
Our second theorem gives a lower bound for the probability that a constrained free energy is large.

\begin{theorem}\label{wide_similar}
For each positive $a_0$, there exist positive constants $\newc\label{wide_similar_c1}, t_0$ such that for each $0< \theta \leq  100$, there exists a positive constant $ N_0$ such that for each $N\geq N_0$, $t\geq t_0$ and  ${\bf p} \in \mathcal{L}_N^{a_0\theta N^{2/3}}$, we have
$$\mathbb{P}\Big(\log Z^{\textup{in}, {\theta N^{2/3}}}_{0, {\bf p}} - 2N{f_d} \leq -tN^{1/3}\Big) \leq  \tfrac{\sqrt{t}}{\theta}e^{-\oldc{wide_similar_c1}\theta t}.$$
\end{theorem}

\begin{theorem}\label{c_up_lb}
For any positive constant $s$, there exist positive constants $\newc\label{c_up_lb_c1},t_0, N_0$ such that for each $N\geq N_0$, $t_0 \leq t \leq  N^{2/3}$, 
$$\mathbb{P}\Big(\log Z^{\textup{in}, {s N^{2/3}}}_{0, N}  - 2N{f_d} \geq tN^{1/3}\Big) \geq e^{-\oldc{c_up_lb_c1} t^{3/2}}.$$
\end{theorem}

\subsubsection{Minimum and maximum estimate for the free energy}
Our first theorem is the box-to-point minimum bound. This was first proved in the zero-temperature setting which appeared in \cite{slowbondproblem} for the Poissonian LPP model, then later in \cite{timecorrflat} for the exponential LPP model. The proof follows the idea from Section C.4 of \cite{timecorrflat}.

\begin{theorem}\label{high_inf}
There exist positive constants $\newc\label{high_inf_c1}, N_0, t_0 $ such that 
for each $N \geq N_0$ and $t\geq t_0$, we have 
$$\mathbb{P}\Big(\min_{{\bf p} \in R^{N^{2/3}}_{0, 9N/10}} \log Z^{\textup{in}, R_{0,N}^{N^{2/3}}}_{{\bf p}, N} - (2N - |{\bf p}|_1){f_d} \leq - tN^{1/3}\Big) \leq e^{-\oldc{high_inf_c1}t}.$$
\end{theorem}

Lastly, we state a box-to-line maximum bound. 

\begin{theorem}\label{btl_upper}
There exist positive constants $\newc\label{btl_upper_c1}, N_0 , t_0$ such that for each $N \geq N_0$ and each $t\geq t_0 $, we have 
$$\mathbb{P}\Big(\max_{{\bf p} \in R^{N^{2/3}}_{0, 9N/10}}\log Z_{{\bf p},\mathcal{L}_{N}}  - (2N - |{\bf p}|_1) {f_d} \geq tN^{1/3}\Big) \leq 
e^{-\oldc{btl_upper_c1}t}.  
$$
\end{theorem}
\begin{remark}
In both bounds of Theorem \ref{high_inf} and Theorem \ref{btl_upper}, the power $1$ on the exponent $t^1$ is not expected to be optimal. 
\end{remark}

\subsubsection{Variance bound for the free energy} \label{sec:Zvar}

We state the variance bound of the free energy which follows directly from the upper and lower bounds for the left and right tails. We omit its proof.
The upper bound was first shown in \cite{poly2} where the inverse-gamma polymer was first introduced.
\begin{theorem}\label{var}
There exist positive constants $\newc\label{var_c1}, \newc\label{var_c2}, N_0$ such that for each $N\geq N_0$, we have 
$$\oldc{var_c1}N^{2/3}\leq \Var(\log Z_{0,N})\leq \oldc{var_c2}N^{2/3}.$$
\end{theorem}

\subsubsection{Nonrandom fluctuation} \label{non_rand}
Finally, we record a lower bound for the nonrandom fluctuation of the free energy in i.i.d.~inverse-gamma polymer. This result follows directly from the Tracy-Widom limit of the inverse-gamma model and the fact that the Tracy-Widom distribution has a negative mean. Our contribution here is  an alternative proof  (in Section \ref{sec:nr}) without relying on the Tracy-Widom limit. 

\begin{theorem}\label{nr_lb}
Let $\epsilon \in(0,\mu/2)$. There exist positive constants $\newc\label{nr_lb_c1}, N_0$ such that for each $N\geq N_0$ and $\rho \in [\epsilon, \mu-\epsilon]$, we have 
$$  2Nf(\rho) -  \mathbb{E}[\log Z_{0, 2N{\boldsymbol\xi}[\rho]}] \geq  \oldc{nr_lb_c1} N^{1/3}.$$
\end{theorem}

\subsection{Stationary inverse-gamma polymer}\label{stat_poly}

The (increment) stationary inverse-gamma polymer (with southwest boundary) is defined on a quadrant. To start, we fix a parameter $\rho\in (0,\mu)$ and a base vertex ${\bf v}\in \mathbb{Z}^2$. For each ${\bf z}\in {\bf v} + \mathbb{Z}_{>0}^2$, the (vertex) bulk weights are defined by $Y_z \sim \text{Ga}^{-1}(\mu)$, {where $\text{Ga}^{-1}(\mu)$ denotes the inverse-gamma distribution with shape parameter $\mu$}. On the boundary ${\bf v} + k{\bf e}_1$, and ${\bf v} + k{\bf e}_2$, the (edge) weights {are denoted by $I$'s and $J$'s}, and they have the distributions
\begin{align}\label{stat_weights}
\begin{aligned}
I^{{\rho}}_{[\![v+(k-1)k{\bf e}_1, v+k{\bf e}_1]\!]} & \sim \text{Ga}^{-1}(\mu-{{\rho}})\\ J^{{\rho}}_{[\![v+(k-1)k{\bf e}_2, v+k{\bf e}_2]\!]} &\sim \text{Ga}^{-1}({{\rho}}).
\end{aligned}
\end{align}
All the weights in the quadrant are independent. 
We denote the probability measure for the stationary inverse-gamma polymer by $\mathbb{P}$ and record the parameter $\rho$ and the base point ${\bf v}$ in the notation of the partition function. For ${\bf w}\in {\bf v}+ \mathbb{Z}_{\geq 0}^2$, let us define  
$$Z^{\rho}_{{\bf v}, {\bf w}} = \sum_{\gamma\in \mathbb{X}_{{\bf v},{\bf w}}} \prod_{i=0}^{|{\bf v}-{\bf w}|_1} \wt{Y}_{\gamma_i} \qquad \text{ where }\wt{Y}_{\gamma_i} = \begin{cases}
1 \quad & \text{if $\gamma_i = {\bf v}$}\\
I^{{\rho}}_{[\![\gamma_i -{\bf e}_1, {\gamma_i }]\!]} \quad & \text{if $\gamma_i \cdot {\bf e}_2 = {\bf v}\cdot {\bf e}_2 $}\\
J^{{\rho}}_{[\![\gamma_i -{\bf e}_2, {\gamma_i }]\!]} \quad & \text{if $\gamma_i \cdot {\bf e}_1 = {\bf v}\cdot {\bf e}_1 $}\\
Y_z \quad & \text{otherwise.} \\
\end{cases}$$ 
And for $\gamma \in \mathbb{X}_{{\bf v},{\bf w}}$, the quenched polymer measure is defined by
$$Q^{{{\rho}}}_{{\bf v}, {\bf w}}(\gamma) = \frac{1}{Z^{{{\rho}}}_{{\bf v}, {\bf w}}}  \prod_{i=0}^{|{\bf v}-{\bf w}|_1} \wt{Y}_{\gamma_i}.$$

The name (increment) stationary inverse-gamma polymer is justified by the next theorem, which first appeared in \cite[Theorem 3.3]{poly2}.
\begin{theorem}\label{stat} 
For each ${\bf w}\in {\bf v} + \mathbb{Z}^2_{>0}$ . We have 
$$\frac{Z^\rho_{{\bf v}, {\bf w}}}{Z^\rho_{{\bf v}, {\bf w}- {\bf e}_1}} \sim \textup{Ga}^{-1}(\mu-\rho) \qquad \text{ and }\qquad \frac{Z^\rho_{{\bf v}, {\bf w}}}{Z^\rho_{{\bf v}, {\bf w}- {\bf e}_2 }}\sim \textup{Ga}^{-1}(\rho).$$
Furthermore, let $\eta = \{\eta_i\}$ be any finite or infinite down-right path in ${\bf v}+ \mathbb{Z}^2_{\geq 0}$. This means $\eta_{i+1} - \eta_i$ is either ${\bf e}_1$ or $-{\bf e}_2$. Then, the increments $\{Z^\rho_{{\bf v}, \eta_{i+1}}/Z^\rho_{{\bf v}, \eta_{i}}\}$ are independent.
\end{theorem}
From Theorem \ref{stat} above, we have the following identity for the expectation of the free energy. {Recall from Section \ref{notation} that $\Psi_0$ is the digamma function,}
\begin{equation}\label{expect_stat}
\mathbb{E}\Big[\log Z^\rho_{0, (a,b)} \Big] = - a\Phi_0(\mu-\rho) -b \Phi_0(\rho).
\end{equation}

Because the weights appearing on the boundary are stochastically larger than the bulk weights, the sampled polymer paths tend to stay on the boundary. However, for each fixed ${{\rho}} \in (0, \mu)$, there is a unique direction for which this effect between the ${\bf e}_1$- and ${\bf e}_2$-boundary is balanced out, we call this the characteristic direction ${\bf \xi}[\rho]$, which is defined previously in \eqref{char_dir}. 

The first estimate below is the upper bound for the right tail of the free energy. It first appeared in the Ph.D. Thesis \cite{Xie2022}. Then, it was proven again in \cite{Lan-Sos-22-a-}. {From \eqref{expect_stat} and the definitions of $f(\rho)$ in \eqref{define_f} and $\boldsymbol\xi[\rho]$ in \eqref{char_dir}, by a substitution, we see that $2Nf(\rho)$ can be thought as the expectation of $\log Z^\rho_{0, 2N\boldsymbol{\xi}[\rho]}$, if we ignore the error from the integer rounding.}
\begin{theorem} \label{stat_up_ub}
Let $\epsilon \in(0,\mu/2)$. There exist positive constants $\newc\label{stat_up_ub_c1}, N_0 $ such that for each $N\geq N_0$, $t \geq 1$ and $\rho\in [\epsilon, \mu-\epsilon]$, we have 
$$\mathbb{P}\Big(\log Z^\rho_{0, 2N{\boldsymbol\xi}[\rho]} - 2Nf(\rho) \geq tN^{1/3}\Big) \leq e^{-\oldc{stat_up_ub_c1} \min\{t^{3/2}, tN^{1/3}\}}.$$

\end{theorem}

Along the characteristic direction, the sampled paths tend to stay on the boundary for order $N^{2/3}$ number of steps. Our next result is a corollary of this fact, which appears as {Corollary 4.2 in \cite{ras-sep-she-}. Fix ${\bf w} \in {\bf v} + \mathbb{Z}^2_{\geq 0}$ and any $k\in \mathbb{R}_{>0}$. Let
$
\{\tau_{{\bf v}, {\bf w}} \geq k\}
$
denote the subset of $\mathbb{X}_{{\bf v}, {\bf w}}$ such that the first $\floor k$ steps of the path are all ${\bf e}_1$-steps. Similarly, $\{\tau_{{\bf v}, {\bf w}} \leq -k\}$ is the subset of $\mathbb{X}_{{\bf v}, {\bf w}}$ whose first $\floor{k}$ steps are all ${\bf e}_2$-steps. When $\tau_{{\bf v}, {\bf w}}$ appears inside a quenched polymer measure as below, we will simplify the notation $\tau_{{\bf v}, {\bf w}} = \tau$ as the starting and end point of the paths are clear.

\begin{theorem}\label{exit_time}
Let $\epsilon \in(0,\mu/2)$. There exist positive constants $\newc\label{exit_time_c1}, \newc\label{exit_time_c2}, N_0$ such that for for all ${\rho} \in [\epsilon, \mu-\epsilon]$, 
$N\geq N_0$ and $r\ge 1$, we have  
$$\mathbb{P}^{{\rho}}(Q_{0, 2N{\boldsymbol\xi}[\rho] + rN^{2/3}{\bf e}_1}\{\tau \leq -1\} \geq e^{-\oldc{exit_time_c1}r^3}) \leq e^{-\oldc{exit_time_c2}r^{3}}.$$
\end{theorem}

Let $\wt Z$ denote the version of the partition function that also includes the weight at the beginning of the path. 
The following is essentially a lower bound for the difference between the free energies of the stationary boundary model and i.i.d.\ bulk polymer. We included an additional boundary weight with the i.i.d.\ bulk free energy in the estimate below because this version will be used to prove Theorem \ref{nr_lb}. Its proof will appear in  Section \ref{sec:nr}.
\begin{theorem}\label{statiid_low}
Let $\epsilon\in (0, \mu/2)$. There exist positive constants $\newc\label{statiid_low_c1}, N_0 $ such that for each $N\geq N_0$, $0< \delta \leq 1/2$ and $\rho\in [\epsilon, \mu-\epsilon]$, we have 
$$\mathbb{P}\Big(\log Z^\rho_{-1, 2N{\boldsymbol\xi}[\rho]} - \Big( \log I^\rho_{[\![(-1,-1), (0,-1) ]\!]} + \log \wt Z_{0, 2N{\boldsymbol\xi}[\rho]} \Big) \leq \delta N^{1/3}\Big) \leq \oldc{statiid_low_c1} |\log (\delta \vee N^{-1/3})|  \cdot (\delta \vee N^{-1/3}) .$$
\end{theorem}

For completeness, we also record the following upper bound for the difference between the stationary and i.i.d.~free energy. This result follows directly from Theorem \ref{stat_up_ub} and Proposition \ref{low_ub} using a union bound, hence we omit its proof.
\begin{theorem}\label{statiid_up}
Let $\epsilon \in(0,\mu/2)$. There exist positive constants $\newc\label{statiid_up_c1}, N_0 $ such that for each $N\geq N_0$, $t \geq 1$ and $\rho\in [\epsilon, \mu-\epsilon]$, we have 
$$\mathbb{P}\Big(\log Z^\rho_{-1, 2N{\boldsymbol\xi}[\rho]} - \log Z_{0, {2N\boldsymbol\xi}[\rho]}  \geq tN^{1/3}\Big) \leq e^{-\oldc{statiid_up_c1} \min\{t^{3/2}, tN^{1/3}\}}.$$
\end{theorem}

\subsection{Random walk comparison for the free energy profile} \label{sec_rw}

The stationary polymer allows one to compare the free energy profile along a segment of a downright path to random walks. This technique has appeared previously in \cite{balzs2019nonexistence, cuberoot, Bus-Sep-22,ras-sep-she-, poly2, seppcoal} and many more places.

To start, fix $\rho \in (0, \mu)$ and define 
$${\bf v}_N = 2N {\boldsymbol \xi}[\rho].$$
Let $\Theta_k$ denote a  down-right path of $k$ (edge) steps that goes through the vertex ${\bf v}_N$.  Order the vertices of $\Theta_k$ as ${\bf z}_0, \dotsc, {\bf z}_k$, where ${\bf z}_0$ has the largest ${\bf e}_2$-coordinate value. Define the free energy profile to be the following collection of random variables
\begin{equation}\label{profile}
\log Z_{0, {\bf z}_{i}} - \log Z_{0, {\bf z}_{i-1}} \quad \text{ where $i = 1, \dots, k$}.
\end{equation}
The proof of the following theorem appears in Appendix \ref{rwrw_proof}.

\begin{theorem}\label{rwrw}
Fix $\epsilon \in (0, \mu/2)$.  There exist positive constants $\newc\label{rwrw_c1}, N_0, s_0, a_0, q_0$ such that for each $\rho \in [\epsilon, \mu-\epsilon]$, $N\geq N_0$, $s_0 \leq s \leq a_0 N^{1/3}$, $1\leq k \leq s N^{2/3}$ and each down-right path $$\Theta_k=\{{\bf z}_0, \dotsc, {\bf z}_k\} \ni \mathbf{v}_N,$$ there exist two collections of random variables $\{X_i\}$ and $\{Y_i\}$ such that the following holds. Set  
$$\lambda = \rho + q_0sN^{-1/3} \qquad \text{ and } \qquad \eta = \rho- q_0sN^{-1/3}.$$
The random variables  $\{X_i\}$ are mutually independent with marginal  distributions  
\begin{alignat*}{2}
X_i &\sim \log(\textup{Ga}^{-1}(\mu-\lambda ))  &&\quad \text{ if ${\bf z}_i - {\bf z}_{i-1} = {\bf e}_1 $} \\
-X_i &\sim \log(\textup{Ga}^{-1}(\lambda))  &&\quad\text{ if ${\bf z}_i - {\bf z}_{i-1} = -{\bf e}_2$ }. 
\end{alignat*}
The random variables  $\{Y_i\}$ are  mutually independent with marginal  distributions  
\begin{alignat*}{2}
Y_i &\sim \log(\textup{Ga}^{-1}(\mu-\eta ))   &&\quad\text{ if ${\bf z}_i - {\bf z}_{i-1} = {\bf e}_1 $} \\
-Y_i &\sim \log(\textup{Ga}^{-1}(\eta))   &&\quad\text{ if ${\bf z}_i - {\bf z}_{i-1} = -{\bf e}_2$ }. 
\end{alignat*}
Furthermore, the sums of $X_i$ and $Y_i$ bound the free energy profile with high probability.
$$\mathbb{P} \Big(\bigcap_{j=1}^k \Big\{\log \tfrac{9}{10}+ \sum_{i=1}^j Y_i \leq \log Z_{0, {\bf z}_{j}} - \log Z_{0, {\bf z}_{0}} \leq \log \tfrac{10}{9} + \sum_{i=1}^j X_i \Big\}\Big) \geq 1- e^{-\oldc{rwrw_c1}s^3}. $$
\end{theorem}

We also note that when $\Theta_k$ is vertical or horizontal, then $X_i$ and $Y_i$ can be coupled together with an explicit joint distribution that allows calculations, see \cite{balzs2019nonexistence, balzs2020local, Bus-Sep-22, ras-sep-she-, seppcoal}. However, we will not use this fact in this paper.

{
\subsection{Maximum bound for the free energy} \label{max_bd}
Finally, we document the following bound for the free energy in a maximized version, rendering the estimates closely resembling those in last-passage percolation. This specific argument persists throughout the remainder of the paper.
For illustrative purposes, we present this upper bound for the point-to-line free energy, bounding it with the maximum version.
\begin{proposition}
Let $\log Z_{0, \mathcal{L}_N}$ be the point-to-line free energy; then the following holds:
$$\log Z_{0, \mathcal{L}_N} \leq \max_{\mathbf{v} \in \mathcal{L}_N} \log Z_{0, \mathbf{v}} + \log (2N+1).$$
\end{proposition}
This directly follows from the fact that
$Z_{0, \mathcal{L}_N} = \sum_{\mathbf{v} \in \mathcal{L}_N} Z_{0, \mathbf{v}} \leq (2N+1) \cdot \max_{\mathbf{v} \in \mathcal{L}N} Z_{0, \mathbf{v}}$. We also note that the fluctuation of both $\log Z_{0, \mathcal{L}_N}$ and $\max_{\mathbf{v} \in \mathcal{L}_N} Z_{0, \mathbf{v}}$ are of order $N^{1/3}$; thus, the $\log (2N+1)$ term is significantly smaller, which does not affect the estimates significantly. Also, in the application of this upper bound, to simplify the notation we may use $2\log N$ instead of  $\log (2N+1)$.}

\section{Local fluctuations}\label{loc_fluc}

In this section, we look at fluctuations for the polymer near $0$ or $(N,N)$ where the time scale can be much smaller than the full scale $N$. 
We start with an estimate for the fluctuation of the free energy profile along the anti-diagonal line. 
This result was first proved for a zero-temperature model (Brownian last-passage percolation) in \cite{BR} using the Brownian Gibbs property. 
Other related results and extensions for the various zero-temperature models have appeared in \cite{timecorriid, timecorrflat, BRKPZ}. Compared to these, our proof does not rely on integrable probability which was used in \cite{BRKPZ, BR}, and we improve the tail estimate from \cite{timecorriid, timecorrflat} to optimal order. 

\begin{proposition} \label{compare_max}
There exist positive constants $\newc\label{compare_max_c1}, \newc\label{compare_max_c2},    c_0, N_0$ such that for each $N\geq N_0$, $1 \leq t \leq c_0 N^{1/2}$, and each $a \in \mathbb{Z}_{\geq 0}$, we have
$$\mathbb{P}\Big(\log Z_{0,\mathcal{L}^a_{N}} - \log Z_{0, N} \geq \oldc{compare_max_c1}t \sqrt{a}\Big) \leq 
e^{-\oldc{compare_max_c2} \min\{t^{2}\!,\, t\sqrt{a}\}}.
$$ 
\end{proposition}

\begin{remark}
{Since the free energy profile $\{\log Z_{0, (N+k, N-k)} -  \log Z_{0, N}\}_{k\in \mathbb{Z}}$ is expected to be locally Brownian after the KPZ rescaling}, the difference of the free energies in the probability above should approximate the running maximum of a two-sided random walk. Thus the tail bound is of optimal exponential order. \end{remark}

\begin{proof}
The case $a=0$ is trivial, so we will always assume $a \in \mathbb{Z}_{>0}$.
As we previously discussed in Section \ref{max_bd}, we may prove the proposition with the maximum version of the free energy since 
$$\log Z_{0,N} \leq \log Z_{0,\mathcal{L}^a_{N}}  \leq \log Z^{\textup{max}}_{0,\mathcal{L}^a_{N}} + 10\log (a+1).$$

Let us also note that when $a\geq t^{2/3}N^{2/3}$, the estimate is straightforward. It holds that 
\begin{align*}
&\mathbb{P}\Big(\log Z^{\textup{max}}_{0,\mathcal{L}^a_{N}}  - \log  Z_{0, N}  \geq Ct\sqrt{a}\Big)\\
& \leq \mathbb{P}\Big(\log Z^{\textup{max}}_{0,\mathcal{L}_{N}}  - \log  Z_{0, N}  \geq C\tfrac{\sqrt{a}}{t^{1/3}N^{1/3}}t^{4/3} N^{1/3}\Big)\\
& \leq \mathbb{P}\Big(\log Z^{\textup{max}}_{0,\mathcal{L}_{N}}  - \log  Z_{0, N}  \geq Ct^{4/3}N^{1/3}\Big)\\
& \leq \mathbb{P}\Big(\log Z^{\textup{max}}_{0,\mathcal{L}_{N}}  - 2N{f_d} \geq \tfrac{C}{2}t^{4/3} N^{1/3}\Big) \\
& \qquad \qquad \qquad +\mathbb{P}\Big(\log 
 Z_{0, N}  - 2N{f_d} \leq -\tfrac{C}{2}t^{4/3} N^{1/3}\Big)
 \leq e^{-Ct^{2}},
\end{align*}
where the last inequality comes from Proposition \ref{lem_ptl} and Proposition \ref{low_ub}.

From now on, we will assume that the integer $a$ satisfies $1\leq a \leq t^{2/3}N^{2/3}$. In addition, note that our estimate for the difference of two free energies does not change if we included the weight $Y_{(0,0)}$ in both partition functions. For the remaining part of the proof, we will also assume this without introducing a new notation for this version of the partition function.

By a union bound, it suffices to prove our estimate for 
\begin{equation}\label{goal1}
\mathbb{P} \Big(\log Z^{\textup{max}}_{0,\mathcal{L}^{a, +}_{N}}  - \log  Z_{0, N} \geq C't\sqrt{a}\Big)
\end{equation}
where $\mathcal{L}^{a, +}_{N}$ is part of $\mathcal{L}^{a}_{N}$ above $(N,N)$.
For any fixed $k =0, \dots, a$, let us rewrite 
$$\log  Z_{0, (N-k, N+k)}  - \log  Z_{0, N} = \sum_{i=1}^k \log Z_{0, (N-i, N+i)} - \log Z_{0, (N-(i-1), N+(i-1))} = S_k.$$
This allows us to work with a running maximum of the walk $S_k$ since
\begin{equation}\label{max_walk}
\eqref{goal1} = \mathbb{P} \Big(\max_{0\leq k \leq a} S_k \geq C't\sqrt{a}\Big).
\end{equation}
The steps of $S_k$ are not i.i.d., however, Theorem \ref{rwrw} allows us to work with an i.i.d.~random walk $\wt S_k$ which upper bounds $S_k$ with high probability. More precisely, the down-right path $\Theta_{2a}$ will be the staircase from $(N-a, N+a)$ to $(N,N)$. Because the steps of $S_k$ and the free energy profile defined in \eqref{profile} differ by a negative sign, the perturbed parameter will be $\eta =  \mu/2 - q_0 t^{2/3}N^{-1/3}$, and the distribution of the steps of $\wt S_k $ is given by $\log(\textup{Ga}^{-1}(\eta)) -  \log(\textup{Ga}^{-1}(\mu-\eta))$.

Let $A$ denote the event that $\log \tfrac{10}{9} + \wt S_k \geq S_k$ for each $k = 0, 1 \dots, a$.
Then, we have \begin{align*}
\eqref{max_walk} &\leq \mathbb{P} \Big(\Big\{\max_{0\leq k \leq a} S_k \geq C'\sqrt{a}t^{3/4}\Big\} \cap A\Big) + \mathbb{P}(A^c)\\
&\leq \mathbb{P} \Big(\Big\{\log \tfrac{10}{9} + \max_{0\leq k \leq a} \wt{S}_k  \geq C'\sqrt{a}t^{3/4}\Big\} \Big) + \mathbb{P}(A^c).
\end{align*}
From Theorem \ref{rwrw}, we know $\mathbb{P}(A^c) \leq e^{-Ct^{2}}$. Absorb the constant $\log({10}/{9})$ into the constant $C$, and it suffices to obtain the  upper bound 
\begin{equation}\label{rw_before_E}
\mathbb{P} \Big(\max_{0\leq k \leq a} \wt{S}_k  \geq C't \sqrt{a} \Big) \leq e^{-C \min\{t^{2}, t\sqrt{a}\}}.
\end{equation}
This is a standard running maximum estimate for an i.i.d.~random walk whose steps are sub-exponential. We omit the details here and postpone the proof of \eqref{rw_before_E}  to the end of Appendix \ref{sec_sub_exp}.

\end{proof}

Next, we extend the value of $t$ in the previous proposition from $t_0 \leq t \leq c_0N^{1/2}$ to all $t\geq t_0$. The cost of this is a non-optimal exponent appearing in the exponential bound.

\begin{proposition}\label{max_all_t}
There exist positive constants $t_0, N_0$ such that for each $N\geq N_0$, $t\geq t_0$, and each $a \in \mathbb{Z}_{\geq 0}$, we have
$$\mathbb{P}\Big(\log Z_{0,\mathcal{L}^a_{N}}\ - \log Z_{0, N} \geq t\sqrt{a}\Big) \leq 
e^{-t^{1/10}}.
$$ 
\end{proposition}

\begin{proof}The case $a=0$ is trivial, so we will always assume $a \geq 1$. Due to Proposition \ref{compare_max}, we only have to show the estimate when $t\geq {\oldc{compare_max_c1}} c_0 N^{1/2}$ where both constants ${\oldc{compare_max_c1}}$ and $c_0$ are from Proposition \ref{compare_max}. 
Suppose $t = zN^{1/2}$ where $z\geq {\oldc{compare_max_c1}}c_0$. Then, 
\begin{align*}
&\mathbb{P}\Big(\log Z_{0,\mathcal{L}^a_{N}}\ - \log Z_{0, N} \geq t \sqrt{a}\Big)
 \leq \mathbb{P}\Big(\log Z_{0,\mathcal{L}^a_{N}}\ - \log Z_{0, N} \geq t\Big)\\
&\quad = \mathbb{P}\Big(\log Z_{0,\mathcal{L}^a_{N}}\ - \log Z_{0, N} \geq (zN^{1/6}) N^{1/3}\Big)\\
&\quad \leq \mathbb{P}\Big(\log Z_{0,\mathcal{L}^a_{N}}\ - 2Nf_d  \geq (\tfrac{1}{2}zN^{1/6}) N^{1/3}\Big) 
+ \mathbb{P}\Big( \log Z_{0, N} - 2Nf_d \leq -(\tfrac{1}{2} zN^{1/6}) N^{1/3}\Big)\\
&\quad\leq e^{-t^{1/10}}.
\end{align*}
The last inequality comes from Proposition \ref{lem_ptl} and Proposition \ref{low_ub}.
\end{proof}

Fix $0\leq r \leq N/2$. Recall that $\mathcal{L}_r$ is the anti-diagonal through the point $(r,r)$.  Let ${\bf p}_*$ denote the random maximizer in   
$$\max_{{\bf p}\in \mathcal{L}_r} \big\{\log Z_{0, {\bf p}}  + \log Z_{{\bf p}, N} \big\} = \log Z_{0, {\bf p}_*}  + \log Z_{{\bf p}_*, N} .$$
The proposition below captures the KPZ transversal fluctuation which says that the maximizer ${\bf p}_*$ cannot be too far from the diagonal on the local scale $r^{2/3}$. This can of course   be much smaller than the global fluctuation scale $N^{2/3}$. This result was first proved in the zero-temperature model in \cite{ubcoal}.

\begin{proposition}\label{fluc_bound_l}
There exist positive constants  $\newc\label{fluc_bound_l_c1}, c_0, t_0, N_0 $ such that for each $N\geq N_0$, $c_0 \leq r \leq N/2$ and $t\geq t_0$, we have
$$\mathbb{P}(|{\bf p}_*-(r,r)|_\infty > tr^{2/3}) \leq e^{-\oldc{fluc_bound_l_c1}{t}^3}.$$
\end{proposition}
\begin{proof}
Abbreviate  $J^{h}=\mathcal{L}^{r^{2/3}}_{(r-2hr^{2/3}, r+2hr^{2/3})}$.  
We bound the probability  as follows. 
\begingroup
\allowdisplaybreaks
\begin{align}
&\mathbb{P}(|{\bf p}_*-(r,r)|_\infty > tr^{2/3}) \nonumber\\
&\leq \mathbb{P}\Big(\max_{{\bf p}\in \mathcal{L}_r\setminus \mathcal{L}_r^{tr^{2/3}}}\Big\{\log Z_{0,{\bf p}}  + \log Z_{{\bf p}, N} \Big \}  > \log Z_{0, r}  + \log Z_{r, N} \Big)\nonumber\\
&\leq \sum_{|h|= \floor{t/2}}^{r^{1/3}} \mathbb{P}\Big( \log Z^{\textup{max}}_{0, J^h} + \log Z^{\textup{max}}_{J^h, N}  > \log Z_{0, r}  + \log Z_{r, N}    \Big)\nonumber\\
&=  \sum_{|h|= \floor{t/2}}^{r^{1/3}}\mathbb{P}\Big(\Big[ \log Z^{\textup{max}}_{0, J^h} - \log Z_{0, r} \Big]   + \Big[\log Z^{\textup{max}}_{J^h, N} - \log Z_{r, N}  \Big]    > 0\Big)\nonumber\\
& \leq  \sum_{|h|= \floor{t/2}}^{r^{1/3}} \bigg[\mathbb{P}\Big( \log Z^{\textup{max}}_{0, J^h } - \log Z_{0, r}  \geq -Dh^2r^{1/3} \Big)\label{close_max}\\
&\qquad \qquad \qquad \qquad + \mathbb{P}\Big(\log Z^{\textup{max}}_{ J^h , N} - \log Z_{r, N}  \geq Dh^2r^{1/3}\Big)\bigg].\label{far_max}
\end{align}
\endgroup
where $D$ is a small positive constant that 
we will fix.

For \eqref{close_max}, provided $t_0$ is fixed sufficiently large,
we may upper bound \eqref{close_max} using Proposition \ref{trans_fluc_loss} and Proposition \ref{low_ub} as the following
\begin{align*}
& \mathbb{P}\Big([ \log Z^{\textup{max}}_{0, J^h} -2r{f_d}] -[\log Z_{0, r} -2r{f_d}] \geq -Dh^2r^{1/3} \Big)\\
& \leq \mathbb{P}\Big(\log Z^{\textup{max}}_{0, J^h} -2r{f_d} \geq -2Dh^2r^{1/3} \Big) + \mathbb{P}\Big(\log Z_{0, r} -2r{f_d} \leq - Dh^2r^{1/3} \Big) 
\leq e^{-C|h|^3}
\end{align*}
provided $D \leq \tfrac{1}{10}\oldc{trans_fluc_loss_c1}$ from Proposition \ref{trans_fluc_loss}.

For \eqref{far_max}, we will split the value of $r$ into two cases, whether $r\leq \epsilon_0 (N-r)$ or $r\geq \epsilon_0 (N-r)$, for $\epsilon_0$ which we will fix below (between \eqref{prob_estt} and \eqref{fix_e_0}).
When $r\leq \epsilon_0 (N-r)$, we upper bound \eqref{far_max} by 
\begin{equation}\label{bound_small_r}
\eqref{far_max} \leq \mathbb{P}\Big(\log Z^{\textup{max}}_{\mathcal{L}_r^{4|h|r^{2/3}}, N} - \log Z_{r, N}  \geq Dh^2r^{1/3}\Big), 
\end{equation}
and we would like to apply Proposition \ref{compare_max}. From there, we let $a = 4|h|r^{2/3}$ and $t = 8|h|^{3/2}\epsilon_0^{3/2}$. 
Then, continue from \eqref{bound_small_r}, 
\begin{align}\eqref{bound_small_r} &\leq \mathbb{P}\Big(\log Z^{\textup{max}}_{\mathcal{L}_r^{4|h|r^{2/3}}, N} - \log Z_{r, N}  \geq \tfrac{1}{10}D|h|^{3/2}\sqrt{4|h|r^{2/3}}\Big)\nonumber\\
&\leq  \mathbb{P}\Big(\log Z^{\textup{max}}_{\mathcal{L}_r^{a}, N} - \log Z_{r, N}  \geq \tfrac{\frac{1}{10}D}{8\epsilon_0^{3/2}}t\sqrt{a}\Big). \label{prob_estt}
\end{align}
Next, we fix $\epsilon_0$ sufficiently small so that $t  \leq \wt c_0 r^{1/2} $, where $\wt c_0$ is the constant $c_0$ from Proposition \ref{compare_max}. Then,  we lower the value of $D$ to get
\begin{equation}\label{fix_e_0}
\tfrac{\frac{1}{10}D}{8\epsilon_0^{3/2}} \leq \oldc{compare_max_c1}
\end{equation}
where \oldc{compare_max_c1} is the constant appearing in Proposition \ref{compare_max}. Finally, by Proposition \ref{compare_max}, the above probability in \eqref{prob_estt} will always be bounded by $e^{-Ct^{2}} = e^{-C|h|^3}$.

On the other hand, when $r\geq \epsilon_0(N-r)$, if the maximizer ${\bf p}_*$ is located more than $|h|r^{2/3}$ away from the diagonal, {it means it is more than $(\epsilon_0/2)^{2/3}|h|(N-r)^{2/3}$, which is the same order as  $(N-r)$, away from the diagonal}. Provided that $t_0$ is fixed sufficiently large depending on $\epsilon_0$, \eqref{far_max} can be upper bounded with a similar argument as in \eqref{close_max}. 

To summarize, the arguments above show that 
$$\sum_{|h| = \floor{t/2}}^{r^{1/3}} \eqref{close_max} + \eqref{far_max}\leq  \sum_{|h| = \floor{t/2}}^{\infty}e^{-C|h|^3} \leq e^{-Ct^3},$$
with this, we have finished the proof of this proposition.
\end{proof}

By symmetry, similar results also hold for the case when $N/2 \leq r \leq N$. We record this in the following proposition.  Let ${\bf p}^*$ denote the random maximizer of 
$$\max_{{\bf p}\in \mathcal{L}_r} \{\log Z_{0, {\bf p}}  + \log Z_{{\bf p}, N}\}.$$

\begin{proposition}\label{fluc_bound_u}
There exist positive constants $\newc\label{fluc_bound_u_c1}, c_0, N_0, t_0$ such that for each $N\geq N_0$, $N/2 \leq r \leq N-c_0$ and $t\geq t_0$, we have
$$\mathbb{P}(|{\bf p}^*-(r,r)|_\infty > t(N-r)^{2/3}) \leq e^{-\oldc{fluc_bound_u_c1}{t}^3}.$$
\end{proposition}

The next estimate quantifies the effect of letting the crossing point of the path on an antidiagonal fluctuate in the KPZ scale versus forcing the path to go through a fixed point.
\begin{proposition} \label{nest_max}
There exist positive constants $c_0, t_0, N_0, \epsilon_0$ such that for each $N\geq N_0$, $N/2 \leq r \leq N-c_0$,  $t\geq t_0$, we have
\begin{align*}
\mathbb{P} \Big( \max_{{\bf p}\in \mathcal{L}_r^{t(N-r)^{2/3}}}\Big \{\log Z_{0,{\bf p}}  + \log Z_{{\bf p} , N} \Big\} 
- \Big[\log Z_{0, r} + \log Z_{r,N} \Big]\geq t(N-r)^{1/3}\Big) \leq e^{- t^{1/10}}.
\end{align*}
\end{proposition}

\begin{proof}
Let us start by rewriting
\begin{align}
&\mathbb{P} \Big( \max_{{\bf p}\in \mathcal{L}_r^{t(N-r)^{2/3}}}\Big\{\log Z_{0,{\bf p}}  + \log Z_{{\bf p} , N}\Big\}  - \Big[\log  Z_{0, r}  + \log Z_{r,N}  \Big]\geq t(N-r)^{1/3}\Big)\nonumber\\
& \leq   \mathbb{P} \Big(\log Z^{\textup{max}}_{0,\mathcal{L}_r^{t(N-r)^{2/3}}}  - \log  Z_{0, r}  \geq \tfrac{1}{2}t(N-r)^{1/3}\Big) \label{far_term} \\
&\qquad \qquad + \mathbb{P} \Big(\log Z^{\textup{max}}_{\mathcal{L}_r^{t(N-r)^{2/3}}, N}  - \log Z_{r,N} \geq \tfrac{1}{2}t(N-r)^{1/3}\Big). \label{close_term}
\end{align}
Both probabilities \eqref{far_term} and \eqref{close_term} can be upper bounded by $e^{-Ct^{1/10}}$ using Proposition \ref{max_all_t}.
\end{proof}

We combine the  previous propositions into  the following statement.  
\begin{proposition} \label{nest_max1}
There exist positive constants $ c_0, t_0, N_0, \epsilon_0$ such that for each $N\geq N_0$, $N/2 \leq r \leq N-c_0$,  $t\geq t_0$, we have
\begin{align*}
&\mathbb{P} \Big( \max_{{\bf p}\in \mathcal{L}_r} \Big \{\log Z_{0,{\bf p}}  + \log Z_{{\bf p} , N} \Big\} - \Big[\log  Z_{0, r}  + \log Z_{r,N} \Big]\geq t(N-r)^{1/3}\Big) \leq e^{- t^{1/10}}.
\end{align*}
\end{proposition}

\begin{proof}
By a union bound, we split the above maximum over ${\bf p}\in \mathcal{L}_r$ to ${\bf p}\not\in \mathcal{L}_r^{t(N-r)^{2/3}}$ and ${\bf p}\in \mathcal{L}_r^{t(N-r)^{2/3}}$.  Propositions \ref{fluc_bound_u} and  \ref{nest_max} show that in both cases the probability can be upper bounded by $e^{-Ct^{1/10}}$. 
\end{proof}

The development culminates in the following theorem.
\begin{theorem}\label{nest}
There exist positive constants $c_0, t_0, N_0, $ such that for each $N\geq N_0$, $N/2 \leq r \leq N-c_0$,  $t\geq t_0$, we have
\begin{align*}
&\mathbb{P} \Big( \log Z_{0,N}  - [\log  Z_{0, r}  + \log Z_{r,N}  ]\geq t(N-r)^{1/3}\Big) \leq e^{-t^{1/10}}.
\end{align*}
\end{theorem}

\begin{proof}
This follows directly from the fact that 
$$\log Z_{0, N} \leq \max_{{\bf p}\in \mathcal{L}_r}\{\log Z_{0,{\bf p}}  + \log Z_{{\bf p} , N}\} + 2\log(N-r)$$
and Proposition \ref{nest_max1}.
\end{proof}

\section{Proof of Theorem \ref{thm_r_large}} \label{pf1}
This section proves Theorem \ref{thm_r_large}. Throughout, $c_0 \leq N-r \leq N/2$ is assumed. Using the following identity, {where the first equality below comes from performing the derivative test for $\lambda$ and finding the minimum,  }
\begin{equation}\label{var_id}
\begin{aligned}
    \Var(U-V) &\geq  \inf_{\lambda\in \mathbb R}\Var(U-\lambda V) = (1-\textup{$\mathbb{C}$orr}^2(U,V))\Var(U)\\
    &= (1-\textup{$\mathbb{C}$orr}(U,V))(1+\textup{$\mathbb{C}$orr}(U,V))\Var(U).
\end{aligned}\end{equation}
Apply this to bound $1-\textup{$\mathbb{C}$orr}(U,V)$ for $U = \log Z_{0,N} $ and $V= \log Z_{0,r} $. 
By the FKG inequality,  $\textup{$\mathbb{C}$orr}(\log Z_{0,N} ,\log Z_{0,r} ) \in [0,1]$.
\eqref{var_id} gives  
\begin{equation}\label{goal_thm_r_large}
 \frac{\inf_{\lambda\in \mathbb R}\Var(\log Z_{0,N} -\lambda \log Z_{0,r} )}{2\Var(\log Z_{0,N} )}\leq 1-\textup{$\mathbb{C}$orr}(\log Z_{0,N} ,\log Z_{0,r} ) \leq \frac{\Var(\log Z_{0,N} -\log Z_{0,r} )}{\Var (\log Z_{0,N} )}.
\end{equation}

Since Theorem \ref{var} gives $\Var(\log Z_{0,N}) \geq CN^{2/3}$, 
 the lower bound of Theorem \ref{thm_r_large} follows from the second inequality of \eqref{goal_thm_r_large} and 
\begin{equation}\label{nest_var}
\Var( \log Z_{0,N}  - \log Z_{0,r} ) \leq C(N-r)^{2/3}.
\end{equation}

 To show \eqref{nest_var}, apply the inequality $\Var(A) \leq 2(\Var(B) + \mathbb{E}[(A-B)^2])$ to $A = \log Z_{0,N}  - \log Z_{0,r}$ and $
B = \log Z_{r,N} $. 
$\Var(B) \leq C(N-r)^{2/3}$ follows from Theorem \ref{var}, and $\mathbb{E}[(A-B)^2]\leq C(N-r)^{2/3}$ follows from Theorem \ref{nest}.  The proof of the lower bound of Theorem \ref{thm_r_large} is complete. 

\medskip

We turn to prove the upper bound of Theorem \ref{thm_r_large},  by bounding a conditional variance.  
Recall that $[\![0, (N,N)]\!]$ is the square with lower left corner at $(0,0)$ and upper right corner at $(N, N)$.
Let $\mathcal{F}$ be the $\sigma$-algebra of the weights in $[\![0, (N,N)]\!]$ that lie on or below the anti-diagonal line $\mathcal{L}_r$. Note that $\log Z_{0,r}$ is $\mathcal F$-measurable
\begin{align} 
\Var(\log Z_{0,N}  | \mathcal F)  
&= \Var(\log Z_{0,N}  - \log Z_{0,r}  | \mathcal F) \nn \\
\label{lb_here}  &
=  \mathbb{E}\Big[\Big(\log Z_{0,N}  - \log Z_{0,r} - \mathbb{E}[\log Z_{0,N}  - \log Z_{0,r} |\mathcal{F}]\Big)^2\Big| \mathcal{F}\Big].
\end{align}
We develop a lower bound for the last conditional expectation above.

By Theorem \ref{nest}, 
\begin{equation} \label{expct_bd}
\Big|\mathbb{E}[\log Z_{0,N}  - \log Z_{0,r} ] - \mathbb{E}[\log Z_{r,N}]\Big| \leq C(N-r)^{1/3}.
\end{equation}
In Proposition \ref{up_lb} the centering  $2N{f_d}$ can be replaced with $\mathbb E[\log Z_{0, N}]$ because $\mathbb E[\log Z_{0, N}]\le 2N{f_d}$ by superadditivity. 
Thus altered, Proposition \ref{up_lb} and \eqref{expct_bd} give
\begin{align*}
 e^{-\oldc{up_lb_c1} t^{3/2}}  &\le   \mathbb{P}\bigl(\log Z_{r, N} - \mathbb E[\log Z_{r, N}] \geq t(N-r)^{1/3}\bigr) \\
 &\le \mathbb{P}\bigl(\log Z_{r, N}  - \mathbb{E}[\log Z_{0,N}  - \log Z_{0,r} ] \geq (t-C)(N-r)^{1/3}\bigr). 
\end{align*}
Let $s_0$ be a large constant and define the event 
\be\label{ArN}  A_{r,N}=\bigl\{ \log Z_{r,N}   - \mathbb{E}[\log Z_{0,N}  - \log Z_{0,r} ] \geq s_0(N-r)^{1/3}\bigr\} . 
\ee
   $A_{r,N}$ is independent of $\mathcal{F}$ and  $\mathbb P(A_{r,N})$ is bounded below   independently of $r$ and $N$.

Next, using Chebyshev's inequality
 we get
\begin{align*}
&\mathbb{P}\Big(\Big|\mathbb{E}[\log Z_{0,N}  - \log Z_{0,r} |\mathcal{F}]-\mathbb{E}[\log Z_{0,N}  - \log Z_{0,r} ]\Big|> t(N-r)^{1/3}\Big) \\
&\leq  \frac{\Var(\mathbb{E}[\log Z_{0,N}  - \log Z_{0,r} |\mathcal{F}])}{t^2 (N-r)^{2/3}}\\ 
\quad &\leq\frac{\Var(\log Z_{0,N}  - \log Z_{0,r} )}{t^2 (N-r)^{2/3}} 
\leq C/t^2 \qquad \textup{by \eqref{nest_var}}.
\end{align*}
By choosing $t$ and $s_0$ large enough, 
  there is an event $B_{r,N}\in \mathcal{F}$, with positive probability  bounded below  independently of $N$ and $r$, on which  
\be\label{BrN} 
\Big|\mathbb{E}[\log Z_{0,N}  - \log Z_{0,r} |\mathcal{F}]-\mathbb{E}[\log Z_{0,N}  - \log Z_{0,r} ]\Big|\leq \frac{s_0}{10}(N-r)^{1/3}.
\ee

On $A_{r,N}\cap B_{r,N}$ we have the following bound, using first superadditivity  $\log Z_{0,N}  - \log Z_{0,r}  \geq \log Z_{r,N} $, then \eqref{BrN} and last \eqref{ArN}: 
\begin{align*}
& \log Z_{0,N}  - \log Z_{0,r} - \mathbb{E}[\log Z_{0,N}  - \log Z_{0,r} |\mathcal{F}] \\
& \geq   \log Z_{r,N} - \mathbb{E}[\log Z_{0,N}  - \log Z_{0,r} ] - \frac{s_0}{10}(N-r)^{1/3}  
\geq \frac{9s_0}{10}(N-r)^{1/3}.
\end{align*}
Square this bound and insert  it inside   the conditional expectation on line  \eqref{lb_here}. 
Continuing from that line, we then have 
\begin{align*} 
\Var(\log Z_{0,N}  | \mathcal F) 
\ge C (N-r)^{2/3}  \,
\mathbb{E}[ \ind_{A_{r,N}}\ind_{B_{r,N}} |  \mathcal F  ]
\ge  C (N-r)^{2/3} \, 
 \ind_{B_{r,N}}. 
\end{align*} 
By the law of total variance,   for all $\lambda\in\R$,  
\begin{align*}
&\Var( \log Z_{0,N}  - \lambda \log Z_{0, r}  )\\
&= \mathbb{E}\big[\Var( \log Z_{0,N}  - \lambda \log Z_{0, r} | \mathcal F ) \big] + \Var\big[ \mathbb{E}(\log Z_{0,N}  - \lambda \log Z_{0, r} |\mathcal F ) \big]\\
& \geq \mathbb{E}\big[\Var( \log Z_{0,N}  - \lambda \log Z_{0, r} | \mathcal F) \big] \\
&= \mathbb{E}\big[\Var( \log Z_{0,N} | \mathcal F) \big]
\geq C(N-r)^{2/3} \, \mathbb P(B_{r,N}) 
\geq C(N-r)^{2/3} . 
\end{align*}
Apply this lower bound to the numerator of the first member of \eqref{goal_thm_r_large} and apply Theorem \ref{var} to the denominator. The upper bound of Theorem \ref{thm_r_large} has been established.

\section{Proof of Theorem \ref{thm_r_small}} \label{pf2}
We assume throughout that $c_0 \leq r \leq N/2$.
First, we prove the upper bound. By the Cauchy-Schwarz inequality and the independence of $Z_{0,r}$ and $Z_{r,N}$, 
\begin{align*}
\Cov (\log Z_{0,r} , \log Z_{0,N}) &= \Cov(\log Z_{0,r} , \log Z_{0,N} -\log Z_{r,N} ) \\
&\leq \Var(\log Z_{0,r})^{1/2} \cdot \Var( \log Z_{0,N}  -\log Z_{r,N}  )^{1/2}.
\end{align*}
It therefore suffices to show that both variances above have upper bounds of the order $r^{2/3}$. The first variance satisfies $\Var(\log Z_{0,r})   \leq Cr^{2/3}$ by  Theorem \ref{var}. The second variance can be  bounded again using the inequality  $\Var(A) \leq 2(\Var(B) + \mathbb{E}[(A-B)^2])$  with $A = \log Z_{0,N}  - \log Z_{r,N}$ and $
B = \log Z_{0,r}$. 
$\Var(B) \leq Cr^{2/3}$ follows from Theorem \ref{var}, and $\mathbb{E}[(A-B)^2]\leq Cr^{2/3}$ from Proposition \ref{nest} with the parameters $r$ and $N-r$ swapped and {the fact that $A-B \geq 0$}.
This finishes the proof of the upper bound. The remainder of this section is dedicated to the lower bound of Theorem \ref{thm_r_small}. 

Our approach follows  ideas from \cite{timecorriid, timecorrflat} that  we now describe. For $\theta>0$, let $\mathcal{F}_\theta$ denote the $\sigma$-algebra generated by the weights in the set $[\![(0,0), (N,N)]\!] \setminus R^{\theta r^{2/3}}_{0,r}$. In Section \ref{cov_lb}, we will show that there exists an event $\mathcal{E}_\theta\in \mathcal{F}_\theta$ with  $\mathbb{P}(\mathcal{E}_\theta) \geq \epsilon_0 > 0$ ($\epsilon_0$ independent of $r$ and $N$) such that  
\begin{equation}\label{main_goal}
\Cov(\log Z_{0,N} , \log Z_{0,r}  | \mathcal{F}_\theta)(\omega) \geq C r^{2/3} \quad \textup{ for } \omega\in \mathcal{E}_\theta.
\end{equation}
{Since the free energy is increasing in the i.i.d.\ environment}, by \eqref{main_goal} and applying the FKG inequality twice, , we have 
\begin{align*}
&\mathbb{E}[\log Z_{0, N}\log Z_{0,r}] \\& =\mathbb{E}\Big[\mathbb{E}[\log Z_{0, N}\log Z_{0,r}|\mathcal{F}_\theta]\Big]\\
& = \int_{\mathcal{E}_\theta} \mathbb{E}[\log Z_{0, N}\log Z_{0,r}|\mathcal{F}_\theta] \,d\mathbb P + \int_{\mathcal{E}_\theta^c} \mathbb{E}[\log Z_{0, N}\log Z_{0,r}|\mathcal{F}_\theta] \,d\mathbb P\\
&\geq \int_{\mathcal{E}_\theta} \mathbb{E}[\log Z_{0, N}|\mathcal{F}_\theta] \mathbb{E}[\log Z_{0,r}|\mathcal{F}_\theta]\,d\mathbb P + C\epsilon_0 r^{2/3} +  \int_{\mathcal{E}_\theta^c} \mathbb{E}[\log Z_{0, N}|\mathcal{F}_\theta] \mathbb{E}[\log Z_{0,r}|\mathcal{F}_\theta]\,d\mathbb P\\
&=\mathbb{E}\Big[\mathbb{E}[\log Z_{0, N}|\mathcal{F}_\theta]\mathbb{E}[\log Z_{0,r}|\mathcal{F}_\theta]\Big] + C\epsilon_0 r^{2/3}\\
&\geq \mathbb{E}[ \log Z_{0, N} ] \mathbb{E}[\log Z_{0,r} ] + C\epsilon_0 r^{2/3}.
\end{align*}
This shows $\Cov(\log Z_{0,N} , \log Z_{0,r}) \geq Cr^{2/3}$, hence the lower bound in our theorem. In the next few sections, we prove \eqref{main_goal}.

\subsection{Barrier event $\barevent$}

\begin{figure}
\begin{center}

\tikzset{every picture/.style={line width=0.75pt}} 

\tikzset{every picture/.style={line width=0.75pt}} 

\begin{tikzpicture}[x=0.75pt,y=0.75pt,yscale=-1,xscale=1]

\draw   (103.43,74.29) -- (249.55,190.22) -- (166.08,295.6) -- (19.97,179.66) -- cycle ;
\draw    (79.82,227.69) -- (163.56,121.27) ;
\draw    (104.22,246.99) -- (187.46,140.64) ;
\draw  [dash pattern={on 0.84pt off 2.51pt}]  (110.11,67.85) -- (181.51,123.31) ;
\draw [shift={(181.51,123.31)}, rotate = 217.83] [color={rgb, 255:red, 0; green, 0; blue, 0 }  ][line width=0.75]    (0,5.59) -- (0,-5.59)   ;
\draw [shift={(110.11,67.85)}, rotate = 217.83] [color={rgb, 255:red, 0; green, 0; blue, 0 }  ][line width=0.75]    (0,5.59) -- (0,-5.59)   ;
\draw  [dash pattern={on 0.84pt off 2.51pt}]  (74.22,234.98) -- (86.98,245.08) ;
\draw [shift={(86.98,245.08)}, rotate = 218.35] [color={rgb, 255:red, 0; green, 0; blue, 0 }  ][line width=0.75]    (0,5.59) -- (0,-5.59)   ;
\draw [shift={(74.22,234.98)}, rotate = 218.35] [color={rgb, 255:red, 0; green, 0; blue, 0 }  ][line width=0.75]    (0,5.59) -- (0,-5.59)   ;
\draw [line width=2.25]    (156.92,180.04) -- (217.59,229.85) ;
\draw [line width=2.25]    (139.92,202.66) -- (201.53,251.86) ;
\draw  [fill={rgb, 255:red, 0; green, 0; blue, 0 }  ,fill opacity=1 ] (90.62,237.48) .. controls (90.46,236.04) and (91.43,234.74) .. (92.8,234.58) .. controls (94.17,234.42) and (95.42,235.46) .. (95.59,236.9) .. controls (95.76,238.34) and (94.78,239.64) .. (93.41,239.8) .. controls (92.04,239.95) and (90.79,238.92) .. (90.62,237.48) -- cycle ;
\draw  [fill={rgb, 255:red, 0; green, 0; blue, 0 }  ,fill opacity=1 ] (172.93,131.73) .. controls (172.76,130.29) and (173.74,129) .. (175.11,128.84) .. controls (176.48,128.68) and (177.73,129.72) .. (177.9,131.16) .. controls (178.07,132.6) and (177.09,133.89) .. (175.72,134.05) .. controls (174.35,134.21) and (173.1,133.17) .. (172.93,131.73) -- cycle ;
\draw  [draw opacity=0] (402.18,72.13) -- (570.19,214.16) -- (505.63,290.53) -- (337.62,148.5) -- cycle ; \draw   (417.46,85.04) -- (352.9,161.41)(432.73,97.95) -- (368.17,174.32)(448,110.87) -- (383.44,187.23)(463.28,123.78) -- (398.72,200.15)(478.55,136.69) -- (413.99,213.06)(493.82,149.6) -- (429.27,225.97)(509.1,162.51) -- (444.54,238.88)(524.37,175.43) -- (459.81,251.79)(539.65,188.34) -- (475.09,264.71)(554.92,201.25) -- (490.36,277.62) ; \draw   (389.27,87.4) -- (557.28,229.43)(376.36,102.68) -- (544.37,244.71)(363.45,117.95) -- (531.46,259.98)(350.54,133.23) -- (518.55,275.26) ; \draw   (402.18,72.13) -- (570.19,214.16) -- (505.63,290.53) -- (337.62,148.5) -- cycle ;
\draw  [dash pattern={on 0.84pt off 2.51pt}]  (563.4,234.9) -- (576.4,219.9) ;
\draw [shift={(576.4,219.9)}, rotate = 130.91] [color={rgb, 255:red, 0; green, 0; blue, 0 }  ][line width=0.75]    (0,5.59) -- (0,-5.59)   ;
\draw [shift={(563.4,234.9)}, rotate = 130.91] [color={rgb, 255:red, 0; green, 0; blue, 0 }  ][line width=0.75]    (0,5.59) -- (0,-5.59)   ;
\draw  [dash pattern={on 0.84pt off 2.51pt}]  (438.36,91.76) -- (453.56,104.16) ;
\draw [shift={(453.56,104.16)}, rotate = 219.21] [color={rgb, 255:red, 0; green, 0; blue, 0 }  ][line width=0.75]    (0,5.59) -- (0,-5.59)   ;
\draw [shift={(438.36,91.76)}, rotate = 219.21] [color={rgb, 255:red, 0; green, 0; blue, 0 }  ][line width=0.75]    (0,5.59) -- (0,-5.59)   ;
\draw  [fill={rgb, 255:red, 155; green, 155; blue, 155 }  ,fill opacity=1 ] (455.09,195.42) -- (470.37,208.45) -- (457.45,223.61) -- (442.17,210.58) -- cycle ;
\draw [color={rgb, 255:red, 0; green, 0; blue, 0 }  ,draw opacity=1 ]   (384.3,266.45) -- (455.95,209.81) ;
\draw [shift={(458.3,207.95)}, rotate = 141.67] [fill={rgb, 255:red, 0; green, 0; blue, 0 }  ,fill opacity=1 ][line width=0.08]  [draw opacity=0] (8.93,-4.29) -- (0,0) -- (8.93,4.29) -- cycle    ;
\draw [line width=2.25]    (363.45,117.95) -- (531.46,259.98) ;
\draw  [color={rgb, 255:red, 0; green, 0; blue, 0 }  ,draw opacity=0 ][fill={rgb, 255:red, 155; green, 155; blue, 155 }  ,fill opacity=0.33 ] (79.82,227.69) -- (162.76,121.37) -- (187.29,140.51) -- (104.35,246.83) -- cycle ;

\draw (150.06,73.21) node [anchor=north west][inner sep=0.75pt]  [rotate=-0.03]  {$\phi _{2} r^{2/3}$};
\draw (47.04,99.25) node [anchor=north west][inner sep=0.75pt]  [rotate=-0.03]  {$U_{1}$};
\draw (192.47,264.84) node [anchor=north west][inner sep=0.75pt]  [rotate=-0.03]  {$U_{2}$};
\draw (39.47,238.75) node [anchor=north west][inner sep=0.75pt]  [rotate=-0.03]  {$\theta r^{2/3}$};
\draw (167.99,208.31) node [anchor=north west][inner sep=0.75pt]  [rotate=-0.03]  {${U_{2}^{k}}$};
\draw (86.98,216.77) node [anchor=north west][inner sep=0.75pt]  [font=\footnotesize,rotate=-0.03]  {$( 0,0)$};
\draw (151.02,134.81) node [anchor=north west][inner sep=0.75pt]  [font=\footnotesize,rotate=-0.03]  {$( r,r)$};
\draw (187.99,185.83) node [anchor=north west][inner sep=0.75pt]  [rotate=-0.03]  {$\overline{U_{2}^{k}}$};
\draw (150.47,231.8) node [anchor=north west][inner sep=0.75pt]  [rotate=-0.03]  {$\underline{U_{2}^{k}}$};
\draw (574.37,226.11) node [anchor=north west][inner sep=0.75pt]    {$r/L^{30}$};
\draw (450.6,77.9) node [anchor=north west][inner sep=0.75pt]    {$r^{2/3} /L^{20}$};
\draw (500,142.4) node [anchor=north west][inner sep=0.75pt]    {$v$};
\draw (526.96,266.38) node [anchor=north west][inner sep=0.75pt]    {$u$};
\draw (355.3,267.1) node [anchor=north west][inner sep=0.75pt]    {$R(u,v)$};
\draw (329.3,101.1) node [anchor=north west][inner sep=0.75pt]    {$\mathcal{L}(u)$};

\end{tikzpicture}

\captionsetup{width=0.8\textwidth}
\caption{\textit{Left}: A demonstration of the region used to define the barrier event $\barevent$ in \eqref{bar_cap}, {the gray rectangle above is $R^{\theta r^{2/3}}_{0, r}$.} \textit{Right}: construction used in the event $A$ in \eqref{define_A}. } \label{cutU}

\end{center}

\end{figure}

This section defines a barrier event $\barevent \in \mathcal{F}_\theta$ and investigates consequences of  conditioning on it. 
Fix the parameters $0<\theta<1/2, \phi_1 = \theta^{-10}, \phi_2=\theta^{-100}$, and $L = \theta^{-1000}.$ We have the freedom to decrease   $\theta$ if necessary, so that $0 < \theta \leq \theta_0$.

Next, we define a barrier event around the rectangle $R^{\theta r^{2/3}}_{0, r}$. This part of the construction is illustrated on the left of Figure \ref{cutU}. The region $R^{\phi_2 r^{2/3}}_{0, r} \setminus R^{\theta r^{2/3}}_{0, r}$ is formed by two disjoint rectangles,   $U_1$  above the diagonal and $U_2$  below the diagonal.

The anti-diagonal lines $\{\mathcal{L}_{kr/L}\}_{k=1}^{L-1}$ cut each of $U_1$ and $U_2$ into $L$ small rectangles. If we fix $r$ sufficiently large depending on $L$, these rectangles are not degenerate.  Denote these small rectangles by $U^k_i$  for $i = 1,2$ and  $k = 1,\dots, L$.
Let $\underline U_i^k$ and $\overline U_i^k$ denote the top and bottom sides (with slope $-1$) of the rectangle $U_i^k$. We define the event 
\begin{equation}\label{bar_cap}
\barevent = \bigcap_{i=1}^2 \bigcap_{k=1}^L\Big\{\log Z^{\textup{in}, U_i}_{\underline{U_i^k}, \overline{U_i^k}}   - 2(r/L){f_d} \leq - Lr^{1/3}\Big\}.
\end{equation}

\begin{lemma}\label{bar_lb}
There exists a positive constant $\theta_0$ such that for each $0<\theta \leq \theta_0$, there exists a positive constant $c_0$ depending on $\theta$ such that for each  $r\ge c_0$, we have
$$\mathbb{P}(\barevent) \geq e^{-e^{L^{100}}}.$$
\end{lemma}

\begin{proof}
Note $\barevent$ is the intersection of $2L$ events, of equal probability by translation invariance. {By the FKG inequality}, it therefore suffices to lower bound the probability 
\begin{equation}\label{one_term_lb}
\mathbb{P}\Big(\log Z^{\textup{in}, U_1}_{\underline{U_1^1}, \overline{U_1^1}}   - 2(r/L){f_d} \leq - Lr^{1/3}\Big).
\end{equation}

The following construction is illustrated on the right of Figure \ref{cutU}. Using diagonal and anti-diagonal lines, we cut the rectangle $U^1_1$ into smaller rectangles whose diagonal $\ell^\infty$-length is $\tfrac{r}{L^{30}}$ and anti-diagonal $\ell^\infty$-length $\tfrac{r^{2/3}}{L^{20}}$. Then, the number of rectangles in this grid (see the right of Figure \ref{cutU}) along the diagonal will be $L^{29}$. And the number of rectangles in the anti-diagonal direction is no more than $L^{21}$. Let us use $R(u, v)$ to enumerate these small rectangles, where the index $u = 1, 2, \dots L^{29}$ records the position along the diagonal direction, and $v =  1, 2, \dots, v_L \leq L^{21} $ enumerates the small rectangles along the each anti-diagonal line. Let us also use $\mathcal{L}(u)$ to denote the anti-diagonal line which contains the upper anti-diagonal side of ${R(u, v)}$, and let $\underline {R(u, v)}$ to denote the lower anti-diagonal side of ${R(u, v)}$.

Let $D$ be the small constant \oldc{itl_bound_c1} from Proposition \ref{itl_bound}, and we define the event
\begin{equation}\label{define_A}
A =  \bigcap_{u, v} \Big\{\log Z_{\underline {R(u, v)}, \mathcal{L}(u)} -  2(r/L^{30})f_d \leq -D (r/L^{30})^{1/3}  \Big\}.
\end{equation}
Using the FKG inequality and Proposition \ref{itl_bound}, we have $\mathbb{P}(A) \geq {e}^{-e^{L^{99}}}$. 

Next, the constrained free energy can be upper bounded {using the maximum bound introduced in Section \ref{max_bd}}
\begin{align*}
\log Z^{\textup{in}, U_1}_{\underline{U_1^1}, \overline{U_1^1}} & \leq \sum_{u=1}^{L^{29}}  \Big( 100 \log L + \max_{v} \log Z_{\underline {R(u, v)}, \mathcal{L}(u)} \Big).\\
& \leq L^{29} \Big( 100 \log L + \max_{u, v} \log Z_{\underline {R(u, v)}, \mathcal{L}(u)} \Big)\\
\text{ restrict to the event $A$} \quad & \leq L^{29} \Big( 100 \log L + 2(r/L^{30})f_d - D r^{1/3}/L^{10}  \Big)\\
& \leq 2(r/L)f_d - DL^{19} r^{1/3} + L^{30} \\
& \leq  2(r/L)f_d - Lr^{1/3}
\end{align*}
provided that $\theta_0$ is sufficiently small (which makes $L$ large) and then increase $c_0$. With this, { we have shown that on $A$, the event in \eqref{one_term_lb} holds}, thus
$$\eqref{one_term_lb} \geq e^{-e^{L^{100}}}.$$
and finished the proof of the lemma by the FKG inequality.

\end{proof}

\subsection{Concentration of the free energy between $(0,0)$ and  $\mathcal{L}_r$}
Our goal in this section is to show that when conditioned on $\barevent$ the free energy $\log Z_{0, \mathcal{L}_r}$ is concentrated on paths that go from $(0,0)$ to $\mathcal{L}_r^{r^{2/3}}$ and are contained between the diagonal sides of the rectangle $\mathcal{R}_{0, r-r/L}^{3\theta r^{2/3}}$. This is stated in Proposition \ref{loc1} at the end of this subsection.

Before stating Proposition \ref{loc1}, we  define our high probability events. To start, split the collection of paths from $(0,0)$ to $\mathcal{L}_r$ as follows. First, let 
\begin{align*}
\mathfrak{A} &= \textup{all paths from $(0,0)$ to $\mathcal{L}_r^{\phi_1 r^{2/3}}$ that stay inside $R_{0,r}^{\phi_2 r^{2/3}}$}.\\
\mathfrak{B} & = \text{all other paths from $(0,0)$ to $\mathcal{L}_r$}.
\end{align*}
Then among $\mathfrak{A}$, let us further split $\mathfrak{A} = \mathfrak{A}_1 \cup \mathfrak{A}_2 \cup \mathfrak{A}_3 \cup \mathfrak{A}_4$ where 
\begin{align}\label{3As}
\begin{split}
\mathfrak{A}_1 &= \textup{paths from $(0,0)$ to $\mathcal{L}_r^{r^{2/3}}$ that stay between the diagonal sides of $R_{0,r-r/L}^{3\theta r^{2/3}}$ and }\\
& \qquad \qquad \qquad\textup{touch each $R_{ir/L,(i+1)r/L}^{\theta r^{2/3}}$ for $i = 0,1,\dots, L-2$},\\
\mathfrak{A}_2&=\textup{paths that avoid at least one of $R_{ir/L,(i+1)r/L}^{\theta r^{2/3}}$ completely for $i = 0,1,\dots, L-2$}.\\
\mathfrak{A}_3 &=\textup{paths that exit from the diagonal sides of $R_{0,r}^{3\theta r^{2/3}}$ and }\\
& \qquad \qquad \qquad \textup{intersect $R_{ir/L,(i+1)r/L}^{\theta r^{2/3}}$ for all $i = 0,1,\dots, L-2$}.\\
\mathfrak{A}_4 &= \textup{paths from $(0,0)$ to $\mathcal{L}_r^{\phi_1 r^{2/3}} \setminus \mathcal{L}_r^{r^{2/3}}$ that stay between the diagonal sides of $R_{0,r-r/L}^{3\theta r^{2/3}}$ and }\\
& \qquad \qquad \qquad\textup{touch each $R_{ir/L,(i+1)r/L}^{\theta r^{2/3}}$ for $i = 0,1,\dots, L-2$},\\
\end{split}
\end{align}
And among $\mathfrak{B}$, we write $\mathfrak{B} = \mathfrak{A}_5 \cup\mathfrak{A}_6$ where 
\begin{align*}
\mathfrak{A}_5 &= \textup{all paths from $(0,0)$ to $\mathcal{L}_{r} \setminus \mathcal{L}_r^{\phi_1 r^{2/3}}$}\\
\mathfrak{A}_6 &= \text{all paths from $(0,0)$ to $\mathcal{L}_r^{\phi_1 r^{2/3}}$ that exit $R_{0,r}^{\phi_2 r^{2/3}}$}
\end{align*}

Let $\log Z_{0, \mathcal{L}_r}(\mathfrak{A}_i)$ denote the free energy when we sum over only the paths inside $\mathfrak{A}_i$, and define the following events:
\begin{align*}
\mathcal{A}_1 &= \big\{\log Z^{\text{in}, \theta r^{2/3}}_{0, r}   - 2r{f_d} \geq - \theta^{-5}r^{1/3}\big\},\\
\mathcal{A}_2 &= \big\{\log Z_{0, \mathcal{L}_r}(\mathfrak{A}_2)  -  2r{f_d} \leq -  \theta^{-900}  r^{1/3}\big\}\\
\mathcal{A}_3 & = \big\{\log Z_{0, \mathcal{L}_r}(\mathfrak{A}_3)  - 2r{f_d} \leq -  \theta^{-900}  r^{1/3}\big\}.\\
\mathcal{A}_4 & = \big\{\log Z_{0, \mathcal{L}_r}(\mathfrak{A}_4)  - 2r{f_d} \leq -  \theta^{-900}  r^{1/3}\big\}.\\
\mathcal{A}_5 &= \big\{\log Z_{0, \mathcal{L}_r}(\mathfrak{A}_5) - 2rf_d \leq -\theta^{-10}r^{1/3} \big\}\\
\mathcal{A}_6 &= \big\{\log Z_{0, \mathcal{L}_r}(\mathfrak{A}_6) - 2rf_d \leq -\theta^{-100}r^{1/3} \big\}
\end{align*}

Next, we show that all six events are likely events. 
\begin{proposition}\label{lem_Ai}
For $i = 1, \dots, 6$. There exists a positive constant $\theta_0$ such that for each $0< \theta \leq \theta_0$, there exists a positive constant $c_0$ such that for each $c_0 \leq r$, 
$$\mathbb{P}(\mathcal{A}_i|\barevent) \geq 1- e^{-\theta^{-2}}.$$
\end{proposition}
 To prove Proposition \ref{lem_Ai}, we will split it into six separated lemmas, according to  $i = 1, \dots, 6.$

\begin{lemma}\label{lem_A1}
There exists a positive constant $\theta_0$ such that for each $0< \theta \leq \theta_0$, there exists a positive constant $c_0$ such that for each $c_0 \leq r$, 
$$\mathbb{P}(\mathcal{A}_1|\barevent) \geq 1- e^{-\theta^{-2}}.$$
\end{lemma}
\begin{proof}
We upper bound $\mathbb{P}(\mathcal{A}_1^c|\barevent)$. By independence and  Theorem \ref{wide_similar},
$$
\mathbb{P}(\mathcal{A}_1^c|\barevent)
 = \mathbb{P}\Big(\log Z^{\textup{in},  {\theta r^{2/3}}}_{0, r}   - 2r{f_d} \leq - \theta^{-5}r^{1/3} \Big) \leq e^{-\theta^{-2}}.
$$
\end{proof}

\begin{lemma}\label{lem_A2}
There exists a positive constant $\theta_0$ such that for each $0< \theta \leq \theta_0$, there exists a positive constant $c_0$ such that for each $c_0 \leq r$, 
$$\mathbb{P}(\mathcal{A}_2| \barevent ) \geq 1-e^{-\theta^{-10}}.$$
\end{lemma}

\begin{proof}
Let us further rewrite $\mathfrak{A}_2$ as a non-disjoint union of paths $\bigcup_{k=1}^{L-2}\mathfrak{A}_2^{k, +} \cup \mathfrak{A}_2^{k, -}$ where $\mathfrak{A}_2^{k, +}$ and $\mathfrak{A}_2^{k, -}$  are the collections of paths which avoid $R_{kr/L, (k+1)r/L}^{\theta r^{2/3}}$ by going above or below. For simplicity of the notation, let 
$$U_k^+ = \mathcal{L}_{(\frac{kr}{L} -  \frac{\phi_2 + \theta}{2} r^{2/3}, \frac{kr}{L} +  \frac{\phi_2 + \theta}{2} r^{2/3})}^{\frac{\phi_2 - \theta}{2}r^{2/3}}\quad \textup{ and } \quad U_k^- = \mathcal{L}_{(\frac{kr}{L} +  \frac{\phi_2 + \theta}{2} r^{2/3}, \frac{kr}{L} -   \frac{\phi_2 +\theta}{2} r^{2/3})}^{\frac{\phi_2 - \theta}{2}r^{2/3}}.$$
For each fixed $k\in [\![1, L-2]\!]$ and $\Box \in \{+, -\}$, we have the upper bound 
\begin{equation}\label{1_est}
\log Z_{0, \mathcal{L}_r}(\mathfrak{A}_2^{k,\Box})  \leq \log Z_{0, \mathcal{L}_{kr/L}}  + \log Z_{U^\Box_k, U^\Box_{k+1}}  + \log Z_{\mathcal{L}_{(k+1)r/L}, \mathcal{L}_{r}}. 
\end{equation}
Since we are conditioned on the event $\barevent$, then 
\begin{equation}\label{2_est}
\log Z_{U^\Box_k, U^\Box_{k+1}}  - 2(r/L){f_d} \leq - Lr^{1/3}.
\end{equation}
{Since the free energy is increasing in the environment while $\barevent$ decreases the environment}, using the FKG inequality and the interval to line estimate from Theorem \ref{ptl_upper}, we have 
\begin{align}
\begin{split}\label{3_ptl_est}
&\mathbb{P}\Big(\log Z_{0, \mathcal{L}_{kr/L}}  - 2(kr/L){f_d} \geq \sqrt{L} r^{1/3} \Big| \barevent\Big)\\
&\qquad\qquad \qquad \leq \mathbb{P}\Big(\log Z_{0, \mathcal{L}_{kr/L}}  - 2(kr/L){f_d} \geq \sqrt{L} r^{1/3} \Big)\leq e^{-\theta^{-100}}\\
&\mathbb{P}\Big(\log Z_{ \mathcal{L}_{(k+1)r/L},\mathcal{L}_r} - 2((L-k-1)r/L){f_d} \geq \sqrt{L} r^{1/3} \Big| \barevent\Big)\\
&\qquad\qquad \qquad \leq \mathbb{P}\Big(\log Z_{ \mathcal{L}_{(k+1)r/L}, \mathcal{L}_r}  - 2((L-k-1)r/L){f_d} \geq \sqrt{L} r^{1/3}\Big) \leq e^{-\theta^{-100}}.
\end{split}
\end{align}
From \eqref{1_est}, \eqref{2_est}, \eqref{3_ptl_est}, We obtain the following estimate for each fixed $k = 1,2, \dots, L-2$
\begin{equation}\label{fix_k_est}
\mathbb{P}\Big(\log Z_{0, \mathcal{L}_r}(\mathfrak{A}_2^{k,\Box})  - 2r{f_d} \geq - \tfrac{L}{2} r^{1/3}\Big| \barevent\Big) \leq e^{-\theta^{-90}}.
\end{equation}
Using \eqref{fix_k_est}, a union bound will give us the desired result from our lemma, 
\begin{align*}
&\mathbb{P}\Big(\mathcal{A}_2^c\Big| \barevent \Big) \\
& \leq \mathbb{P}\Big(\Big\{\log\Big(\sum_{k=1}^{L-2} Z_{0, \mathcal{L}_r}(\mathfrak{A}_2^{k, +}) + Z_{0, \mathcal{L}_r}(\mathfrak{A}_2^{k, -}) \Big)-  2r{f_d} \geq - \tfrac{L}{10} r^{1/3}\Big\} \Big| \barevent \Big) \\
& \leq \mathbb{P}\Big(\Big\{\max_{\substack{k \in [\![1, L-2]\!]\\ \Box \in\{ +, -\}}}\log Z_{0, \mathcal{L}_r}(\mathfrak{A}_2^{k, \Box})  + 2 \log  L-   2r{f_d} \geq - \tfrac{L}{10} r^{1/3}  \Big\} \Big| \barevent \Big) \\
& \leq \sum_{k=1}^{L-2} \sum_{\Box \in \{+, -\}}\mathbb{P}\Big(\Big\{\log  Z_{0, \mathcal{L}_r}(\mathfrak{A}_2^{k, \Box})  -  2r{f_d} \geq -\tfrac{L}{9}r^{1/3}\Big\} \Big| \barevent \Big) \\
&\leq e^{-\theta^{-10}} 
\end{align*}
where the last inequality comes from \eqref{fix_k_est} and provided $\theta_0$ is sufficiently small.
\end{proof}

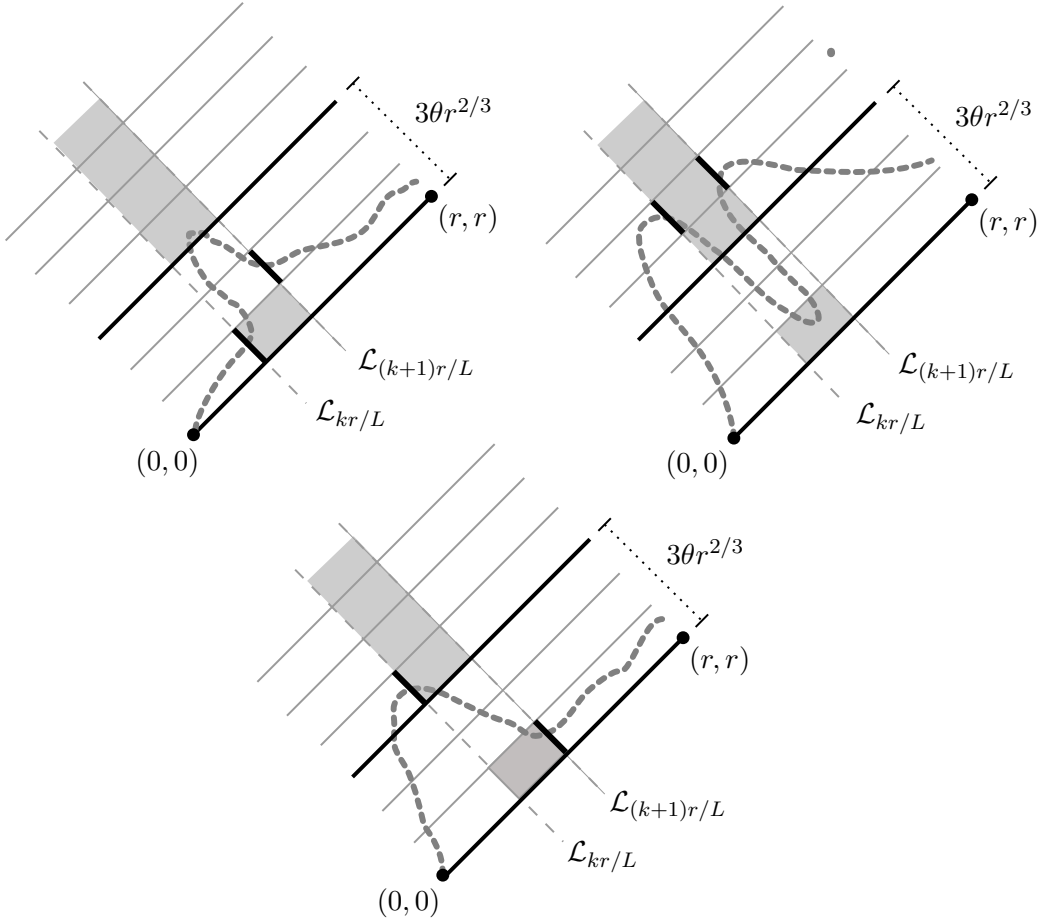
\begin{figure}
\begin{center}

\tikzset{every picture/.style={line width=0.75pt}} 

\begin{tikzpicture}[x=0.75pt,y=0.75pt,yscale=-0.8,xscale=0.8]

\draw [color={rgb, 255:red, 155; green, 155; blue, 155 }  ,draw opacity=1 ]   (118.76,261.59) -- (268.36,111.59) ;
\draw [color={rgb, 255:red, 155; green, 155; blue, 155 }  ,draw opacity=1 ] [dash pattern={on 4.5pt off 4.5pt}]  (41.16,90.92) -- (209.4,262.78) ;
\draw [color={rgb, 255:red, 155; green, 155; blue, 155 }  ,draw opacity=1 ][fill={rgb, 255:red, 155; green, 155; blue, 155 }  ,fill opacity=1 ] [dash pattern={on 4.5pt off 4.5pt}]  (68.96,60.12) -- (234.4,229.78) ;
\draw [color={rgb, 255:red, 155; green, 155; blue, 155 }  ,draw opacity=1 ]   (98.76,241.59) -- (248.36,91.59) ;
\draw  [dash pattern={on 0.84pt off 2.51pt}]  (238.87,60.87) -- (300.2,121.53) ;
\draw [shift={(300.2,121.53)}, rotate = 224.69] [color={rgb, 255:red, 0; green, 0; blue, 0 }  ][line width=0.75]    (0,5.59) -- (0,-5.59)   ;
\draw [shift={(238.87,60.87)}, rotate = 224.69] [color={rgb, 255:red, 0; green, 0; blue, 0 }  ][line width=0.75]    (0,5.59) -- (0,-5.59)   ;
\draw [color={rgb, 255:red, 155; green, 155; blue, 155 }  ,draw opacity=1 ]   (59.96,200.12) -- (209.56,50.12) ;
\draw [color={rgb, 255:red, 155; green, 155; blue, 155 }  ,draw opacity=1 ]   (38.63,181.45) -- (188.23,31.45) ;
\draw [color={rgb, 255:red, 155; green, 155; blue, 155 }  ,draw opacity=1 ]   (19.96,160.12) -- (169.56,10.12) ;
\draw  [draw opacity=0][fill={rgb, 255:red, 155; green, 155; blue, 155 }  ,fill opacity=0.5 ] (192.74,187.4) -- (213.32,208.14) -- (184.22,237.02) -- (163.64,216.27) -- cycle ;
\draw [line width=1.5]    (288.36,132.25) -- (253.02,167.69) -- (138.36,282.65) ;
\draw [color={rgb, 255:red, 128; green, 128; blue, 128 }  ,draw opacity=1 ][line width=2.25] [line join = round][line cap = round] [dash pattern={on 2.53pt off 3.02pt}]  (139.53,280.2) .. controls (139.53,269.45) and (156.62,242.45) .. (163.53,235.53) .. controls (167.66,231.4) and (174.3,224.09) .. (174.87,219.53) .. controls (175.6,213.7) and (169.92,205.26) .. (165.53,200.87) .. controls (164.71,200.05) and (159.93,198.74) .. (159.53,197.53) .. controls (158.44,194.26) and (150.38,190.85) .. (147.53,186.87) .. controls (140.95,177.65) and (134.89,168.42) .. (135.53,156.2) .. controls (135.57,155.5) and (136.83,155.6) .. (137.53,155.53) .. controls (139.97,155.31) and (142.46,155.1) .. (144.87,155.53) .. controls (145.29,155.61) and (150.39,160.58) .. (151.53,160.87) .. controls (165.18,164.28) and (170.65,178.03) .. (186.87,175.53) .. controls (195.22,174.25) and (199.23,168.39) .. (207.53,166.2) .. controls (220.73,162.73) and (233.89,162.23) .. (244.2,154.87) .. controls (255.5,146.79) and (255.46,130.97) .. (267.8,129) .. controls (273.08,128.15) and (273.47,123.44) .. (278.8,123) ;
\draw  [fill={rgb, 255:red, 0; green, 0; blue, 0 }  ,fill opacity=1 ] (134.93,282.65) .. controls (134.93,280.76) and (136.46,279.22) .. (138.36,279.22) .. controls (140.26,279.22) and (141.79,280.76) .. (141.79,282.65) .. controls (141.79,284.55) and (140.26,286.09) .. (138.36,286.09) .. controls (136.46,286.09) and (134.93,284.55) .. (134.93,282.65) -- cycle ;
\draw  [fill={rgb, 255:red, 0; green, 0; blue, 0 }  ,fill opacity=1 ] (284.93,132.25) .. controls (284.93,130.36) and (286.46,128.82) .. (288.36,128.82) .. controls (290.26,128.82) and (291.79,130.36) .. (291.79,132.25) .. controls (291.79,134.15) and (290.26,135.69) .. (288.36,135.69) .. controls (286.46,135.69) and (284.93,134.15) .. (284.93,132.25) -- cycle ;
\draw [line width=2.25]    (163.64,216.27) -- (184.22,237.02) ;
\draw [color={rgb, 255:red, 0; green, 0; blue, 0 }  ,draw opacity=1 ][line width=2.25]    (174,166.95) -- (193.56,186.59) ;
\draw [color={rgb, 255:red, 155; green, 155; blue, 155 }  ,draw opacity=1 ]   (459.43,263.72) -- (609.03,113.72) ;
\draw [color={rgb, 255:red, 155; green, 155; blue, 155 }  ,draw opacity=1 ] [dash pattern={on 4.5pt off 4.5pt}]  (384.5,90.93) -- (549,262.84) ;
\draw [color={rgb, 255:red, 155; green, 155; blue, 155 }  ,draw opacity=1 ][fill={rgb, 255:red, 155; green, 155; blue, 155 }  ,fill opacity=1 ] [dash pattern={on 4.5pt off 4.5pt}]  (412.5,61.43) -- (577,231.93) ;
\draw [color={rgb, 255:red, 155; green, 155; blue, 155 }  ,draw opacity=1 ]   (439.43,243.72) -- (589.03,93.72) ;
\draw  [fill={rgb, 255:red, 0; green, 0; blue, 0 }  ,fill opacity=1 ] (625.59,134.39) .. controls (625.59,132.49) and (627.13,130.95) .. (629.03,130.95) .. controls (630.92,130.95) and (632.46,132.49) .. (632.46,134.39) .. controls (632.46,136.28) and (630.92,137.82) .. (629.03,137.82) .. controls (627.13,137.82) and (625.59,136.28) .. (625.59,134.39) -- cycle ;
\draw  [dash pattern={on 0.84pt off 2.51pt}]  (579.53,63) -- (640.87,123.67) ;
\draw [shift={(640.87,123.67)}, rotate = 224.69] [color={rgb, 255:red, 0; green, 0; blue, 0 }  ][line width=0.75]    (0,5.59) -- (0,-5.59)   ;
\draw [shift={(579.53,63)}, rotate = 224.69] [color={rgb, 255:red, 0; green, 0; blue, 0 }  ][line width=0.75]    (0,5.59) -- (0,-5.59)   ;
\draw [color={rgb, 255:red, 155; green, 155; blue, 155 }  ,draw opacity=1 ]   (400.63,202.25) -- (550.23,52.25) ;
\draw [color={rgb, 255:red, 155; green, 155; blue, 155 }  ,draw opacity=1 ]   (379.29,183.59) -- (528.89,33.59) ;
\draw [color={rgb, 255:red, 155; green, 155; blue, 155 }  ,draw opacity=1 ]   (360.63,162.25) -- (510.23,12.25) ;
\draw  [draw opacity=0][fill={rgb, 255:red, 155; green, 155; blue, 155 }  ,fill opacity=0.5 ] (534.59,188.36) -- (553.97,207.88) -- (524.68,236.93) -- (505.31,217.4) -- cycle ;
\draw [line width=1.5]    (629.03,134.39) -- (479.03,284.79) ;
\draw [color={rgb, 255:red, 155; green, 155; blue, 155 }  ,draw opacity=1 ]   (277.43,540.05) -- (427.03,390.05) ;
\draw [color={rgb, 255:red, 155; green, 155; blue, 155 }  ,draw opacity=1 ] [dash pattern={on 4.5pt off 4.5pt}]  (202.44,367.79) -- (371.43,539.38) ;
\draw [color={rgb, 255:red, 155; green, 155; blue, 155 }  ,draw opacity=1 ][fill={rgb, 255:red, 155; green, 155; blue, 155 }  ,fill opacity=1 ] [dash pattern={on 4.5pt off 4.5pt}]  (230.04,338.59) -- (397.64,509.39) ;
\draw [color={rgb, 255:red, 155; green, 155; blue, 155 }  ,draw opacity=1 ]   (257.43,520.05) -- (407.03,370.05) ;
\draw  [fill={rgb, 255:red, 0; green, 0; blue, 0 }  ,fill opacity=1 ] (443.59,410.72) .. controls (443.59,408.82) and (445.13,407.28) .. (447.03,407.28) .. controls (448.92,407.28) and (450.46,408.82) .. (450.46,410.72) .. controls (450.46,412.61) and (448.92,414.15) .. (447.03,414.15) .. controls (445.13,414.15) and (443.59,412.61) .. (443.59,410.72) -- cycle ;
\draw  [dash pattern={on 0.84pt off 2.51pt}]  (397.53,339.33) -- (458.87,400) ;
\draw [shift={(458.87,400)}, rotate = 224.69] [color={rgb, 255:red, 0; green, 0; blue, 0 }  ][line width=0.75]    (0,5.59) -- (0,-5.59)   ;
\draw [shift={(397.53,339.33)}, rotate = 224.69] [color={rgb, 255:red, 0; green, 0; blue, 0 }  ][line width=0.75]    (0,5.59) -- (0,-5.59)   ;
\draw [color={rgb, 255:red, 155; green, 155; blue, 155 }  ,draw opacity=1 ]   (218.63,478.58) -- (368.23,328.58) ;
\draw [color={rgb, 255:red, 155; green, 155; blue, 155 }  ,draw opacity=1 ]   (197.29,459.92) -- (346.89,309.92) ;
\draw [color={rgb, 255:red, 155; green, 155; blue, 155 }  ,draw opacity=1 ]   (178.63,438.58) -- (328.23,288.58) ;
\draw  [color={rgb, 255:red, 155; green, 155; blue, 155 }  ,draw opacity=1 ][fill={rgb, 255:red, 194; green, 190; blue, 190 }  ,fill opacity=1 ] (353.79,463.5) -- (373.16,483.03) -- (343.7,512.25) -- (324.33,492.73) -- cycle ;
\draw [line width=1.5]    (447.03,410.72) -- (297.03,561.12) ;
\draw [color={rgb, 255:red, 128; green, 128; blue, 128 }  ,draw opacity=1 ][line width=2.25] [line join = round][line cap = round] [dash pattern={on 2.53pt off 3.02pt}]  (295.53,560.53) .. controls (295.53,550.49) and (292.82,543.56) .. (288.2,535.86) .. controls (287.13,534.07) and (283.7,529.27) .. (283.53,528.53) .. controls (280.74,516.4) and (278.4,507.02) .. (274.2,496.53) .. controls (271.29,489.26) and (265.18,483.7) .. (264.87,477.2) .. controls (264.64,472.53) and (265.09,467.86) .. (264.87,463.2) .. controls (264.86,462.95) and (264.01,455.77) .. (264.87,453.2) .. controls (269.39,439.63) and (285.75,440.97) .. (298.2,447.2) .. controls (309.6,452.9) and (324.35,460.9) .. (336.2,463.86) .. controls (343.51,465.69) and (347.37,473.7) .. (355.53,472.53) .. controls (371.97,470.18) and (377.37,454.03) .. (386.2,445.2) .. controls (390.85,440.55) and (397.77,435.41) .. (403.53,432.53) .. controls (405.14,431.72) and (408.66,432.07) .. (410.2,430.53) .. controls (417.12,423.61) and (424.27,399) .. (433.8,399) ;
\draw [color={rgb, 255:red, 0; green, 0; blue, 0 }  ,draw opacity=1 ][line width=2.25]    (353.6,463.39) -- (373.16,483.03) ;
\draw  [fill={rgb, 255:red, 0; green, 0; blue, 0 }  ,fill opacity=1 ] (292.1,560.53) .. controls (292.1,558.63) and (293.64,557.1) .. (295.53,557.1) .. controls (297.43,557.1) and (298.97,558.63) .. (298.97,560.53) .. controls (298.97,562.43) and (297.43,563.96) .. (295.53,563.96) .. controls (293.64,563.96) and (292.1,562.43) .. (292.1,560.53) -- cycle ;
\draw  [color={rgb, 255:red, 128; green, 128; blue, 128 }  ,draw opacity=1 ][line width=3] [line join = round][line cap = round] (540.3,42) .. controls (540.3,41.67) and (540.3,41.33) .. (540.3,41) ;
\draw [color={rgb, 255:red, 128; green, 128; blue, 128 }  ,draw opacity=1 ][line width=2.25] [line join = round][line cap = round] [dash pattern={on 2.53pt off 3.02pt}]  (479.36,284.79) .. controls (476.42,264.23) and (468.31,244.59) .. (455.76,227.6) .. controls (446.59,215.17) and (432.84,202.11) .. (427.36,187.59) .. controls (425.72,183.26) and (414.66,157.82) .. (421.36,150.79) .. controls (437.01,134.4) and (456.25,159.26) .. (469.36,167.6) .. controls (484.13,179.54) and (495.31,195.1) .. (511.47,204.21) .. controls (516.69,207.15) and (525.18,213.77) .. (530.74,211.55) .. controls (536.2,209.37) and (527.47,196.88) .. (524.13,192.88) .. controls (516.8,186.21) and (507.71,178.15) .. (500.8,168.88) .. controls (487.97,157.41) and (465.8,140.33) .. (470.56,119.2) .. controls (474.47,101.8) and (513.79,115.39) .. (518.56,115.59) .. controls (547,116.83) and (577.1,119.59) .. (603.76,109.6) ;
\draw  [fill={rgb, 255:red, 0; green, 0; blue, 0 }  ,fill opacity=1 ] (475.59,284.79) .. controls (475.59,282.89) and (477.13,281.35) .. (479.03,281.35) .. controls (480.92,281.35) and (482.46,282.89) .. (482.46,284.79) .. controls (482.46,286.68) and (480.92,288.22) .. (479.03,288.22) .. controls (477.13,288.22) and (475.59,286.68) .. (475.59,284.79) -- cycle ;
\draw  [draw opacity=0][fill={rgb, 255:red, 155; green, 155; blue, 155 }  ,fill opacity=0.5 ] (79.42,70.25) -- (153.55,147.17) -- (124.27,175.39) -- (50.14,98.47) -- cycle ;
\draw [color={rgb, 255:red, 0; green, 0; blue, 0 }  ,draw opacity=1 ][line width=1.5]    (78.75,222.17) -- (228.35,72.17) ;
\draw  [draw opacity=0][fill={rgb, 255:red, 155; green, 155; blue, 155 }  ,fill opacity=0.5 ] (422.39,71.47) -- (495.43,147.25) -- (465.94,175.67) -- (392.91,99.88) -- cycle ;
\draw [color={rgb, 255:red, 0; green, 0; blue, 0 }  ,draw opacity=1 ][line width=2.25]    (455.83,107.65) -- (475.43,127.25) ;
\draw [color={rgb, 255:red, 0; green, 0; blue, 0 }  ,draw opacity=1 ][line width=2.25]    (427.6,135.73) -- (447.16,155.37) ;
\draw [color={rgb, 255:red, 0; green, 0; blue, 0 }  ,draw opacity=1 ][line width=1.5]    (418.78,222.96) -- (513.11,128.37) -- (568.38,72.96) ;
\draw  [draw opacity=0][fill={rgb, 255:red, 155; green, 155; blue, 155 }  ,fill opacity=0.5 ] (239.3,346.67) -- (313.43,423.58) -- (284.15,451.8) -- (210.02,374.88) -- cycle ;
\draw [color={rgb, 255:red, 0; green, 0; blue, 0 }  ,draw opacity=1 ][line width=2.25]    (264.55,432.2) -- (284.15,451.8) ;
\draw [color={rgb, 255:red, 0; green, 0; blue, 0 }  ,draw opacity=1 ][line width=1.5]    (238.63,498.58) -- (388.23,348.58) ;

\draw (213.33,259.07) node [anchor=north west][inner sep=0.75pt]    {$\mathcal{L}_{kr/L}$};
\draw (240,227.07) node [anchor=north west][inner sep=0.75pt]    {$\mathcal{L}_{( k+1) r/L}$};
\draw (100,289.27) node [anchor=north west][inner sep=0.75pt]    {$( 0,0)$};
\draw (290.36,135.65) node [anchor=north west][inner sep=0.75pt]    {$( r,r)$};
\draw (276,67.07) node [anchor=north west][inner sep=0.75pt]    {$3\theta r^{2/3}$};
\draw (554,261.2) node [anchor=north west][inner sep=0.75pt]    {$\mathcal{L}_{kr/L}$};
\draw (580.67,229.2) node [anchor=north west][inner sep=0.75pt]    {$\mathcal{L}_{( k+1) r/L}$};
\draw (434,290.73) node [anchor=north west][inner sep=0.75pt]    {$( 0,0)$};
\draw (631.03,137.79) node [anchor=north west][inner sep=0.75pt]    {$( r,r)$};
\draw (616.67,69.2) node [anchor=north west][inner sep=0.75pt]    {$3\theta r^{2/3}$};
\draw (372,537.53) node [anchor=north west][inner sep=0.75pt]    {$\mathcal{L}_{kr/L}$};
\draw (398.67,505.53) node [anchor=north west][inner sep=0.75pt]    {$\mathcal{L}_{( k+1) r/L}$};
\draw (252,567.06) node [anchor=north west][inner sep=0.75pt]    {$( 0,0)$};
\draw (449.03,414.12) node [anchor=north west][inner sep=0.75pt]    {$( r,r)$};
\draw (434.67,345.53) node [anchor=north west][inner sep=0.75pt]    {$3\theta r^{2/3}$};

\end{tikzpicture}

\captionsetup{width=0.8\textwidth}
\caption{An illustration of the paths from the collection $\mathfrak{A}_3^{k, +}$. The paths in this collection must intersect the two gray rectangles shown in the picture. The top two paths cross the lines $\mathcal{L}_{kr/L}$ and $\mathcal{L}_{(k+1)r/L}$ at neighboring segments (the case $|i-j| \leq 1$ from the proof of Lemma \ref{lem_A3}). They must have a high transversal fluctuation between $\mathcal{L}_{kr/L}$ and $\mathcal{L}_{(k+1)r/L}$ because they have to reach the gray rectangles. The bottom picture is a path that crosses the lines $\mathcal{L}_{kr/L}$ and $\mathcal{L}_{(k+1)r/L}$ at non-neighboring positions (the case $|i-j| \geq 2$). This path has a high transversal fluctuation  between $\mathcal{L}_{kr/L}$ and $\mathcal{L}_{(k+1)r/L}$ because of these non-neighboring crossing positions.} \label{A3}
\end{center}
\end{figure}

\begin{lemma}\label{lem_A3}
There exists a positive constant $\theta_0$ such that for each $0< \theta \leq \theta_0$, there exists a positive constant $c_0$ such that for each $c_0 \leq r$, 
$$\mathbb{P}\Big(\mathcal{A}_3\Big| \barevent \Big) \geq 1- e^{-\theta^{-10}}.$$
\end{lemma}
\begin{proof}
As before, let us rewrite $\mathcal{A}_3$ as a non-disjoint union of paths $\bigcup_{k=0}^{L-2}\mathfrak{A}_3^{k, +} \cup \mathfrak{A}_3^{k, -}$ where $\mathfrak{A}_3^{k, +}$ and $\mathfrak{A}_3^{k, -}$  are the collections of paths which exit from the upper and lower diagonal sides of the rectangle $R_{kr/L,(k+1)r/L}^{3\theta r^{2/3}}$.

Let us fix $k$ and look at $\mathfrak{A}_3^{k, +}$. We will show that all paths in this collection must have high transversal fluctuations. This fact is illustrated in Figure \ref{A3}. First, we further break up this collection of paths by where they cross the lines $\mathcal{L}_{kr/L}$ and $\mathcal{L}_{(k+1)r/L}$. For $i,j \in \mathbb{Z}_{\geq 0}$, let 
\begin{align*}
I^i_{kr/L} &= \mathcal{L}_{(kr/L - (i+ \frac{1}{2})\theta r^{2/3},kr/L  + (i+ \frac{1}{2})\theta r^{2/3})}^{\frac{1}{2}\theta r^{2/3}}\\
J^{j}_{(k+1)r/L} &= \mathcal{L}_{((k+1)r/L - (j+\frac{1}{2})\theta r^{2/3},(k+1)r/L  + (j+\frac{1}{2})\theta r^{2/3})}^{\frac{1}{2}\theta r^{2/3}}.
\end{align*}
Then any path in $\mathfrak{A}_3^{k, +}$ must cross $I^i_{kr/L}$ and $J^{j}_{(k+1)r/L}$ for some $i, j \in [\![0, \phi_2 \theta^{-1}]\!]$. Thus we may rewrite $\mathfrak{A}_3^{k, +}$ as a non-disjoint union $\bigcup_{i,j = 0}^{\phi_2 \theta^{-1}} \mathfrak{A}_3^{k, +}(i,j)$ where $\mathfrak{A}_3^{k, +}(i,j)$ is the collection of paths inside $\mathfrak{A}_3^{k, +}$ that goes through $I^i_{kr/L}$ and $J^{j}_{(k+1)r/L}$. 
Next, we will split the case of $i$ and $j$ into two cases, when $|i-j|\leq 1$ or otherwise. 

By our assumption, the paths inside $\mathfrak{A}_3^{k, +}$ must intersect $R^{\theta r^{2/3}}_{kr/L, (k+1)r/L}$ while also exiting the upper diagonal side of $R^{3\theta r^{2/3}}_{kr/L, (k+1)r/L}$. If $|i-j|\leq 1$, there must be an unusually large transversal of size at least $\theta r^{2/3} = 
({\theta}{L^{2/3}}) (r/L)^{2/3}$ for the segment of the path $\mathfrak{A}_3^{k, +}(i, j)$ between $I^i_{kr/L}$ and $J^{j}_{(k+1)r/L}$. We may invoke Theorem \ref{trans_fluc_loss3} and obtain that for $|i-j| \leq 1$, 
\begin{equation}\label{ij_close}
\mathbb{P}\Big(\log Z_{I^i_{kr/L},J^{j}_{(k+1)r/L}}(\mathfrak{A}_3^{k, +}(i, j))  - 2(r/L){f_d} \geq - C(\theta L^{2/3})^2 (r/L)^{1/3}\Big) \leq e^{-\theta^{-100}}
\end{equation}
for some small constant $D$.

Next, when $|i-j| \geq 2$, then there is already a large transversal fluctuation of size at least $\theta r^{2/3}= ({\theta}{L^{2/3}}) (r/L)^{2/3}$ between for the segment of the path $\mathfrak{A}_3^{k, +}(i, j)$ between $I^i_{kr/L}$ and $J^{j}_{(k+1)r/L}$. By Proposition \ref{trans_fluc_loss2}, we obtain that for $|i-j|\geq 2$, 
\begin{equation}\label{ij_far}
\mathbb{P}\Big(\log Z_{I^i_{kr/L},J^{j}_{(k+1)r/L}}(\mathfrak{A}_3^{k, +}(i, j))  - 2(r/L){f_d} \geq -D(\theta L^{2/3})^2 (r/L)^{1/3}\Big) \leq e^{-\theta^{-100}}
\end{equation}
for some small constant $D$.

By the FKG inequality, \eqref{ij_close} and \eqref{ij_far} still hold if we replace the probability measure $\mathbb{P}$ with $\mathbb{P}(\boldsymbol{\cdot} | \barevent)$. 
Now as before, we upper bound the value of the free energy of the paths outside the region  between $\mathcal{L}_{kr/L}$ and $\mathcal{L}_{(k+1)r/L}$ by \eqref{3_ptl_est}. We obtain 
\begin{equation}\label{fix_kij_est}
\mathbb{P}\Big( \log Z_{0, \mathcal{L}_r}(\mathfrak{A}_3^{k,\Box}(i,j))  - 2r{f_d} \geq -  L^{0.99}   r^{1/3}  \Big| \barevent \Big) \leq e^{-\theta^{-90}}.
 \end{equation}
We have 
\begin{align*}
&\mathbb{P}\Big(\mathcal{A}_3^c\Big| \barevent \Big) \\
& \leq \mathbb{P}\Big(\Big\{\log\Big(\sum_{k=0}^{L-2} \sum_{\Box \in \{+, -\}} \sum_{i,j = 0}^{\phi_2 \theta^{-1}} Z_{0, \mathcal{L}_r}(\mathfrak{A}_3^{k, \Box}(i,j)) \Big)-  2r{f_d} \geq -  {L^{0.9}}  r^{1/3}\Big\} \Big| \barevent \Big) \\
& \leq \mathbb{P}\Big(\Big\{\max_{\substack{k \in [\![0, L-2]\!]\\ \Box \in\{ +, -\}\\i, j \in [\![0, \phi_2 \theta^{-1}]\!]}}\log  Z_{0, \mathcal{L}_r}(\mathfrak{A}_3^{k, \Box}(i,j))  + 100 \log  L-  2r{f_d} \geq -  L^{0.9}   r^{1/3}  \Big\}\Big| \barevent \Big) \\
& \leq \sum_{k=1}^{L-2} \sum_{\Box \in \{+, -\}} \sum_{i,j = 0}^{\phi_2 \theta^{-1}}\mathbb{P}\Big( \log Z_{0, \mathcal{L}_r}(\mathfrak{A}_2^{k, \Box})   -  2r{f_d} \geq - L^{0.95}  r^{1/3}  \Big| \barevent \Big) \\
& \leq  e^{-\theta^{-80}}
\end{align*}
where the last inequality uses  \eqref{fix_kij_est} and provided $\theta_0$ is sufficiently small.
\end{proof}

\begin{lemma}\label{lem_A4}
There exists a positive constant $\theta_0$ such that for each $0< \theta \leq \theta_0$, there exists positive constant $c_0$ such that for each $c_0 \leq r$, 
$$\mathbb{P}\Big(\mathcal{A}_4\Big| \barevent \Big) \geq 1- e^{-\theta^{-2}}.$$
\end{lemma}

\begin{proof}
This is simply because the last part of the paths has a very large transversal fluctuation, 
\begin{align}
&\mathbb{P}\Big(\log Z_{0, \mathcal{L}_r^{\phi_1 r^{2/3}} \setminus \mathcal{L}_r^{r^{2/3}}}(\mathfrak{A}_4) - 2r{f_d} \geq -L^{0.9} r^{1/3} \Big|\barevent \Big)\nonumber \\
& \leq \mathbb{P}\Big(\log Z_{0, \mathcal{L}_{r-r/L}} + \log Z_{\mathcal{L}_{r-r/L}^{3\theta r^{2/3}}, \mathcal{L}_r^{\phi_1 r^{2/3}} \setminus \mathcal{L}_r^{r^{2/3}}}  - 2r{f_d} \geq -L^{0.9} r^{1/3} \Big|\barevent \Big)\nonumber\\
& \leq \mathbb{P}\Big(\log Z_{0, \mathcal{L}_{r-r/L}}  - 2(r- r/L){f_d} \geq (L -L^{0.9}) r^{1/3} \Big|\barevent \Big)\label{longterm}\\
&\qquad  \qquad \qquad  +  \mathbb{P}\Big(\log Z_{\mathcal{L}_{r-r/L}^{3\theta r^{2/3}}, \mathcal{L}_r^{\phi_1 r^{2/3}} \setminus \mathcal{L}_r^{r^{2/3}}}  - 2(r/L){f_d} \geq -L  r^{1/3} \Big|\barevent \Big)\label{shortterm}.
\end{align}
By the FKG inequality and Theorem \ref{ptl_upper}, 
$$\eqref{longterm} \leq \mathbb{P}\Big(\log Z_{0, \mathcal{L}_{r-r/L}}  - 2(r- r/L){f_d} \geq (L -L^{0.9}) r^{1/3} \Big) \leq e^{-L^{0.1}}.$$
By the FKG inequality and Proposition \ref{trans_fluc_loss2}, 
$$\eqref{shortterm} \leq \mathbb{P}\Big(\log Z_{\mathcal{L}_{r-r/L}^{3\theta r^{2/3}}, \mathcal{L}_r^{\phi_1 r^{2/3}} \setminus \mathcal{L}_r^{r^{2/3}}} - 2(r/L){f_d} \geq -L  r^{1/3}  \Big) \leq e^{-L^{0.1}}.$$
With these, we have finished the proof of this theorem.

\end{proof}

\begin{lemma}\label{lem_A5}
There exists a positive constant $\theta_0$ such that for each $0< \theta \leq \theta_0$, there exists positive constant $c_0$ such that for each $c_0 \leq r$, 
$$\mathbb{P}\Big(\mathcal{A}_5\Big| \barevent \Big) \geq 1- e^{-\theta^{-2}}.$$
\end{lemma}

\begin{proof}
By the FKG inequality, it suffices to show  $\mathbb{P}(\mathcal{A}_5^c)\leq e^{-\theta^{-2}}.$ Then, this estimate follows directly from 
 Proposition \ref{trans_fluc_loss2}. 

\end{proof}

\begin{lemma}\label{lem_A6}
There exists a positive constant $\theta_0$ such that for each $0< \theta \leq \theta_0$, there exists positive constant $c_0$ such that for each $c_0 \leq r$, 
$$\mathbb{P}\Big(\mathcal{A}_6\Big| \barevent \Big) \geq 1- e^{-\theta^{-2}}.$$
\end{lemma}

\begin{proof}
By the FKG inequality, it suffices to show $\mathbb{P}(\mathcal{A}_6)\leq e^{-\theta^{-2}}.$ Then, this estimate follows directly from Theorem \ref{trans_fluc_loss3}.

\end{proof}

With these lemmas, we have shown Proposition \ref{lem_Ai}.
Finally, we have the following proposition.
\begin{proposition}\label{loc1}
 On the event $\cap_{i=1}^6 \mathcal{A}_i$, we have
$$
\log Z_{0, \mathcal{L}_r}(\mathfrak{A}_1)  \leq \log Z_{0, \mathcal{L}_r} \leq \log Z_{0, \mathcal{L}_r}(\mathfrak{A}_1)  + \log 6.
$$
\end{proposition}

\begin{proof}
This follows directly from the definition of our events $\mathcal{A}_i$ and the fact
$$
\max_{j}\{ \log Z_{0, \mathcal{L}_r}(\mathfrak{A}_j) \} \leq \log Z_{0, \mathcal{L}_r} \leq\max_{j}\{ \log Z_{0, \mathcal{L}_r}(\mathfrak{A}_j) \} + \log 6
$$
and $\log Z_{0, \mathcal{L}_r}(\mathfrak{A}_1)  \geq \log Z^{\textup{in},  {\theta r^{2/3}}}_{0, r}  $.
\end{proof}

\subsection{Concentration of the global free energy along $\mathcal{L}_r$}

Define ${\bf p}^*$ to be the maximizer in  
$$\max_{{\bf p}\in\mathcal{L}_r } \bigl\{ \log Z_{0, {\bf p}}  + \log Z_{{\bf p}, N}\bigr\}.$$
Our goal in this section is to show that when conditioned on $\barevent$, with high probability, ${\bf p}^* \in \mathcal{L}_r^{r^{2/3}}$. This is stated as Proposition \ref{loc2} at the end of this subsection.

Again, we start by defining our high probability events,
\begin{align*}
\mathcal{E}_1 &= \bigcap_{j=1}^{\phi_2} \Big\{\log Z^{\textup{max}}_{\mathcal{L}_r^{jr^{2/3}}, N}  - \log Z_{r,N}  \leq \theta^{-1}\sqrt{jr^{2/3}} \Big\}\\
\mathcal{E}_2 &= \Big\{\max_{{\bf p}\in\mathcal{L}_r \setminus\mathcal{L}_r^{\phi_1 r^{2/3}} } \log Z_{0, {\bf p}}  + \log Z_{{\bf p}, N}  \leq \log Z^{\textup{in} ,  {\theta r^{2/3}}}_{0, r}   + \log Z_{r, N} -\theta^{-1}r^{1/3} \Big\}.
\end{align*}

The next two lemmas show that $\mathcal{E}_1$ and $\mathcal{E}_2$ are high probability events. 

\begin{lemma}\label{E1_bd}
There exists a positive constant $\theta_0$ such that for each $0<\theta \leq \theta_0$, there exist positive constants $c_0, N_0$  such that for each $N\geq N_0$, $c_0\leq r \leq N/2$, we have
$$\mathbb{P}(\mathcal{E}_1|\barevent) \geq 1- e^{-\theta^{-1/100}}.$$
\end{lemma}

\begin{proof}
By independence, it suffices to prove this estimate for $\mathbb{P}(\mathcal{E}_1).$
We fix $\theta>0$ small. We will upper bound $\mathbb{P}(\mathcal{E}_1^c)$ using Proposition \ref{max_all_t}. In our application of this proposition, the variables $a = \sqrt{jr^{2/3}}$ and $t = \theta^{-1}$. 
By Proposition \ref{max_all_t},
$$\mathbb{P}(\mathcal{E}_1^c) \leq\sum_{j=1}^{\phi_2}
\mathbb{P}\Big( \log Z^{\textup{max}}_{\mathcal{L}_r^{jr^{2/3}}, N}  - \log Z_{r,N}  \geq \theta^{-1}\sqrt{jr^{2/3}}\Big) \leq \sum_{j=1}^{\phi_2}
e^{-\theta^{-1/50}} \leq e^{-\theta^{-1/100}}.$$
\end{proof}

\begin{lemma}\label{E2_bd}
There exists a positive constant $\theta_0$ such that for each $0<\theta \leq \theta_0$, there exist positive constants $c_0, N_0$  such that for each $N\geq N_0$, $c_0\leq r \leq N/2$, we have
$$\mathbb{P}(\mathcal{E}_2|\barevent) \geq 1- e^{-\theta^{-2}}.$$
\end{lemma}

\begin{proof}
By the FKG inequality, it suffices to show $\mathbb{P}(\mathcal{E}_2)\geq 1- e^{-\theta^{-2}}.$
To do this, we upper bound  $\mathbb{P}(\mathcal{E}_2^c)$. 
For simplicity, let us denote $\mathcal{L}^{r^{2/3}}_{(r-2hr^{2/3}, r+2hr^{2/3})}$ simply as $J^{h}$.
\begin{align}
\mathbb{P}(\mathcal{E}_2^c)
& \leq \sum_{|h| = \phi_1/2}^{r^{1/3}} \mathbb{P}\Big( \max_{{\bf p}\in J^{h}} \log Z_{0, {\bf p}}  + \log Z_{{\bf p}, N} \geq \log Z^{\textup{in} ,  {\theta r^{2/3}}}_{0, r}  + \log Z_{r, N} - \theta^{-1}r^{1/3} \Big) \nonumber\\
&\leq \sum_{|h| = \phi_1/2}^{r^{1/3}} \mathbb{P}\Big( \log Z^\textup{max}_{0, J^h}  + \log Z^\textup{max}_{J^h, N}  - \log Z^{\textup{in} , {\theta r^{2/3}}}_{0, r}   - \log Z_{r, N} \geq - \theta^{-1}r^{1/3} \Big)\nonumber\\
&\leq \sum_{|h| = \phi_1/2}^{r^{1/3}} \mathbb{P}\Big( \log Z^\textup{max}_{0, J^h}   - \log Z^{\textup{in} ,  {\theta r^{2/3}}}_{0, r}   \geq -  C'h^2r^{1/3}\Big)\label{E4_1} \\
& \qquad \qquad \qquad\qquad  + \mathbb{P}\Big(\log Z^\textup{max}_{J^h, N}  - \log Z_{r, N} \geq  (C'h^2 -\theta^{-1}) r^{1/3}\Big)\label{E4_2}
\end{align}
where $C'$ is a positive constant which we will fix (independent of $\theta$). 

Next, since $h^2\geq \theta^{-5}$, we see that \eqref{E4_2} is bounded by $e^{-C|h|^3}$ as it is exactly the same as \eqref{far_max} appearing in the proof of Proposition \ref{fluc_bound_l}. 

The probability in \eqref{E4_1} can be bounded as 
\begin{align*}
\eqref{E4_1} \leq \mathbb{P}\Big( \log Z^\textup{max}_{0, J^h}   - 2r{f_d} \geq -2C'h^2r^{1/3} \Big) +\mathbb{P}\Big( \log  Z^{\textup{in} ,  {\theta r^{2/3}}}_{0, r}   - 2r{f_d}   \leq  -C'h^2r^{1/3}\Big).
\end{align*}
Provided that $C'$ is fixed sufficiently small, the two probabilities above are upper bounded by $e^{-Ch^2}$ using Proposition \ref{trans_fluc_loss} and Theorem \ref{wide_similar}.
To summarize, we have shown that 
$$\mathbb{P}(\mathcal{E}_2^c) \leq \sum_{|h| = \phi_1/2}^\infty e^{-Ch^2} \leq e^{-\theta^{-2}},$$
and this finishes the proof of the lemma.
\end{proof}

\begin{proposition}\label{loc2}
On the event $(\cap_{i=1}^6 \mathcal{A}_i )\bigcap( \cap_{j=1}^2 \mathcal{E}_j)$, we have 
${\bf p}^* \in \mathcal{L}_r^{r^{2/3}}.$
\end{proposition}
\begin{proof}
The main idea of the proof is the following. First, we know that the  inequality below  must hold
$$
\max_{{\bf p}\in\mathcal{L}_r } \{\log Z_{0, {\bf p}} + \log Z_{{\bf p}, N}\} -  \log Z_{r,N} - \log Z^{\textup{in} , {\theta r^{2/3}}}_{0, r} \geq 0.
$$
And we will show that if ${\bf p} \not \in \mathcal{L}_{r}^{r^{2/3}}$, then on the event $(\cap_{i=1}^6 \mathcal{A}_i )\bigcap( \cap_{j=1}^2 \mathcal{E}_j)$ ,
$$\log Z_{0, {\bf p}} + \log Z_{{\bf p}, N} -  \log Z_{r,N} - \log Z^{\textup{in} , {\theta r^{2/3}}}_{0, r} < - r^{2/3} < 0,$$
hence it must be true that the maximizer ${\bf p}^* \in \mathcal{L}_r^{r^{2/3}}$.

First, we note that because we are on the event $\mathcal{E}_2$, then ${\bf p}^*$ must be in $\mathcal{L}_r^{\phi_1 r^{2/3}}$.
Then, within $\mathcal{L}_r^{\phi_1 r^{2/3}}$, the event $\cap_{i=1}^6 \mathcal{A}_i$  says we would lose more than $\theta^{-50} r^{1/3}$ amount of free energy comparing with going from $(0,0)$ to ${\bf p} \in \mathcal{L}_r^{\phi_1 r^{2/3}} \setminus \mathcal{L}_r^{\phi_1 r^{2/3}}$ instead of going  to $(r,r)$. 
{And for the free energy from $(N,N)$ down to $\mathcal{L}_r$}, $\mathcal{E}_1$ says for any ${\bf p}$ inside $\mathcal{L}_r^{\phi_1 r^{2/3}}$, we gain at most $\theta^{-1} \sqrt{\phi_1} r^{1/3}$ amount of free energy  comparing with going from $(r,r)$ to $(N, N)$. Thus, the loss $\theta^{-50} r^{1/3}$ is greater than the gain  $\theta^{-1}\sqrt{\phi_1} r^{1/3}$, hence, ${\bf p}^* \in \mathcal{L}_r^{r^{2/3}}.$ 

\end{proof}

\subsection{Expectation bounds}
In this subsection, we prove two propositions about the expected difference of free energies when conditioning on $\barevent$.

\begin{proposition}\label{2term_bound}
There exist positive constants $\newc\label{2term_bound_c1}, \theta_0$ such that for $0< \theta < \theta_0$, there exists a positive constant $c_0$ such that for each $r\geq c_0$, we have
$$\mathbb{E}\Big[\Big(\log Z_{0, \mathcal{L}_{r}}  - \log Z^{\textup{in} , {3\theta r^{2/3}}}_{0, r} \Big)^2 \Big| \barevent \Big] \leq  \oldc{2term_bound_c1}r^{2/3}.$$
\end{proposition}

\begin{proof}
Let us denote the high probability event 
$$\mathcal{D} = \cap_{i=1}^6 \mathcal{A}_i,$$
and we have $\mathbb{P}(\mathcal{D}^c| \barevent) \leq \mathbb{P}(\mathcal{D}^c)\leq e^{-\theta^{-2}}$ which is the statement of  Proposition \ref{lem_Ai}.

Now let us look at the expectation on the event $\mathcal{D}$. Using Proposition \ref{loc1}, we obtain the first inequality below. The second inequality follows from  $Z_{0, \mathcal{L}_{r}}(\mathfrak{A}_1) + \log 6 \geq \log Z^{\textup{in} ,  {3\theta r^{2/3}}}_{0, r}$ on the event $\mathcal{D}$, so reducing the value of $Z^{\textup{in} ,  {3\theta r^{2/3}}}_{0, r}$ makes the expectation bigger. To simplify the notation, let 
$R^* = R_{r-\theta^{3/2}r, r- \theta^{3/2}r + r/L}^{\theta r^{2/3}},$
and we have
\begin{align}
&\mathbb{E}\Big[\Big(\log Z_{0, \mathcal{L}_{r}}  - \log Z^{\textup{in} ,  {3\theta r^{2/3}}}_{0, r} \Big)^2 \mathbbm{1}_\mathcal{D} \Big| \barevent \Big] \leq \mathbb{E}\Big[\Big(\log Z_{0, \mathcal{L}_{r}}(\mathfrak{A}_1) - \log Z^{\textup{in} ,  {3\theta r^{2/3}}}_{0, r}  + \log 6 \Big)^2  \mathbbm{1}_\mathcal{D} \Big| \barevent \Big] \nonumber\\
&\qquad  \leq \mathbb{E}\Big[\Big(\log Z_{0, \mathcal{L}_{r}}(\mathfrak{A}_1) - \max_{R^*} \Big\{\log Z^{\textup{in}, R_{0,r}^{3\theta r^{2/3}}}_{0, {\bf p}} + \log Z^{\textup{in}, R_{0,r}^{\theta r^{2/3}}}_{ {\bf p}, r} \Big\}+  \log 6 \Big)^2 \Big| \barevent \Big] \label{exp_max}
\end{align}

Now, recall the definition of $\mathfrak{A}_1$ in \eqref{3As}, every path must touch 
$R^*.$
If we let ${\bf p}^*$ be the maximizer for 
$\max_{{\bf p } \in R^*} \log Z^{\text{in}, R_{0,r}^{3\theta r^{2/3}}}_{0, {\bf p}} + \log Z^{\text{in}, R_{0,r}^{\theta r^{2/3}}}_{{\bf p}, r},$
then 
\begin{align*}
\log Z_{0, \mathcal{L}}(\mathfrak{A}_1) &\leq 
 \log Z^{\textup{in} ,  {3\theta r^{2/3}}}_{0, {\bf p}^*} + \log Z_{{\bf p}^*, \mathcal{L}_{r}}.
\end{align*}
Therefore, the expectation \eqref{exp_max} can be upper bounded by  
\begin{align}
\eqref{exp_max} &\leq 10 \cdot \mathbb{E}\Big[\max_{{\bf p} \in  R^*} \Big(\log Z_{{\bf p},  \mathcal{L}_{r}}- \log Z^{\textup{in} ,  R_{0, r}^{\theta r^{2/3}}}_{{\bf p}, r}   \Big)^2 \big| \barevent\Big]\nonumber \\
&\leq  10 \cdot \mathbb{E}\Big[\max_{{\bf p} \in  R^*} \Big(\log Z_{{\bf p},  \mathcal{L}_{r}}- \log Z^{\textup{in} ,  R_{0, r}^{\theta r^{2/3}}}_{{\bf p}, r}\Big)^2 \Big]
\label{lastineq}\end{align}
{where the second inequality follows from $\log Z_{{\bf p},  \mathcal{L}_{r}}- \log Z^{\textup{in} ,  R_{0, r}^{\theta r^{2/3}}}_{{\bf p}, r} \geq 0$ and conditioning on $\barevent$ would decrease the difference by making $\log Z_{{\bf p},  \mathcal{L}_{r}}$ smaller.} 
Finally, to bound \eqref{lastineq}, from  monotonicity 
$$\log Z_{{\bf p}, \mathcal{L}_{r}^{r^{2/3}}} \geq \log Z^{\textup{in} , R_{0,r}^{\theta r^{2/3}}}_{{\bf p}, r},$$
and Theorem \ref{high_inf} and Theorem \ref{btl_upper} (note  $h = 1$ in these theorems), we have 
\begin{align}
\begin{split}\label{2l2bd}
\mathbb{E}\Big[\max_{{\bf p } \in R^*} \Big(\log Z_{{\bf p}, \mathcal{L}_{r}^{r^{2/3}}}  - (2r - |{\bf p}|_1)f_d\Big)^2\Big] &\leq   C \theta r^{2/3} \\
\mathbb{E}\Big[\max_{{\bf p } \in R^*}\Big(\log Z^{\textup{in} , R_{0,r}^{\theta r^{2/3}}}_{{\bf p}, r} - (2r - |{\bf p}|_1)f_d\Big)^2\Big]  &\leq C \theta r^{2/3}.
\end{split}
\end{align}
{Using the fact that 
$$\max_k (a_k - b_k)^2 \leq 10 \max_k \Big[ (a_k -c_k)^2 + (b_k-c_k)^2 \Big] \leq 10 \Big[\max_k (a_k -c_k)^2 + \max_{k'} (b_{k'}-c_{k'})^2\Big],$$}
the estimate \eqref{2l2bd} above  implies $\eqref{lastineq} \leq C \theta r^{2/3}$.

Next, for $\mathcal{D}^c$, we use the FKG inequality and Cauchy-Schwarz inequality, 
\begin{align*}
&\mathbb{E}\Big[\Big(\log Z_{0, \mathcal{L}_{r}}  - \log Z^{\textup{in} ,  {3\theta r^{2/3}}}_{0, r} \Big)^2 \mathbbm{1}_{\mathcal{D}^c} \Big| \barevent \Big]\\
& \leq \mathbb{E}\Big[\Big(\log Z_{0, \mathcal{L}_{r}}  - \log Z^{\textup{in} ,  {\theta r^{2/3}}}_{0, r} \Big)^2 \mathbbm{1}_{\mathcal{D}^c} \Big]\\
& \leq \mathbb{E}\Big[\Big(\log Z_{0, \mathcal{L}_{r}}  - \log Z^{\textup{in} ,  {\theta r^{2/3}}}_{0, r} \Big)^4 \Big]^{1/2} \mathbb{P}(\mathcal{D}^c)^{1/2}.
\end{align*}
Again, because $$\log Z_{0, \mathcal{L}_{r}}  - \log Z^{\textup{in} , {\theta r^{2/3}}}_{0, r} \geq 0,$$
the fourth moment term can be bounded by $C\theta^4 r^{4/3}$ using Theorem \ref{ptl_upper} and Theorem \ref{wide_similar}, which provide both  right and left tail upper bound for both.
Combined with the fact that $\mathbb{P}(\mathcal{D}^c)^{1/2} \leq e^{-\theta^{-1}}$, we have finished the proof.
\end{proof}

\begin{proposition}\label{4term_bound}
There exist  positive constants $ \newc\label{4term_bound_c1}, \theta_0$ such that for $0< \theta < \theta_0$, there exists a positive constant $c_0$ such that for each $r\geq c_0$, we have
$$\mathbb{E}\Big[\Big(\log Z_{0,N}  - \log Z_{r, N} - \log Z^{\textup{in} , {3\theta r^{2/3}}}_{0, r}\Big)^2 \Big| \barevent \Big] \leq \oldc{4term_bound_c1}r^{2/3}.$$
\end{proposition}

\begin{proof}
First, note we may replace $\log Z_{0, N}$ by $\max_{{\bf p}\in\mathcal{L}_r } \{\log Z_{0, {\bf p}}  + \log Z_{{\bf p}, N}) \}$ because
$$\log Z_{0,N} \geq \log Z_{r, N} + \log Z^{\textup{in} , R_{0,r}^{3\theta r^{2/3}}}_{0, r} $$
and
$$\log Z_{0, N}  \leq \max_{{\bf p}\in\mathcal{L}_r } \{\log Z_{0, {\bf p}}  + \log Z_{{\bf p}, N} \} + 100 \log r.$$

Next, let us define the event 
$$\mathcal{D}' = (\cap_{i=1}^6 \mathcal{A}_i )\bigcap( \cap_{j=1}^2 \mathcal{E}_j).$$
We have $\mathbb{P}(\mathcal{D'}^c|\barevent) \leq \mathbb{P}(\mathcal{D'}^c) \leq e^{-\theta^{-1/100}}$ from  Lemma \ref{E1_bd} and Lemma \ref{E2_bd}.
We again will split the expectation into two parts according to $\mathcal{D}'$. 

On the event $\mathcal{D}$, by Proposition \ref{loc2}, the maximizer ${\bf p}^*$ is contained inside $\mathcal{L}_r^{r^{2/3}}$,
\begin{align*}
&\mathbb{E}\Big[\Big(\max_{{\bf p}\in\mathcal{L}_r } \{\log Z_{0, {\bf p}}  + \log Z_{{\bf p}, N} \}- \log Z_{r, N} - \log Z^{\textup{in} , {3\theta r^{2/3}}}_{0, r}\Big)^2 \mathbbm{1}_\mathcal{D'} \Big| \barevent \Big]\\
&\leq 2\mathbb{E}\Big[\Big(\max_{{\bf p}\in\mathcal{L}_r } \log Z_{0, {\bf p}}  - \log Z^{\textup{in} ,  {3\theta r^{2/3}}}_{0, r} \Big)^2  \Big| \barevent \Big] + 2\mathbb{E}\Big[\Big(\max_{{\bf p}\in\mathcal{L}_r^{r^{2/3}} }   \log Z_{{\bf p}, N} - \log Z_{r,N} \Big)^2 \Big]\\
& \leq Cr^{2/3}.
\end{align*}
where the last bound follows from Proposition \ref{2term_bound} and Proposition \ref{max_all_t}.

On the event $\mathcal{D'}^c$, we again just bound with the FKG and the Cauchy-Schwartz inequality,
\begin{align}
&\mathbb{E}\Big[\Big(\max_{{\bf p}\in\mathcal{L}_r } \{\log Z_{0, {\bf p}}  + \log Z_{{\bf p}, N} \} - \log Z_{r, N} - \log Z^{\textup{in} , {3\theta r^{2/3}}}_{0, r}\Big)^2 \mathbbm{1}_{\mathcal{D'}^c} \Big| \barevent \Big]\nonumber\\ 
& \leq \mathbb{E}\Big[\Big(\max_{{\bf p}\in\mathcal{L}_r } \{\log Z_{0, {\bf p}} + \log Z_{{\bf p}, N} \} - \log Z_{r, N} - \log Z^{\textup{in} , {\theta r^{2/3}}}_{0, r}\Big)^2 \mathbbm{1}_{\mathcal{D'}^c}\Big]\nonumber\\
& \leq \mathbb{E}\Big[\Big(\max_{{\bf p}\in\mathcal{L}_r } \{\log Z_{0, {\bf p}}  + \log Z_{{\bf p}, N} \} - \log Z_{r, N} - \log Z^{\textup{in} , {\theta r^{2/3}}}_{0, r}\Big)^4 \Big]^{1/2}\mathbb{P}(\mathcal{D'}^c)^{1/2}.\label{last_4}
\end{align}
The fourth moment term above can be bounded as 
\begin{align*}
& C\mathbb{E}\Big[\Big(\max_{{\bf p}\in\mathcal{L}_r } \{\log Z_{0, {\bf p}}  + \log Z_{{\bf p}, N} \} - \log Z_{r, N} - \log Z_{0, r}\Big)^4 \Big] + C\mathbb{E}\Big[\Big(\log Z_{0,r} - \log Z^{\textup{in} , {\theta r^{2/3}}}_{0, r}\Big)^4 \Big]\\
&\leq C\mathbb{E}\Big[\Big(\max_{{\bf p}\in\mathcal{L}_r } \{\log Z_{0, {\bf p}}  + \log Z_{{\bf p}, N} \} - \log Z_{r, N} - \log Z_{0, r}\Big)^4 \Big] + C\mathbb{E}\Big[\Big(\log Z_{0,r} - 2rf_d \Big)^4 \Big] \\ 
& \qquad \qquad \qquad \qquad\qquad \qquad \qquad \qquad  + C\mathbb{E}\Big[\Big(2rf_d- \log Z^{\textup{in} , {\theta r^{2/3}}}_{0, r}\Big)^4 \Big]\\
&\leq Cr^{4/3} + Cr^{4/3} + C\theta^{-4}r^{4/3}
\end{align*}
using  Proposition \ref{nest_max1}, Proposition \ref{up_ub}, and Theorem \ref{wide_similar}.
Since $\mathbb{P}(\mathcal{D}'^c)^{1/2} \leq e^{-\theta^{-1/200}}$, the expectation on $\mathcal{D}'^c$ is also upper bounded by $Cr^{2/3}$. With this, we have finished the proof.

\end{proof}

\subsection{Constrained variance lower bound}
The main purpose of this section is to prove the following theorem on the lower bound of the constrained variance. 
Recall $\mathcal{F}_\theta$ as the $\sigma$-algebra generated by the weights in the set $[\![(0,0), (N,N)]\!] \setminus R^{\theta r^{2/3}}_{0,r}$.
\begin{theorem}\label{con_var}
There exist positive constants $\newc\label{con_var_c1} ,\theta_0$ such that for each $0< \theta \leq \theta_0$, there exists a positive constant $c_0$ and an event $\mathcal{B}'\subset \barevent$ with $\mathbb{P}(\mathcal{B}'|\barevent) \geq 1/2$ such that for each $r\geq c_0$, we have 
$$\Var\Big( \log Z^{\textup{in},{3\theta r^{2/3}}}_{0,r} \Big | \mathcal{F}_\theta\Big) (\omega) \geq \oldc{con_var_c1} \theta^{-1/2}r^{2/3}\qquad \textup{ for each $\omega \in \mathcal{B}'$.}$$
\end{theorem}

First, let us define the following sequence of events. For $i\in (\tfrac{1}{3}{\theta^{-3/2}}, \tfrac{2}{3}{\theta^{-3/2}})$ and a positive constant $q^*$ which we  fix  in Proposition \ref{i_likely} below, 
\begin{align*}
\mathcal{U}_i &= \Big\{\log Z_{0,\mathcal{L}_{i\theta^{3/2}r}} - \log Z^{\textup{in},{3\theta r^{2/3}}}_{0,i\theta^{3/2}r} \leq q^* \sqrt{\theta}r^{1/3}\Big\}\\
\mathcal{V}_i &= \Big\{\log Z_{\mathcal{L}_{(i+1)\theta^{3/2}r}, r} - \log Z^{\textup{in}, {3\theta r^{2/3}}}_{(i+1)\theta^{3/2}r,r} \leq q^* \sqrt{\theta}r^{1/3}\}\\
\mathcal{W}_i &= \Big\{\log Z_{\mathcal{L}_{i\theta^{3/2}r}^{3\theta r^{2/3}}, \mathcal{L}_{(i+1)\theta^{3/2}r}^{3\theta r^{2/3}}}   - 2\theta^{3/2}r{f_d} \leq q^* \sqrt{\theta}r^{1/3}\Big\}\\
\end{align*}

On the event $\barevent$, the events defined above happen with high probability.
\begin{proposition}\label{i_likely}
There exist positive constants $q^*, \theta_0$ such that for each $0 < \theta \leq \theta_0$, there exists a positive constant $c_0$ such that for each  $r\geq c_0$, 
$$\mathbb{P}(\mathcal{U}_i\cap \mathcal{V}_i\cap \mathcal{W}_i | \barevent) \geq 1- 10^{-4}.$$
\end{proposition}

\begin{proof}
We upper bound $\mathcal{U}_i^c$ and $\mathcal W_i^c$. And by symmetry, the estimate for $\mathcal V_i^c$ is the same as $\mathcal U_i^c$.

First, by the FKG inequality and Theorem \ref{ptl_upper},
$$\mathbb{P}(\mathcal W_i^c|\barevent) \leq \mathbb{P}\Big(\log Z^{\textup{in},{3\theta r^{2/3}}}_{\mathcal{L}_{i\theta^{3/2}r}^{3\theta r^{2/3}}, \mathcal{L}_{(i+1)\theta^{3/2}r}^{3\theta r^{2/3}}}  - 2\theta^{3/2}r{f_d} \leq q^* \sqrt{\theta}r^{1/3}\Big)  \leq e^{-Cq^*}.$$

Next, we upper bound $\mathbb{P}(\mathcal U_i^c|\barevent)$.  Let $\mathfrak{A}$ denote the collection of paths going from $(0,0)$ to $\mathcal{L}_{i\theta^{3/2}r}^{3\theta r^{2/3}}$ such that they stay within the diagonal sides of $R_{0, (i-1)\theta^{3/2}r}^{3\theta r^{2/3}}$ and they touch the box $R_{(i-2)\theta^{3/2}r, (i-1)\theta^{3/2}r}^{\theta r^{2/3}}$. (Note the $\mathfrak{A}$ here plays the role of $\mathfrak{A}_1$ from \eqref{3As}.) Applying Proposition \ref{loc1} (for the free energy from $0$ to $i\theta^{3/2}r$ instead of from $0$ to $r$), we know that for $\theta$ sufficiently small.
$$\mathbb{P}\Big(\log Z_{0,\mathcal{L}_{i\theta^{3/2}r}}  \leq \log Z_{0,\mathcal{L}_{i\theta^{3/2}r}}(\mathfrak{A}) + \log 6  \Big |\barevent \Big) \geq 1-10^{-8}.$$
Then, it suffices for us to upper bound the event 
\begin{equation}\label{weshow111}
\Big\{\log Z_{0,\mathcal{L}_{i\theta^{3/2}r}} (\mathfrak{A})  - \log Z^{\textup{in},{3\theta r^{2/3}}}_{0,i\theta^{3/2}r} >  q^* \sqrt{\theta}r^{1/3} - \log 6\Big\}.
\end{equation}

Now, since all the paths in $\mathfrak{A}$ enter the box 
$$R^* = R^{\theta r^{2/3}}_{(i-2)\theta^{3/2}r,(i-1)\theta^{3/2}r },$$ let $\mathfrak{A}_{\bf p}$ denote the collection of paths in $\mathfrak{A}$ that go through ${\bf p} \in R^*$. Let ${\bf p}^*$ be the maximizer of 
$$\max_{{\bf p} \in R^*} \Big\{\log Z_{0,\mathcal{L}_{i\theta^{3/2}r}}(\mathfrak{A}_{{\bf p}})\Big\}.$$
And it holds that 
$$ \log Z_{0,\mathcal{L}_{i\theta^{3/2}r} } (\mathfrak{A}) \leq \log Z_{0,\mathcal{L}_{i\theta^{3/2}r} } (\mathfrak{A}_{{\bf p}^*}) + 100 \log r.$$
Now, to bound \eqref{weshow111}, it suffices for us to upper bound 
$$
\mathbb{P}\Big(\log Z_{0,\mathcal{L}_{i\theta^{3/2}r} } (\mathfrak{A}_{\bf p^*}) - \log Z^{\textup{in}, {3\theta r^{2/3}}}_{0,i\theta^{3/2}r} >  \tfrac{1}{2} q^* \sqrt{\theta}r^{1/3}\Big| \barevent \Big).
$$
{We will replace these two free energies appearing above with the right side below}
\begin{align*}
\log Z_{0,\mathcal{L}_{i\theta^{3/2}r} } (\mathfrak{A}_{\bf p^*}) = \log Z_{0,{\bf p}^*}  + \log Z_{{\bf p}^*,\mathcal{L}_{i\theta^{3/2}r}}\\
\log Z^{\textup{in},{3\theta r^{2/3}}}_{0,i\theta^{3/2}r} \geq \log Z^{\textup{in},R_{0, r}^{3\theta r^{2/3}}}_{0,{\bf p}^*}   + \log Z^{\textup{in}, R_{0,r}^{\theta r^{2/3}}}_{{\bf p}^*,i\theta^{3/2}r }.
\end{align*}
{After the substitution, because 
$\log Z_{0,{\bf p}^*} - \log Z^{\textup{in},R_{0, r}^{3\theta r^{2/3}}}_{0,{\bf p}^*}$ is non-negative},
it suffices for us to upper bound
\begin{align} 
&\mathbb{P}\Big(\log Z_{{\bf p}^*,\mathcal{L}_{i\theta^{3/2}r}}  - \log Z^{\textup{in},R_{0,r}^{\theta r^{2/3}}}_{{\bf p}^*,i\theta^{3/2}r }  >  \tfrac{1}{2} q^* \sqrt{\theta}r^{1/3} \Big| \barevent \Big)\nonumber\\
& \leq \mathbb{P}\Big(\log Z_{{\bf p}^*,\mathcal{L}_{i\theta^{3/2}r}}  - \log Z^{\textup{in},R_{0,r}^{\theta r^{2/3}}}_{{\bf p}^*,i\theta^{3/2}r }  >  \tfrac{1}{2} q^* \sqrt{\theta}r^{1/3} \Big). \label{finalbd}
\end{align}
And we may bound \eqref{finalbd} as
\begin{align*}
\eqref{finalbd}
&\leq \mathbb{P}\Big(\max_{{\bf p}\in R^*}\log Z_{{\bf p},\mathcal{L}_{i\theta^{3/2}r} }  -  (2i\theta^{3/2}r - |{\bf p}|_1){f_d}  >  \tfrac{1}{4} q^* \sqrt{\theta}r^{1/3} \Big)\\
&\qquad \qquad + \mathbb{P}\Big(\min_{{\bf p}\in R^*}\log Z^{\textup{in},R_{0, r}^{\theta r^{2/3}}}_{{\bf p},i\theta^{3/2}r }   - (2i\theta^{3/2}r - |{\bf p}|_1){f_d} <  -\tfrac{1}{4} q^* \sqrt{\theta}r^{1/3} \Big).
\end{align*}
Both of these probabilities are bounded by $e^{-{q^*}^{1/10}}$ by Theorem \ref{high_inf} and Theorem \ref{btl_upper} . Finally, by fixing $q^*$ sufficiently large, this completes the proof of this proposition.
\end{proof}

We say an index $i\in (\tfrac{1}{3}{\theta^{-3/2}}, \tfrac{2}{3}{\theta^{-3/2}})$ is \textit{good} if 
$$\mathbb{P}(\mathcal{U}_i\cap \mathcal{V}_i\cap \mathcal{W}_i | \mathcal{F}_\theta)(\omega) \geq 1-10^{-2} \qquad  \textup{ where $\omega \in \barevent$}.$$
Note for a given $\omega \in \barevent$, the set of good indices is deterministic.

\begin{lemma} \label{find_B}
Let $\mathcal{B}' \subset \barevent$  denote the event that the number of good indices is at least $\tfrac{1}{6}\theta^{-3/2}$.  Then $\mathbb{P}(\mathcal{B}'| \barevent) \geq 1/2.$
\end{lemma}
\begin{proof} Since $\barevent$ is $\mathcal{F}_\theta$-measurable,  by Markov's inequality,
$$\mathbb{P}(\textup{$i$ is bad}|\barevent) \leq \mathbb{P}\Big(\{\omega : \mathbb{P}(\mathcal{U}_i^c\cup \mathcal{V}_i^c\cup \mathcal{W}_i^c|\mathcal{F}_\theta)(\omega)> 0.01\}\Big|\barevent\Big) \leq 100\mathbb{P}(\mathcal{U}_i^c\cup \mathcal{V}_i^c\cup \mathcal{W}_i^c|\barevent) \leq 1/10 .$$
Then, the expected number of bad indices (conditional on $\barevent$) is upper bounded by $\tfrac{1}{10}\cdot \tfrac{2}{3}\theta^{-3/2}$, and a further application of Markov's inequality completes the proof.
\end{proof}

Next we prove Theorem \ref{con_var} using the $\mathcal{B}'$ obtained in Lemma \ref{find_B}. Let us fix an 
$\omega'\in \mathcal{B}'$ for the remainder of this proof; {and for any configuration $\omega$ we shall define, its projection onto the vertices outside $R_{0,r}^{\theta r^{2/3}}$ will agree with $\omega'$.} Given this $\omega'$, the collection of good indices is known. Let us fix an enumeration of a portion of the good indices 
$$J = \{i_1, i_2, \dots, i_K\}$$
where $K \geq \tfrac{1}{6}\theta^{-3/2}$.
Now, define a sequence of $\sigma$-algebras $\mathcal{S}_0\subset \mathcal{S}_1 \subset \mathcal{S}_2 \subset \dots \subset \mathcal{S}_K$ where $\mathcal{S}_0$ is generated by the configuration $\omega'\in \mathcal{B}'$ together with the configuration on $R^{\theta r^{2/3}}_{i\theta^{3/2}r, (i+1)\theta^{3/2}r}$ for all $i\not\in J$, and for $j\geq 1$, $\mathcal{S}_j$ is the $\sigma$-algebra generated by $\mathcal{S}_{j-1}$ and the configuration inside $R^{\theta r^{2/3}}_{i_j\theta^{3/2}r, (i_j+1)\theta^{3/2}r}$. Note that  $S_K$ is the $\sigma$-algebra of the entire weight configuration.

Consider the Doob martingale 
$$M_j = \mathbb{E}\big[\log Z_{0,r}^{\text{in}, 3\theta r^{2/3}} \big|\mathcal{S}_j\big].$$ By the  variance decomposition of a Doob martingale, it follows that 
$$\Var\Big( \log Z^{\textup{in}, {3\theta r^{2/3}}}_{0,r} \Big | \mathcal{F}_\theta\Big)(\omega') \geq \sum_{j=1}^K \mathbb{E}[(M_{j} -  M_{j-1})^2|\mathcal{F}_\theta](\omega').$$
  Theorem \ref{con_var} follows  directly from the lemma below. The proof of the lemma goes by a re-sampling argument. This idea first appeared in the zero-temperature set-up in \cite{timecorriid, timecorrflat}, although the setup there is different from ours.
\begin{lemma}\label{c_var_lem}
There exist positive constants $\theta_0, 
\newc\label{c_var_lem_c1}$ such that for each $0< \theta < \theta_0$, there exists a positive constant $ c_0$ such that for each $r\geq c_0$, the following holds: for each $i_j\in J$, there is an $\mathcal{S}_{j-1}$ measurable event $\mathcal{G}_j$ with $\mathbb{P}(\mathcal{G}_j| \mathcal{F}_\theta)(\omega') \geq 1/2$ for each  $\omega'\in \mathcal{B}'$, and for each $\omega\in\mathcal{G}_j$ we have 
$$\mathbb{E}[(M_j - M_{j-1})^2|\mathcal{S}_{j-1}](\omega) \geq \oldc{c_var_lem_c1}\theta r^{2/3}.$$
\end{lemma}
\begin{proof}

Define the event
$$F = \Big\{ \log Z^{\textup{in}, {\theta r^{2/3}}}_{i_j\theta^{3/2}r, (i_j+1)\theta^{3/2} r}  - 2(\theta^{3/2}r){f_d} \geq 100q^*\sqrt{\theta} r^{1/3}\Big\}$$
where the fixed constant $q^*$ is from Proposition \ref{i_likely}.  By Theorem \ref{c_up_lb}, $\mathbb{P}(F) \geq c >0$. 
Let $\omega_1$ denote the configuration on $R_{i_j\theta^{3/2}r, (i_j+1)\theta^{3/2} r}^{\theta r^{2/3}}$ drawn from i.i.d.~inverse-gamma distribution. And let $\wt\omega_1$ denote the configuration on $R_{i_j\theta^{2/3}r, (i_j+1)\theta^{3/2} r}^{\theta r^{2/3}}$  which is drawn from i.i.d.~inverse-gamma distribution but conditioned on $F$. By the FKG inequality and Strassen's Theorem \cite{Stra-65}, there exists a coupling measure of the joint distribution $(\omega_1, \wt\omega_1)$ such that $\omega_1 \leq  \wt \omega_1$ coordinatewise. Let $\beta$ denote such a coupling measure.

Let $\omega_0$ denote the configuration on all vertices that are revealed in $\mathcal{S}_{j-1}$, and recall the projection of $\omega_0$ outside of $R^{\theta r^{2/3}}_{0,r}$ is the same as $\omega' \in \mathcal{B}'$, the environment which we have fixed previously. And $\omega_2$ is the remaining weight configurations besides $\omega_0, \omega_1, \widetilde \omega_1$.
Let $\underline{\omega} = (\omega_0, \omega_1, \omega_2)$ and $\wt{\underline{\omega}} = (\omega_0, \wt\omega_1, \omega_2)$ to denote the two coupled environments under $\beta$, which was defined in the previous paragraph. We have 
\begin{equation}\label{int_lb}
\frac{1}{\mathbb{P}(F)}\mathbb{E}[M_j \mathbbm{1}_F|\mathcal{S}_{j-1}] - M_{j-1} = \int\log(Z_{0,r}^{\text{in}, 3\theta r^{2/3}})(\wt{\underline{\omega}}) - \log Z_{0,r}^{\text{in}, 3\theta r^{2/3}}({\underline{\omega}})  \beta(d\omega_1, d\wt\omega_1) \mathbb{P}(d\omega_2).
\end{equation}
Since $F$ is independent of $\mathcal{S}_{j-1}$ and $M_{j-1}$, we also have 
\begin{equation}\label{exp_lb}
\begin{split}
\mathbb{E}[(M_j- M_{j-1})^2|\mathcal{S}_{j-1}] & \geq \pb(F)\frac{1}{\pb(F)} \mathbb{E}[(M_j- M_{j-1})^2 \mathbbm{1}_F| \mathcal{S}_{j-1}]\\
& \geq \mathbb{P}(F) \Big(\frac{1}{\pb(F)}\mathbb{E}[M_j\mathbbm{1}_F| \mathcal{S}_{j-1}]- M_{j-1}\Big)^2.
\end{split}
\end{equation}

Next, we construct the event $\mathcal{G}_j$. 
Note that $\mathcal{U}_{i_j}$ is $\mathcal{S}_{j-1}$ measurable, but $\mathcal{V}_{i_j}$ is not. Define another $\mathcal{S}_{j-1}$  measurable event  $$\wt{\mathcal{V}}_{i_j} = \{\omega_0 : \mathbb{P}(\mathcal{V}_{i_j}| \mathcal{S}_{j-1})(\omega_0) \geq 0.9\}.$$
We set $\mathcal{G}_{j}=\wt{\mathcal{U}}_{i_j}\cap\wt{\mathcal{V}}_{i_j}$
where $$\wt{\mathcal{U}}_{i_j} = \{\omega_0 : \mathbbm{1}_{\mathcal{U}_{i_j}} =1\}.$$
By the definition that $i_j$ is good, $\mathbb{P}(\mathcal{U}_{i_j}|\mathcal{F}_\theta)(\omega') \geq 1-10^{-2}$.   By Markov's inequality, 
$$\mathbb{P}(\wt{\mathcal{V}}_{i_j}^c| \mathcal{F}_\theta)(\omega') = \mathbb{P}(\{\omega_0 : \mathbb{P}(\mathcal{V}_{i_j}^c| S_{j-1})(\omega_0) \geq 0.1 \}|\mathcal{F}_\theta)(\omega') \leq 10 \mathbb{P}(\mathcal{V}_{i_j}^c|\mathcal{F}_\theta)(\omega') \leq 0.1,$$
which implies $\mathbb{P}(\wt{\mathcal{V}}_{i_j}^c| \mathcal{F}_\theta)(\omega') \geq 0.9$. Hence, $\mathcal{G}_j$ satisfies the requirement
$$\mathbb{P}(\mathcal{G}_j|\mathcal{F}_\theta)(\omega') \geq 1/2.$$


For $\omega_0\in \mathcal{G}_{j}$, {starting with the inequality \eqref{exp_lb} then applying \eqref{int_lb}}, we obtain that 
\begin{equation}\label{lb_w0}
\mathbb{E}[(M_j- M_{j-1})^2|\mathcal{S}_{j-1}](\omega_0) \geq \mathbb{P}(F) \Big(\int\log Z_{0,r}^{\text{in}, 3\theta r^{2/3}} (\wt{\underline{\omega}}) - \log Z_{0,r}^{\text{in}, 3\theta r^{2/3}} ({\underline{\omega}})  \beta(d\omega_1, d\wt\omega_1)  \mathbb{P}(d\omega_2) \Big)^2.
\end{equation}
Since the integrand above is non-negative, we further lower bound the integral in the right-hand side of \eqref{lb_w0} by 

\begin{equation}
    \label{lb_next}
    \int_{\omega_1\in \wt{\mathcal{W}}_{i_j}, \omega_{2}\in D} \left(\log Z_{0,r}^{\text{in}, 3\theta r^{2/3}} (\wt{\underline{\omega}}) - \log Z_{0,r}^{\text{in}, 3\theta r^{2/3}} ({\underline{\omega}})\right)  \beta(d\omega_1, d\wt\omega_1)  \mathbb{P}(d\omega_2)
\end{equation}
where $$\wt{\mathcal{W}}_{i_j} = \{\omega_1 : \mathbbm{1}_{\mathcal{W}_{i_j}}(\underline{\omega}) =1\}.$$
Note that since $\omega_0$ and $\omega'$ has been fixed, $\mathbbm{1}_{\mathcal{W}_{i_j}}$ is determined by $\omega_1$, and 
$$D=D(\omega_0)=\Big\{\omega_2 : \text{ $\mathcal{V}_{i_j}$ holds on $(\omega_0, \omega_2)$}\Big\}$$
where $\mathcal{V}_{i_j}$ is determined by $(\omega_0,\omega_2)$.


Now, it holds that 
$$\log Z_{0,r}^{\textup{in}, 3\theta r^{2/3}} (\wt{\underline{\omega}}) \geq \log Z_{0,i_j\theta^{3/2}r}^{\textup{in}, 3\theta r^{2/3}} (\wt{\underline{\omega}})+ \log Z_{i_j\theta^{3/2}r,(i_j+1)\theta^{3/2}r}^{\textup{in},\theta r^{2/3}} (\wt{\underline{\omega}})+\log Z_{(i_j+1)\theta^{3/2}r,r}^{\textup{in},3\theta r^{2/3}} (\wt{\underline{\omega}})$$
and 
$$\log Z_{0,r}^{\textup{in},3\theta r^{2/3}} ({\underline{\omega}}) \leq\log Z _{0, \mathcal{L}_{i_j\theta^{3/2}r} } ({\underline{\omega}})
+ \log Z^{\textup{in},{3\theta r^{2/3}}}_{\mathcal{L}_{i_j\theta^{3/2}r}^{3\theta r^{2/3}}, \mathcal{L}_{(i_j+1)\theta^{3/2}r}^{3\theta r^{2/3}}} ({\underline{\omega}})+ \log Z_{\mathcal{L}_{(i_j+1)\theta^{3/2}r},r} ({\underline{\omega}}).$$
Since $\omega_0 \in \mathcal{G}_j$, we have on  $\{\omega_1\in \wt{\mathcal{W}}_{i_j}, \omega_{2}\in D\}$,
\begin{align*}
&\log Z_{0,r}^{\textup{in}, 3\theta r^{2/3}} (\wt{\underline{\omega}}) - \log Z_{0,r}^{\textup{in}, 3\theta r^{2/3}} ({\underline{\omega}}) \\
& \qquad\qquad \geq    \log Z_{i_j\theta^{3/2}r,(i_j+1)\theta^{3/2}r}^{\textup{in}, \theta r^{2/3}} (\wt \omega_1) - \log Z^{\textup{in},{3\theta r^{2/3}}}_{\mathcal{L}_{i_j\theta^{3/2}r}^{3\theta r^{2/3}}, \mathcal{L}_{(i_j+1)\theta^{3/2}r}^{3\theta r^{2/3}}} (\omega_1) - 2q^* \sqrt{\theta}r^{1/3}\\
&\geq 50q^*\sqrt{\theta}r^{1/3}
\end{align*}
where the last inequality holds since $\wt{\omega}_1\in F$ by definition.

We can therefore lower bound \eqref{lb_next} by 
$$50q^*\sqrt{\theta}r^{1/3}\int_{ \omega_1\in \wt{\mathcal{W}}_{i_j}} \beta(d\omega_1, d\wt\omega_1)  \int_{\omega_2\in D} \mathbb{P}(d\omega_2).$$
Since, $\mathcal{W}_{i_j}(\underline{\omega})$ is determined by $\omega'$ and $\omega_1$, it follows that 
$$\int_{ \omega_1\in \wt{\mathcal{W}}_{i_j}} \beta(d\omega_1, d\wt\omega_1) = \mathbb{P}(\mathcal{W}_{i_j}|\mathcal{F}_\theta)(\omega')\geq 0.99$$ since $i_j$ is good. Since $\mathcal{G}_j \subseteq \wt{\mathcal{V}}_{i_j}$, it follows from the definition of $D$ that $\int_{\omega_2\in D} d\omega_2\ge 0.9$ for all $\omega_0\in \mathcal{G}_j$. We can therefore lower bound the right hand side of \eqref{lb_w0} by $\theta r^{2/3}\mathbb{P}(F)\times (0.9\times 0.99 \times 50q^*)^2$
thereby completing the proof. 

\end{proof}

\subsection{Covariance lower bound}\label{cov_lb}

To start with fix $\theta$ sufficiently small. Let us recall three subsets of $\barevent$ from Theorem \ref{con_var}, Proposition \ref{2term_bound} and Proposition \ref{4term_bound}, 
\begin{align*}
\mathcal{B}' &: \Var\Big( \log Z^{\textup{in},{3\theta r^{2/3}}}_{0,r} \Big | \mathcal{F}_\theta\Big)(\omega) \geq \oldc{con_var_c1} \theta^{-1/2}r^{2/3}\quad \textup{ for $\omega \in \mathcal{B}'$.}\\
\mathcal{B}'' &: \mathbb{E}\Big[\Big(\log Z_{0, \mathcal{L}_{r}}  - \log Z^{\textup{in} , {3\theta r^{2/3}}}_{0, r} \Big)^2 \Big| \mathcal{F}_\theta \Big](\omega) \leq  100\oldc{2term_bound_c1}r^{2/3} \quad \text{ for $\omega \in \mathcal{B}''$}\\
\mathcal{B}''' &: \mathbb{E}\Big[\Big(\log Z_{0,N}  - \log Z_{r, N} - \log Z^{\textup{in} , {3\theta r^{2/3}}}_{0, r}\Big)^2 \Big| \mathcal{F}_\theta \Big](\omega) \leq 100 \oldc{4term_bound_c1}r^{2/3} \quad \text{ for $\omega \in \mathcal{B}'''$}.
\end{align*}
Besides $\mathbb{P}(\mathcal{B}'|\barevent) \geq 1/2$, we also know $\mathbb{P}(\mathcal{B}''|\barevent) \geq 0.99$ and $\mathbb{P}(\mathcal{B}'''|\barevent) \geq 0.99$ by Markov inequality applied to their complements.

Now going back to \eqref{main_goal}, let us define
$$\mathcal{E}_\theta =  \mathcal{B}' \cap \mathcal{B}'' \cap \mathcal{B}'''.$$
We know $\mathbb{P}({\mathcal{E}_\theta}) > \epsilon_0 > 0$ because $\mathbb{P}( \mathcal{E}_\theta|\barevent) \geq 0.1$.
Finally, let us show \eqref{main_goal}. This follows directly from a sequence of inequalities. For $\omega \in \mathcal{E}_\theta$
\begin{align*}
&\Cov(\log Z_{0,r}, \log Z_{0, N} |\mathcal{F}_\theta)(\omega)\\
& = \Cov(\log Z_{0,r}, \log Z_{0, N} - \log Z_{r, N} |\mathcal{F}_\theta)(\omega) \\
& = \Cov\Big(\log Z^{\textup{in},{3\theta r^{2/3}}}_{0,r}, \log Z_{0, N} - \log Z_{r, N}\Big|\mathcal{F}_\theta\Big)(\omega) \\
&\qquad \qquad \qquad + \Cov\Big(\log Z_{0,r} - \log Z^{\textup{in},{3\theta r^{2/3}}}_{0,r}, \log Z_{0, N} - \log Z_{r, N}\Big|\mathcal{F}_\theta\Big)(\omega)\\
&= \Var \Big( \log Z^{\textup{in},{3\theta r^{2/3}}}_{0,r} \Big| \mathcal{F}_\theta \Big)(\omega)+ \Cov\Big(\log Z^{\textup{in},{3\theta r^{2/3}}}_{0,r}, \log Z_{0, N} - \log Z_{r, N} - \log Z^{\textup{in},{3\theta r^{2/3}}}_{0,r}\Big|\mathcal{F}_\theta\Big)(\omega)\\
&  \qquad \qquad \qquad  + \Cov \Big(\log Z_{0,r} - \log Z^{\textup{in},{3\theta r^{2/3}}}_{0,r}, \log Z^{\textup{in},{3\theta r^{2/3}}}_{0,r}  \Big| \mathcal{F}_\theta \Big)(\omega) \\
& \qquad \qquad \qquad + \Cov\Big(\log Z_{0,r} - \log Z^{\textup{in},{3\theta r^{2/3}}}_{0,r}, \log Z_{0, N} - \log Z_{r, N} - \log Z^{\textup{in},{3\theta r^{2/3}}}_{0,r}\Big|\mathcal{F}_\theta\Big)(\omega)\\
& \geq  \Var \Big( \log Z^{\textup{in},{3\theta r^{2/3}}}_{0,r} \Big| \mathcal{F}_\theta \Big)(\omega)-  \sqrt{\Var \Big( \log Z^{\textup{in},{3\theta r^{2/3}}}_{0,r} \Big| \mathcal{F}_\theta \Big)(\omega)} \sqrt{\Var \Big(\log Z_{0,r} - \log Z^{\textup{in},{3\theta r^{2/3}}}_{0,r} \Big| \mathcal{F}_\theta \Big)(\omega)}\\
& \qquad \qquad  -  \sqrt{\Var \Big( \log Z^{\textup{in},{3\theta r^{2/3}}}_{0,r} \Big| \mathcal{F}_\theta \Big)(\omega)}\sqrt{\Var \Big(\log Z_{0,N} - Z_{r, N} - \log Z^{\textup{in},{3\theta r^{2/3}}}_{0,r} \Big| \mathcal{F}_\theta \Big)(\omega)}\\
&\qquad \qquad  - \sqrt{\Var \Big(\log Z_{0,r} - \log Z^{\textup{in},{3\theta r^{2/3}}}_{0,r} \Big| \mathcal{F}_\theta \Big)(\omega)} \sqrt{\Var \Big(\log Z_{0,N} - Z_{r, N} - \log Z^{\textup{in},{3\theta r^{2/3}}}_{0,r} \Big| \mathcal{F}_\theta \Big)(\omega)}\\
& \geq C \theta^{-1/2} r^{2/3}
\end{align*}
where the last inequality holds by the definition of the event $\mathcal{E}_\theta$ for $\theta$ sufficiently small.

\section{Nonrandom fluctuation lower bound}\label{sec:nr}

First, let us prove Theorem \ref{statiid_low}.
\begin{proof}[Proof of Theorem \ref{statiid_low}]
To simplify the notation, let us simply use $Z$ (instead of $\wt Z$) to denote the version of the partition function where we also include the weight at the starting point. Note this does not apply to $Z^\rho$ as $Z^\rho$ does not pick up any vertex weight at its starting point.
Also, let us define 
$${\bf v}_N = 2N{\boldsymbol\xi}[\rho].$$
To prove the theorem, it suffices for us to assume that 
\begin{equation}\label{d_assump}
\delta \geq N^{-1/3}.
\end{equation}
Recall the definition of the exit time $\tau$ above Theorem \ref{exit_time}, and we start with a simple bound
\begin{align}
&\mathbb{P}\Big(\log Z^\rho_{-1, {\bf v}_N} - \Big( \log I^\rho_{[\![(-1,-1), (0,-1) ]\!]} + \log Z_{0, {\bf v}_N}\Big)  \leq \delta N^{1/3}\Big)\nonumber\\
& \leq \mathbb{P}\Big(\log Z^\rho_{-1, {\bf v}_N} (1\leq \tau \leq N^{2/3}) - \Big(\log I^\rho_{[\![(-1,-1), (0,-1) ]\!]} +  \log Z_{0, {\bf v}_N} \Big) \leq \delta N^{1/3}\Big)\nonumber\\
& \leq \mathbb{P}\Big(\max_{1\leq k \leq N^{2/3}} \log Z^\rho_{-1, \mathbf{v}_N} (\tau = k) - \Big(\log I^\rho_{[\![(-1,-1), (0,-1) ]\!]} +  \log Z_{0, {\bf v}_N} \Big) \leq \delta N^{1/3}\Big)\label{max_est}
.\end{align}
For each $k =1, \dots, N^{2/3}$, let us denote the term inside our maximum as 
\begin{align*}
 S_k  &= \log Z^{\rho}_{-1, {\bf v}_N}(\tau = k) - \Big(\log I^\rho_{[\![(-1,-1), (0,-1) ]\!]} +  \log Z_{0, {\bf v}_N} \Big) \\
&=   (\log Z_{(k, 0), {\bf v}_N} - \log Z_{0, {\bf v}_N}) + \sum_{i=1}^k I^\rho_{_{[\![(i-1,-1), (i, -1)]\!]}}.
\end{align*}

Then, our estimate can be written as a running maximum. 
\begin{equation}\label{max_est2}
\eqref{max_est} = \mathbb{P} \Big(\max_{1\leq k \leq N^{2/3}} S_k \leq \delta N^{1/3}\Big) .
\end{equation}
The steps of $S_k$ are not i.i.d.~because of the term $\log Z_{(k, 0), {\bf v}_N} - \log Z_{0, {\bf v}_N}.$ Our next step is to lower bound $S_k$ by a random walk with i.i.d. steps using Theorem \ref{rwrw}.
In the application of Theorem \ref{rwrw}, we will rotate our picture $180^\circ$ so the path $\Theta_{1+N^{2/3}}$ is the segment $[\![(0,0), (N^{2/3}, 0)]\!]$. And our perturbed parameter will be 
$\lambda =  \rho + q_0 sN^{-1/3}$ where 
$s = |\log \delta|$. Note our condition $\delta \geq N^{-1/3}$ verifies the assumption $s\leq a_0 N^{2/3}$ from Theorem \ref{rwrw}. 

Let us denote the lower bounding i.i.d.~random walk as $\wt S_k$, and the distribution of the steps of $\wt S_k $ is given by the independent sum $-\log(\textup{Ga}^{-1}(\mu - \lambda)) +\log(\textup{Ga}^{-1}(\mu-\rho))$. Define the event 
$$A = \{ S_k \geq \log \tfrac{9}{10} + \wt S_k \text{ for $k = 0, 1, \dots, N^{2/3}$}\}.$$
Continuing from \eqref{max_est2}, we have 
\begin{align*}
\eqref{max_est2} &\leq \mathbb{P} \Big(\Big\{\max_{0\leq k \leq  N^{2/3}} S_k \leq \delta N^{1/3}\Big\} \cap A\Big) + \mathbb{P}(A^c)\\
&\leq \mathbb{P} \Big(\Big\{\ \max_{0\leq k \leq  N^{2/3}} \wt{S}_k  \leq \delta N^{1/3} + \log \tfrac{10}{9}\Big\} \Big) + \mathbb{P}(A^c).
\end{align*}
By Theorem \ref{rwrw}, we know $\mathbb{P}(A^c) \leq e^{-Cs^3} \leq \delta$. It remains to upper bound the running maximum 
\begin{equation}\label{rw_last}
\mathbb{P} \Big(\max_{0\leq k \leq  N^{2/3}} \wt{S}_k  \leq \delta N^{1/3} + \log \tfrac{10}{9}\Big).
\end{equation}
Lastly, Proposition \ref{rwest} gives $\eqref{rw_last} \leq C|\log \delta|\delta$. With this, we have finished the proof of the theorem.

\end{proof}

Our nonrandom fluctuation lower bound follows directly from Theorem \ref{statiid_low}.

\begin{proof}[Proof of Theorem \ref{nr_lb}]
By Theorem \ref{statiid_low}, there exists $\delta_0, N_0$ such that for all $N\geq N_0$, we have
$$\mathbb{P}\Big(\log Z^\rho_{-1, 2N\boldsymbol{\xi}[\rho]} - \Big(\log I^\rho_{[\![(-1,-1), (0, -1)]\!]} + \log Z_{0, 2N\boldsymbol{\xi}[\rho]} \Big) \geq \tfrac{1}{2}\delta_0 N^{1/3}\Big) \geq 0.99.$$
Let us denote the above event as $A$, and let 
$$X = \log Z^\rho_{-1, 2N\boldsymbol{\xi}[\rho]} - \Big(\log I^\rho_{[\![(-1,-1), (0, -1)]\!]} + \log Z_{0, 2N\boldsymbol{\xi}[\rho]} \Big).$$
Note that $X\geq 0$. Using \eqref{define_f} and \eqref{expect_stat}, we have
$$\Big|\mathbb{E}\big[ Z^\rho_{-1, 2N\boldsymbol{\xi}[\rho]}\big] - 2Nf(\rho)\Big| \leq 10\max\{|\Phi_0(\rho)|, |\Phi_0(\mu - \rho)|\}.$$ 
With these, we have
\begin{align*}
&2Nf(\rho) - \Big(\mathbb{E}[\log I^\rho_{[\![(-1,-1), (0, -1)]\!]}] + \mathbb{E}[\log Z_{0, 2N\boldsymbol{\xi}[\rho]}]\Big)\\
&  \geq  \mathbb{E}[X]  - 10\max\{|\Phi_0(\rho)|, |\Phi_0(\mu - \rho)|\}\\
  &=  \mathbb{E}[ X \mathbbm{1}_A] + \mathbb{E}[ X \mathbbm{1}_{A^c}] - 10\max\{|\Phi_0(\rho)|, |\Phi_0(\mu - \rho)|\}\\
 & \geq \tfrac{1}{10} \delta_0 N^{1/3} - 10\max\{|\Phi_0(\rho)|, |\Phi_0(\mu - \rho)|\}.
\end{align*}
The calculation above translates to our desired lower bound 
$$2Nf(\rho) -  \mathbb{E}[\log Z_{0, 2N\boldsymbol{\xi}[\rho]}] \geq \tfrac{1}{10} \delta_0 N^{1/3} - 10\max\{|\Phi_0(\rho)|, |\Phi_0(\mu - \rho)|\} - |\mathbb{E}[\log I^\rho_{[\![(-1,-1), (0, -1)]\!]}]| .$$ 
provided that $N_0$ is fixed sufficiently large depending on $\epsilon$ (recall $\rho \in [\epsilon, \mu- \epsilon]$). 
\end{proof}

\section{Moderate deviation bounds for the left tail}\label{sec_tail}

\subsection{Upper bound for the left tail}

Without introducing a new notation, let us assume that we are working with the version of the partition function $Z$ which includes the weight at the starting point. Once the upper bound for the left tail is proved for this version of $Z$, by a union bound it easily implies that the same result holds for the partition function which does not include the starting weight.

\begin{proof}[Proof of Proposition \ref{low_ub}]
For simplicity of the notation, let us denote 
$${\bf v}_N  = 2N{\boldsymbol\xi}[\rho].$$

We will first show the upper bound in the theorem for $t$ in the range 
\begin{equation}\label{first_t}
t_0 \leq t \leq  a_0 N^{2/3}
\end{equation}
for some positive $a_0$ which we will fix during the proof.
Because of Theorem \ref{stat_up_ub}, it suffices for us to show that 
\begin{equation}\label{stat_compare}
\mathbb{P}\Big(\frac{ Z^{\rho}_{-1, {\bf v}_N}}{  Z_{0, {\bf v}_N} }\geq e^{C'tN^{1/3}}\Big) \leq e^{-Ct^{3/2}}.
\end{equation}
We start the estimate  
\begin{align}
\textup{left side of  }\eqref{stat_compare} &= \mathbb{P}\Big(\frac{ Z^{\rho}_{-1, {\bf v}_N}(|\tau| \leq  \sqrt t N^{2/3})}{ Z_{0, {\bf v}_N} \cdot Q^\rho_{-1,{\bf v}_N} \{|\tau|\leq  \sqrt t N^{2/3}\}}\geq e^{C'tN^{1/3}}\Big)\nonumber\\
& \leq \mathbb{P}\Big(\frac{ Z^{\rho}_{-1, {\bf v}_N}(|\tau| \leq  \sqrt t N^{2/3})}{ Z_{0, {\bf v}_N}}\geq \tfrac{1}{2}e^{C'tN^{1/3}}\Big) \label{we_est}\\& \qquad \qquad  + \mathbb{P} (Q^\rho_{-1, {\bf v}_N}\{|\tau|\leq \sqrt t N^{2/3}\} \leq 1/2). \label{no_need} 
\end{align}
By Theorem \ref{exit_time}, the  probability in \eqref{no_need} above is bounded by $e^{-Ct^{3/2}}$. Also note that by a union bound, 
$$  \eqref{we_est} \leq \mathbb{P}\Big(\frac{ Z^{\rho}_{-1, {\bf v}_N}(1 \leq \tau \leq \sqrt t N^{2/3})}{ Z_{0, {\bf v}_N}}\geq \tfrac{1}{4}e^{C'tN^{1/3}}\Big) + \mathbb{P}\Big(\frac{ Z^{\rho}_{-1, {\bf v}_N}(-\sqrt t N^{2/3} \leq -\tau \leq -1)}{ Z_{0, {\bf v}_N}}\geq \tfrac{1}{4}e^{C'tN^{1/3}}\Big).$$
The estimate for these two terms are similar, so we work with 
\begin{equation}\label{we_est1}
\mathbb{P}\Big(\frac{ Z^{\rho}_{-1, {\bf v}_N}(1 \leq \tau \leq \sqrt t N^{2/3})}{ Z_{0, {\bf v}_N}}\geq \tfrac{1}{4}e^{C'tN^{1/3}}\Big) 
\end{equation}

As the numerator appearing in the probability measure of  \eqref{we_est1} can be bounded as follows:
$$Z^{\rho}_{-1, {\bf v}_N}(1 \leq \tau \leq \sqrt t N^{2/3}) \leq \max_{1\leq k \leq \sqrt t N^{2/3}} Z^{\rho}_{-1, {\bf v}_N}(\tau = k) + 100\log N,$$
to get \eqref{we_est1}, we will upper bound the probability 
\begin{equation}\label{we_est2}
\mathbb{P}\Big(\frac{ \max_{1\leq k \leq \sqrt t N^{2/3}} Z^{\rho}_{-1, {\bf v}_N}(\tau = k)}{ Z_{0, {\bf v}_N}}\geq \tfrac{1}{5}e^{C'tN^{1/3}}\Big). 
\end{equation}

For any fixed $k =1, \dots, \sqrt{t}N^{2/3}$, let us denote
\begin{align*}
&\log Z^{\rho}_{-1, {\bf v}_N}(\tau = k) - \log Z_{0, {\bf v}_N} \\
&= \log I^\rho_{[\![(-1,-1), (0, -1)]\!]} + \Big(  (\log Z_{(k, 0), {\bf v}_N}- \log Z_{0, {\bf v}_N} ) + \sum_{i=1}^k I^\rho_{_{[\![(i-1,-1), (i, -1)]\!]}} \Big) \\
&= \log I^\rho_{[\![(-1,-1), (0, -1)]\!]} + S_k .
\end{align*}
By a union bound
\begin{equation}\label{max_walk2}
\eqref{we_est2} \leq  \mathbb{P} \Big(\max_{1\leq k \leq \sqrt{t}N^{2/3}} S_k \geq \tfrac{C'}{2}tN^{1/3}\Big) + \mathbb{P} \Big(\log I_{[\![(-1,-1), (0, -1)]\!]}^\rho  \geq \tfrac{C'}{10}tN^{1/3}\Big) .
\end{equation}
The second probability decays as $e^{-CtN^{1/3}}$, so it remains to bound the first probability with the running maximum.

The steps of $S_k$ are not i.i.d.\ because of the term $\log Z_{(k, 0), {\bf v}_N}- \log Z_{0, {\bf v}_N} $, but we can upper bound $S_k$ with a random walk with i.i.d.\ steps with exponentially high probability, using Theorem \ref{rwrw}.
In the application of Theorem \ref{rwrw}, we will rotate our picture $180^\circ$ so the path $\Theta_{1+\sqrt{t}N^{2/3}}$ is the horizontal segment $[\![(0,0), (\sqrt{t}N^{2/3}, 0)]\!]$. Our perturbed parameter will be 
$\eta =  \rho - q_0 \sqrt{t}N^{-1/3}$. Let us denote the upper bounding i.i.d.~random walk as $\wt S_k$, and the distribution of the steps of $\wt S_k $ is given by the independent sum $-\log(\textup{Ga}^{-1}(\mu - \eta)) +\log(\textup{Ga}^{-1}(\mu-\rho))$. Define the event 
$$A = \{ S_k \leq \log \tfrac{10}{9} + \wt S_k \text{ for $k = 0, 1, \dots, \sqrt{t}N^{2/3}$}\}.$$
Continuing from \eqref{max_walk2}, we have 
\begin{align*}
\eqref{max_walk2} &\leq \mathbb{P} \Big(\Big\{\max_{0\leq k \leq \sqrt t N^{2/3}} S_k \geq \tfrac{C'}{2}tN^{1/3}\Big\} \cap A\Big) + \mathbb{P}(A^c)\\
&\leq \mathbb{P} \Big(\Big\{\ \max_{0\leq k \leq \sqrt t N^{2/3}} \wt{S}_k  \geq \tfrac{C'}{3}tN^{1/3}\Big\} \Big) + \mathbb{P}(A^c).
\end{align*}
By Theorem \ref{rwrw}, we know $\mathbb{P}(A^c) \leq e^{-Ct^{3/2}}$. It remains to upper bound the running maximum 
\begin{equation}\label{rw_before}
\mathbb{P} \Big(\max_{0\leq k \leq \sqrt t N^{2/3}} \wt{S}_k  \geq \tfrac{C'}{3}tN^{1/3} \Big).
\end{equation}
This is now a classic random walk bound, and it is actually a special case of \eqref{rw_before_E} where $a = t^{1/4}N^{1/3}$. Thus, \eqref{rw_before} is upper bounded by $e^{-Ct^{3/2}}$. With this, we have finished the proof for the case $t_0 \leq t \leq a_0 N^{2/3}$.  

Next,  we generalize the range of $t$ from \eqref{first_t}.  First, we extend the range to $t_0 \leq t \leq \alpha N^{2/3}$ for any large positive  $\alpha \geq a_0$. To see this, suppose $t = z N^{2/3}$ for $z\in [a_0, \alpha]$. Then, $\tfrac{a_0}{z} t$ satisfies our previous assumption \eqref{first_t}.
Hence, we have for each $t_0 \leq t \leq \alpha N^{2/3}$,
$$\mathbb{P}(\log Z_{0,  {\bf v}_N} - 2Nf(\rho) \leq -tN^{1/3}) \leq \mathbb{P}(\log Z_{0,  {\bf v}_N} - 2Nf(\rho) \leq -(\tfrac{a_0}{z} t)N^{1/3}) \leq e^{-C(\tfrac{a_0}{\alpha})^2 t^{3/2}}.$$

Finally, we will show that for a $\alpha$ sufficiently large, and for $t\geq \alpha N^{2/3}$
$$\mathbb{P}(\log Z_{0,  {\bf v}_N} - 2Nf(\rho) \leq -tN^{1/3}) \leq e^{-CtN^{1/3}}.$$
Let us define $t = zN^{2/3}$ where $z\geq \alpha$. Then, we may replace the free energy with the sum of weights along a single path $\gamma \in \mathbb{X}_{0, {\bf v}_N}$, which has a smaller value. Then, fix $\alpha$ sufficiently large, for $z\geq \alpha$, we have  
\begin{equation}
\begin{aligned}\label{LDP}
\mathbb{P}(\log Z_{0, {\bf v}_N} - 2Nf(\rho) \leq -tN^{1/3}) &\leq \mathbb{P}\Big(  \sum_{i=1}^{2N} \log Y_{\gamma_i}  \leq -\tfrac{1}{2}zN\Big)\\
& = \mathbb{P}\Big(  \sum_{i=1}^{2N} \log Y^{-1}_{\gamma_i}  \geq \tfrac{1}{2}zN\Big) \\
&\leq e^{-CzN} = e^{-CtN^{1/3}},
\end{aligned}
\end{equation}
where the last inequality follows Theorem \ref{max_sub_exp}. With this, we have finished the proof of our theorem. 
\end{proof}

\subsection{Lower bound for the left tail}

The approach we employ here follows from the idea of Theorem 4 in \cite{bootstrap}, which proves the optimal lower bound in the setting of last-passage percolation. The same idea was adapted to the O'Connell-Yor polymer in \cite{OC_tail}.
To start, we have the following proposition. 
\begin{proposition}\label{first_lb}
Let $\rho \in (0, \mu)$. There exist positive constants $\newc\label{first_lb_c1}, \newc\label{first_lb_c2}, N_0$ such that for each $N \geq N_0$, we have
$$
\mathbb{P}(\log Z_{0, 2N\boldsymbol{\xi}[\rho]}  - 2Nf_d \leq -\oldc{first_lb_c1} N^{1/3}) \geq \oldc{first_lb_c2}.$$
\end{proposition}
\begin{proof}
This follows directly from Proposition \ref{time_const1} which says $2Nf_d \geq 2Nf(\rho)$, and Theorem \ref{nr_lb} which says $2Nf(\rho) \geq \mathbb{E}[\log Z_{0, 2N\boldsymbol{\xi}[\rho]}] + CN^{1/3}$. Note the probability of $\{\log Z_{0, 2N\boldsymbol{\xi}[\rho]}  - \mathbb{E}[\log Z_{0, 2N\boldsymbol{\xi}[\rho]}] \leq 0\}$ is {bounded uniformly from below because of the lower bound of the right tail in Proposition \ref{up_lb}}. Then, on the event $\{\log Z_{0, 2N\boldsymbol{\xi}[\rho]}  - \mathbb{E}[\log Z_{0, 2N\boldsymbol{\xi}[\rho]}] \leq 0\}$, we have $\{\log Z_{0, 2N\boldsymbol{\xi}[\rho]}  - 2Nf_d  \leq -CN^{1/3}\}$.
\end{proof}

Using a step-back argument, we obtain an interval to interval lower bound.
\begin{proposition}\label{iti_bound}
There exist positive constants $\newc\label{iti_bound_c1}, \newc\label{iti_bound_c2}, \eta, N_0$ such that for each $N\geq N_0$ and each integer $h\in [-\eta N^{1/3}, \eta N^{1/3}]$,  we have 
$$\mathbb{P}\Big(\log Z^{\textup{max}}_{\mathcal{L}_{0}^{N^{2/3}}, \mathcal{L}^{N^{2/3}}_{(N-2hN^{2/3}, N+2hN^{2/3})}}  - 2N{f_d}\leq -\oldc{iti_bound_c1} N^{1/3}\Big) \geq \oldc{iti_bound_c2}.$$
\end{proposition}

\begin{proof}
For simplicity of the notation, let us denote $$J^{h} = \mathcal{L}^{N^{2/3}}_{(N-2hN^{2/3}, N+2hN^{2/3})} \quad \textup{ and } \quad I = \mathcal{L}_{0}^{N^{2/3}}.$$

The proof uses a step-back argument.
For any $\epsilon >0$, let us first define $I^\epsilon = \epsilon^{2/3}I$. We may increase the cutoff $N_0$ depending on $\epsilon$ so that $I^\epsilon$ is non-empty.  We cover $I$ by a sequence of shifted $I^\epsilon$'s, i.e. $$I \subset \bigcup_{i=-K}^K I^\epsilon_i $$
where $I^\epsilon_i = (-2i(\epsilon N)^{2/3}, 2i (\epsilon N)^{2/3}) + I^\epsilon$ and $K = \floor{1/\epsilon}+ 1$. We do the same for $J^h$ and obtain the collection $\{J^{h, \epsilon}_j\}_{j=-K}^K$.
Next, we will show that for each pair $i, j \in [\![-K, K]\!]$, there exists $c, c'$ such that 
\begin{equation}\label{e_bound}
\mathbb{P}(\log Z^{\textup{max}}_{I^\epsilon_i, J^{h, \epsilon}_j}  - 2Nf(\rho)\leq -c N^{1/3}) \geq c'.
\end{equation}
Let us define ${\bf u}^*\in I^\epsilon_i$ and ${\bf v}^*\in J^{h, \epsilon}_j$ be the pair of points such that 
$$Z_{{\bf u}^*, {\bf v}^*} = Z^{\textup{max}}_{I^\epsilon_i, J^{h, \epsilon}_j}.$$
And let us denote the midpoints of $I^\epsilon_i$ and $J^{h, \epsilon}_j$ as $\wt{\bf a}$ and $\wt{\bf b}$.

Next, we define the step back points ${\bf a} = \wt{\bf a} - \epsilon (N,N)$ and ${\bf b} = \wt{\bf b} + \epsilon(N,N)$. With these new endpoints, we have
\begin{equation}\label{step_back}
\log  Z_{{\bf a}, {\bf b}}  \geq \log Z_{{\bf a}, {\bf u}^* }  + \log Z^{\textup{max}}_{I^\epsilon_i, J^{h, \epsilon}_j}  + \log Z_{{\bf v}^*, {\bf b}} .
\end{equation}

Let us look at the term $\log  Z_{{\bf a}, {\bf b}} $ on the left side. By Proposition \ref{time_const1},   we have
$$\Lambda({\bf b} - {\bf a}) \leq  2(N+2\epsilon N)f_d.$$
By Proposition \ref{first_lb}, we know there exists an event $A$ with $\mathbb{P}(A) \geq \oldc{first_lb_c2}$ such that on the event $A$, we have
\begin{equation}\label{ab_bound}
\log  Z_{{\bf a}, {\bf b}}   - 2(N+2\epsilon N)f_d \leq - \oldc{first_lb_c1} (N+2\epsilon N)^{1/3}.
\end{equation}
Next, we show that on a high probability event $B$ with $\mathbb{P}(B) \geq 1- \oldc{first_lb_c2}/2$, we have
\begin{equation}\label{pti}
\log Z_{{\bf a}, {\bf u}^* }  + \log Z_{{\bf v}^*, {\bf b}}  - 4\epsilon N f(\rho) \geq - \tfrac{\oldc{first_lb_c1}}{2} N^{1/3}.
\end{equation}
Once we have these, on the event $A\cap B$ which has probability at least $\oldc{first_lb_c2}/2$, estimates \eqref{step_back}, \eqref{ab_bound} and \eqref{pti} will imply
$$\log Z^{\textup{max}}_{I^\epsilon_i, J^{h, \epsilon}_j}  -  2Nf_d  \leq - \tfrac{\oldc{first_lb_c1}}{2} N^{1/3},$$
which is the statement in \eqref{e_bound}.

By symmetry, we will work with the term $\log Z_{{\bf a}, {\bf u}^* } $. By Theorem \ref{high_inf}, 
$$\mathbb{P}(\log Z_{{\bf a}, {\bf u}^* }  - 2 \epsilon N f_d < - M(\epsilon N)^{1/3}) \leq e^{-CM} \leq \tfrac{\oldc{first_lb_c2}}{10} $$
provided $M$ is fixed sufficiently large. Let $B_1$ denote the complement of the event above, and let $B_2$ be the similar event defined for $\log Z_{{\bf v}^*, {\bf b} }$. We define $B = B_1 \cap B_2$, and $\mathbb{P}(B) \geq 1- \tfrac{\oldc{first_lb_c2}}{2}.$
Let us fix $\epsilon$ sufficiently small so that $M \epsilon^{1/3} \leq \tfrac{\oldc{first_lb_c1}}{2}$. With this, we have shown \eqref{pti}, thus finishing the proof for \eqref{e_bound}.

Finally, to prove the proposition, note 
$$\{\log Z^{\textup{max}}_{I, J^h} - 2Nf_d \leq -c N^{1/3}\} \supset \bigcap_{i,j = -K}^K \Big\{\log(Z^{\textup{max}}_{I^\epsilon_i, J^{h, \epsilon}_j}) - 2Nf_d \leq -c N^{1/3}\Big \}.$$
By the FKG inequality 
$$\mathbb{P}\Big(\bigcap_{i,j = -K}^K \Big\{\log Z^{\textup{max}}_{I^\epsilon_i, J^{h, \epsilon}_j}  - 2Nf_d \leq -c N^{1/3}\Big \}\Big) \geq \prod_{i, j = -K}^K\mathbb{P}\Big(\log Z^{\textup{max}}_{I^\epsilon_i, J^{h, \epsilon}_j} - 2Nf(\rho)\leq -c N^{1/3}\Big),$$
and \eqref{pti} says each term inside the product is lower bounded by some positive $c'$. Hence, we obtain that 
$$\mathbb{P}\Big(\log Z^{\textup{max}}_{I, J^h}  - 2Nf(\rho)\leq -c N^{1/3}\Big) \geq (c')^{(2/\epsilon)^2} = \oldc{iti_bound_c2}$$
and we have finished the proof of this proposition.
\end{proof}

Using the FKG inequality, we will further improve our lower bound to the following. 
\begin{proposition}\label{itl_bound}
There exist positive constants $\newc\label{itl_bound_c1}, \newc\label{itl_bound_c2}, N_0$ such that for all $N\geq N_0$ 
$$\mathbb{P}\Big(\log Z^{\textup{max}}_{\mathcal{L}_0^{N^{2/3}}, \mathcal{L}_{N}}  - 2Nf_d\leq -\oldc{itl_bound_c1} N^{1/3}\Big) \geq \oldc{itl_bound_c2}.$$
\end{proposition}

\begin{proof}
For simplicity of the notation, let us denote $$J^{h} = \mathcal{L}^{N^{2/3}}_{(N-2hN^{2/3}, N+2hN^{2/3})} \quad \textup{ and } \quad I = \mathcal{L}_{0}^{N^{2/3}}.$$

The main idea is to cover the line  $\mathcal{L}$ by $J^h$ for $h \in \mathbb{Z}$. For some large fixed $h_0$ which will be chosen later, we then split the possible values of $h$ into two parts $[\![- h_0, h_0]\!]$ and $\mathbb{Z}\setminus [\![- h_0, h_0]\!]$. For $h\in [\![- h_0, h_0]\!]$ we use the FKG inequality and the lower bound from Proposition \ref{iti_bound}. 
On the other hand, for $h\in \mathbb{Z}\setminus [\![- h_0, h_0]\!]$, Proposition \ref{trans_fluc_loss} show that the probability is actually exponentially high, i.e. 
$$\mathbb{P}(\log Z^{\textup{max}}_{I, J^h}  - 2Nf(\rho)\leq -c N^{1/3}) \geq 1- e^{-C|h|^3},$$
provided $c$ is sufficiently small. 
Thus, we have the lower bound
$$\mathbb{P}(\log Z^{\textup{max}}_{I, \mathcal{L}}  - 2N{f_d}\leq -c N^{1/3}) \geq \oldc{iti_bound_c2}^{100h_0}\prod_{|h|= h_0}^\infty (1-e^{-C|h|^3})  = \oldc{itl_bound_c2},$$
where $\oldc{iti_bound_c2}$ is the probability lower bound from Proposition \ref{iti_bound}. With this, we have finished the proof of this proposition.
\end{proof}

We prove a lower bound for the constrained free energy.

\begin{proposition}\label{con_lb}
There exists constants $\newc\label{con_lb_c1}, \newc\label{con_lb_c2} N_0, t_0, a_0$ such that for each $N\geq N_0$, $t_0 \leq t \leq a_0N^{2/3}/(\log N)^2$ and $0< l \leq N^{1/3}$, we have
$$\mathbb{P}(\log Z^{\text{in}, l N^{2/3}}_{0, N}  - 2Nf_d \leq - \oldc{con_lb_c2}t N^{1/3}) \geq e^{-\oldc{con_lb_c2}lt^{5/2}}.$$
\end{proposition}
\begin{proof}
Using diagonal and anti-diagonal lines, we cut the rectangle $R_{0, N}^{l N^{2/3}}$ into smaller rectangles with diagonal $\ell^\infty$-length $N/t^{3/2}$ and anti-diagonal $\ell^\infty$-length $(N/t^{3/2})^{2/3}$,  Let us denote these small rectangles as $R(u, v)$ where the index $u = 1, 2, \dots, t^{3/2}$ indicates the anti-diagonal level, and $v =  1, 2, \dots, lt$ enumerates the rectangle inside the same anti-diagonal level. Recall the notation $\overline{R(u,v)} $ and  $\underline{R(u,v)}$ denote the upper and lower anti-diagonal sides of $R(u,v)$. Let us also finally define $\mathcal{L}(u)$ to denote the anti-diagonal line which contains $\overline{R(u,v)} $.

Let us define the event 
$$A  = \bigcap_{u, v} \Big\{\log Z_{\underline {R(u, v)}, \mathcal{L}(u)} -  2(N/t^{3/2})f_d \leq -\oldc{itl_bound_c1} (N/t^{3/2})^{1/3}  \Big\},$$
where the constant $\oldc{itl_bound_c1}$ is from  Proposition \ref{itl_bound}. By the FKG inequality and Proposition \ref{itl_bound},  we know $\mathbb{P}(A) \geq e^{-Cl t^{5/2}}$.

Next, we see that our constrained free energy can be upper-bounded as follows. 
\begin{align*}
\log Z^{\text{in}, l N^{2/3}}_{0, N}  &\leq t^{3/2} \Big(\log (tl) + \max_{u,v} \log Z_{\underline {R(u, v)}, \mathcal{L}(u)}\Big)\\
\text{ on the event $A$} \quad & \leq t^{3/2} \Big(\log (tl) + 2(N/t^{3/2}) f_d -\oldc{itl_bound_c1}(N/t^{3/2})^{1/3}\Big)\\
&\leq 2N f_d  -  \oldc{itl_bound_c1}tN^{1/3} + t^{3/2}\log (tl).
\end{align*}
Finally, fix $a_0$ sufficiently small; the assumption $t \leq a_0 N^{2/3}/ (\log N)^2$ implies that 
$\tfrac{1}{2} \oldc{itl_bound_c1}tN^{1/3} \geq  t^{3/2}\log (tl).$
With this, we have shown that 
$$\log Z^{\text{in}, l N^{2/3}}_{0, N}  - 2Nf_d \leq - \tfrac{1}{2}\oldc{itl_bound_c1}t  N^{1/3} \qquad \text{ on the event $A$}.,$$
hence, finished the proof.
\end{proof}

Finally, we prove Proposition \ref{ptp_low}. 

\begin{proof}[Proof of Proposition \ref{ptp_low}]
This follows from the FKG inequality, Proposition \ref{con_lb} and Theorem \ref{trans_fluc_loss3}. 
Set the parameter $l = \sqrt {t}$ in Proposition \ref{con_lb}, then,
\begin{align*}
&\mathbb{P}(\log Z_{0, N}  - 2Nf_d \leq - C't N^{1/3}) \\
&= \mathbb{P}\Big(\Big\{\log Z^{\text{in}, \sqrt{t}N^{2/3}}_{0, N}  - 2Nf_d \leq - C't N^{1/3}) \Big\}\cap \Big\{\log Z^{\text{exit}, \sqrt{t}N^{2/3}}_{0, N}  - 2Nf_d \leq - C't N^{1/3}) \Big\}\Big)\\
& \geq \mathbb{P}\Big(\log Z^{\text{in}, \sqrt{t}N^{2/3}}_{0, N}  - 2Nf_d \leq - C't N^{1/3}) \Big) \mathbb{P}\Big(\log Z^{\text{exit}, \sqrt{t}N^{2/3}}_{0, N}  - 2Nf_d \leq - C't N^{1/3}) \Big)\\
&\geq e^{-Ct^3} \cdot (1-e^{-Ct^{3/2}}).
\end{align*}
provided that $C'$ is fixed sufficiently small.
\end{proof}

\appendix
\section{Proofs of the free energy estimates of Section \ref{free_est}} \label{appen}

\subsection{Free energy and path fluctuations}\label{free fluc}

\subsubsection{Proof of Proposition \ref{trans_fluc_loss} }

\begin{proof}[Proof of Proposition \ref{trans_fluc_loss}]

For this proof, let us define
$$I = \mathcal{L}_{0}^{N^{2/3}} \qquad J^h = \mathcal{L}_{(N-2hN^{2/3}, N+ 2hN^{2/3})}^{N^{2/3}}.$$ 

Without the loss of generality, we will assume $h\in \mathbb{Z}_{\geq 0}$ and it satisfies
\begin{equation}\label{h_range}
0\leq h \leq \frac{1}{2} N^{1/3},
\end{equation}
since $Z_{I, J^h}$ is $0$ otherwise. We also note that {by the maximum bound introduced in Section \ref{max_bd}}, it suffices to prove this estimate for $Z^\textup{max}_{I, J^h}$, since 
$$\log Z_{I, J^h}  \leq 100\log N  + \log Z^\textup{max}_{I, J^h}.$$

\begin{figure}
\begin{center}

\tikzset{every picture/.style={line width=0.75pt}} 

\begin{tikzpicture}[x=0.75pt,y=0.75pt,yscale=-1,xscale=1]

\draw  [dash pattern={on 0.84pt off 2.51pt}]  (299.67,130.17) -- (289.67,170.17) ;
\draw  [dash pattern={on 0.84pt off 2.51pt}]  (289.67,170.17) -- (220.67,240.17) ;
\draw  [dash pattern={on 0.84pt off 2.51pt}]  (299.67,130.17) -- (370.67,58.17) ;
\draw  [fill={rgb, 255:red, 0; green, 0; blue, 0 }  ,fill opacity=1 ] (217.13,240.17) .. controls (217.13,238.21) and (218.71,236.63) .. (220.67,236.63) .. controls (222.62,236.63) and (224.21,238.21) .. (224.21,240.17) .. controls (224.21,242.12) and (222.62,243.71) .. (220.67,243.71) .. controls (218.71,243.71) and (217.13,242.12) .. (217.13,240.17) -- cycle ;
\draw  [fill={rgb, 255:red, 0; green, 0; blue, 0 }  ,fill opacity=1 ] (367.13,58.17) .. controls (367.13,56.21) and (368.71,54.63) .. (370.67,54.63) .. controls (372.62,54.63) and (374.21,56.21) .. (374.21,58.17) .. controls (374.21,60.12) and (372.62,61.71) .. (370.67,61.71) .. controls (368.71,61.71) and (367.13,60.12) .. (367.13,58.17) -- cycle ;
\draw  [fill={rgb, 255:red, 0; green, 0; blue, 0 }  ,fill opacity=1 ] (286.13,170.17) .. controls (286.13,168.21) and (287.71,166.63) .. (289.67,166.63) .. controls (291.62,166.63) and (293.21,168.21) .. (293.21,170.17) .. controls (293.21,172.12) and (291.62,173.71) .. (289.67,173.71) .. controls (287.71,173.71) and (286.13,172.12) .. (286.13,170.17) -- cycle ;
\draw  [fill={rgb, 255:red, 0; green, 0; blue, 0 }  ,fill opacity=1 ] (296.13,130.17) .. controls (296.13,128.21) and (297.71,126.63) .. (299.67,126.63) .. controls (301.62,126.63) and (303.21,128.21) .. (303.21,130.17) .. controls (303.21,132.12) and (301.62,133.71) .. (299.67,133.71) .. controls (297.71,133.71) and (296.13,132.12) .. (296.13,130.17) -- cycle ;
\draw    (269.33,149.58) -- (310,190.75) ;
\draw    (279.33,109.58) -- (320,150.75) ;

\draw (351,44.4) node [anchor=north west][inner sep=0.75pt]    {$\mathbf{b}$};
\draw (201.13,233.57) node [anchor=north west][inner sep=0.75pt]    {$\mathbf{a}$};
\draw (284,177.4) node [anchor=north west][inner sep=0.75pt]    {$\wt{\mathbf{a}}$};
\draw (293,104.4) node [anchor=north west][inner sep=0.75pt]    {$\wt{\mathbf{b}}$};
\draw (258,142.4) node [anchor=north west][inner sep=0.75pt]    {$I$};
\draw (320,142.4) node [anchor=north west][inner sep=0.75pt]    {$J^{h}$};

\end{tikzpicture}
\captionsetup{width=0.8\textwidth}
\caption{ An illustration of the points $\wt {\bf a}, \wt {\bf b}, {\bf a},  {\bf b}$ from the step back argument. }\label{show_ab}
\end{center}
\end{figure}
Next, we describe the step back argument, the points $\wt {\bf a}, \wt {\bf b}, {\bf a},  {\bf b}$ below are illustrated in Figure \ref{show_ab}.
Let $\wt {\bf a}$ and $\wt {\bf b}$ denote the midpoints of $I$ and $J^h$. Let us define the step back points ${\bf a} = \wt{\bf a}-w_0(N, N)$ and ${\bf b} = \wt{\bf b} + w_0(N, N)$ where $w_0$ is a constant that we will fix later. Let us use ${\bf u}^* \in I$ and ${\bf v}^* \in J^h$ to denote the random points such that $Z^{\textup{max}}_{I, J^h} = Z_{{\bf u}^*, {\bf v}^*}$. Then, we have 
\begin{equation}\label{step_back2}
\log Z_{{\bf a}, {\bf b}} \geq \log Z_{{\bf a}, {\bf u}^* } + \log Z^{\textup{max}}_{I, J^{h}}  + \log Z_{{\bf v}^*, {\bf b}}.
\end{equation}


Since $|{\bf b} - {\bf a}|_1 = 2(1+2w_0)N$, 
we rewrite the vector ${\bf b}- {\bf a}$ as $2(1+2w_0)N{\boldsymbol\xi}[\mu/2+z_{{\bf a}, {\bf b}}]$  for some nonnegative constant $z_{{\bf a}, {\bf b}}$.
Note the perpendicular $\ell^1$-distance from ${\bf b}- {\bf a}$ to the diagonal line is 
$$({\bf b}- {\bf a}) \cdot ({\bf e}_1 - {\bf e}_2) = (\wt{\bf b}- \wt{\bf a}) \cdot ({\bf e}_1 - {\bf e}_2),$$
which is the same as the $\ell^1$-distance from $\wt{\bf b} - \wt{\bf a}$ to the diagonal line. For each fixed $h$ in our range \eqref{h_range}, it holds that
$$(\wt{\bf b}- \wt{\bf a}) \cdot ({\bf e}_1 - {\bf e}_2) = 2hN^{2/3}.$$ 
From the perpendicular distance to the diagonal, we see that
\begin{equation}\label{slope_int}
\textup{slope of ${\bf b}-{\bf a}$ } = \tfrac{2(1+2w_0)N + hN^{2/3}}{2(1+2w_0)N - hN^{2/3}}   = 1 + \tfrac{2h}{2(1+2w_0)N^{1/3} - h}.
\end{equation}
Because of the upper bound $h\leq \tfrac{1}{2}N^{1/3}$ from \eqref{h_range}, we can choose $w_0$ and $N_0$ to be large enough so  that for all $N\geq N_0$ and $0 \leq h \leq \tfrac{1}{2}N^{1/3}$, the slope in \eqref{slope_int} is contained inside the interval $[1-\epsilon, 1+ \epsilon]$ from Proposition \ref{slope_z}. Then, by Proposition \ref{slope_z} and possibly increasing the value $w_0$ if necessarily, there exists a positive constant $C$ such that
\begin{equation}\label{relate_z_h}
 z_{{\bf a}, {\bf b}}\in \Big[ \tfrac{1}{C} \tfrac{2h}{2(1+2w_0)N^{1/3} - h}, C \tfrac{2h}{2(1+2w_0)N^{1/3} - h}\Big]
\end{equation}
where the constant $C$ above is independent of $h$ as long as $0 \leq h\leq \tfrac{1}{2}N^{1/3}$.  

Subsequently, if required, we can further increase $w_0$ so that the interval from \eqref{relate_z_h} falls within the small interval $[-\epsilon, \epsilon]$ from Proposition \ref{time_const}. Then, by leveraging Proposition \ref{time_const} and further increasing $w_0$ if necessary, we obtain the first inequality below, while the second inequality is derived from \eqref{relate_z_h}.
\begin{equation}\label{curv2}
2(1+2w_0)N\big[f(\mu/2+z_{{\bf a}, {\bf b}}) - {f_d}\big] \leq (1+2w_0)N\big[-Cz_{{\bf a}, {\bf b}}^2\big] \leq - \wt{C} h^2 N^{1/3}.
\end{equation}
Note that $w_0$ has now been fixed, so we absorb it into $\wt{C}$.

Finally, let us upper bound the probability stated in the proposition, 
\begin{align}
&\mathbb{P} \Big(\log Z^{\textup{max}}_{I, J^h}  - 2N{f_d} \geq (-C h^2 + t) N^{1/3}\Big)\nonumber\\
&\leq \mathbb{P}\Big(\log Z_{{\bf a}, {\bf b}} - \log Z_{{\bf a}, {\bf u}^* } - \log Z_{{\bf v}^*, {\bf b}}  - 2N{f_d} \geq (-C h^2 +t)N^{1/3}\Big)\nonumber\\
&= \mathbb{P}\Big([\log  Z_{{\bf a}, {\bf b}} -  (2+4w_0)N{f_d}]  - [\log Z_{{\bf a}, {\bf u}^*}  - 2w_0N{f_d}] \nonumber\\
&\qquad \qquad \qquad \qquad - [\log Z_{{\bf v}^*, {\bf b}} -2w_0N{f_d}]\geq (-C h^2+t) N^{1/3}\Big)\nonumber\\
\textup{ by \eqref{curv2} } \quad &\leq \mathbb{P}\Big([\log  Z_{{\bf a}, {\bf b}}  -  (2+4w_0)Nf(\mu/2+z_{{\bf a}, {\bf b}})]  - [\log(Z_{{\bf a}, {\bf u}^* }) - 2w_0N{f_d}] \nonumber\\
&\qquad \qquad \qquad \qquad - [\log Z_{{\bf v}^*, {\bf b}} -2w_0N{f_d}]\geq  ((\wt{C} -C )h^2 + t)N^{1/3} \Big)\nonumber\\
&\leq \mathbb{P}\Big(\log Z_{{\bf a}, {\bf b}} -  2(1+2w_0)Nf(\mu/2+z_{{\bf a}, {\bf b}}) \geq  \tfrac{1}{3}((\wt{C} -C )h^2 + t)N^{1/3}\Big) \label{one}\\
& \qquad + \mathbb{P}\Big( \log Z_{{\bf a}, {\bf u}^* } - 2w_0N{f_d} \leq  -\tfrac{1}{3}((\wt{C} -C )h^2 + t)N^{1/3} \Big) \label{two}\\
& \qquad + \mathbb{P}\Big(\log Z_{{\bf v}^* , {\bf b}}  - 2w_0N{f_d} \leq  -\tfrac{1}{3}((\wt{C} -C )h^2 + t)N^{1/3} \Big)\label{three}
\end{align}
Fix $C$ small so that $\wt C- C>0$ in the expression above, we will now show that all three probabilities decay faster than $e^{-C(|h|^3+ \min\{t^{3/2}, tN^{1/3}\})}$.
The term $\eqref{one} \leq e^{-C(|h|^3+ \min\{t^{3/2}, tN^{1/3}\})}$ follows from Proposition \ref{up_ub}.
For the remaining two other terms, by symmetry, their estimates are the same. Let us work with \eqref{two}. 
Since $u^*$ only depends on the edges between the lines $\mathcal{L}_0$ and $\mathcal{L}_N$, then by Proposition \ref{low_ub}
\begin{align*}
&\mathbb{P}\Big(\log Z_{{\bf a}, {\bf u}^* }  - 2w_0 N{f_d} < -\tfrac{1}{3}((\wt{C} -C_1 )h^2 + t) N^{1/3}\Big) \\
&\leq \sup_{{\bf u}\in I} \mathbb{P}\Big(\log Z_{{\bf a}, {\bf u} } - 2w_0 N{f_d} < -\tfrac{1}{3}((\wt{C} -C_1 )h^2 + t) N^{1/3}\Big)\\
& \leq e^{-C(|h|^3 + \min\{t^{3/2}, tN^{1/3}\})}.
\end{align*}
With this, we have finished the proof of this theorem.
\end{proof}

\subsubsection{Proof of proposition \ref{trans_fluc_loss2}}

\begin{proof}[Proof of Proposition \ref{trans_fluc_loss2}]
For this proof, let us define
$$I^k = \mathcal{L}_{(-2kN^{2/3}, 2kN^{2/3})}^{N^{2/3}} \qquad J^h = \mathcal{L}_{(N-2hN^{2/3}, N+ 2hN^{2/3})}^{N^{2/3}}.$$ 

{Because $\mathcal{L}^{sN^{2/3}}_0$ and $ \mathcal{L}_{0}\setminus\mathcal{L}_{0}^{(s+t)N^{2/3}}$ are disjoint}, then the number of points between $\mathcal{L}^{sN^{2/3}}_0$ and $ \mathcal{L}_{N}\setminus\mathcal{L}_{N}^{(s+t)N^{2/3}}$ which are connected by directed paths is at most $N^{100}$. We may work with the maximum version of the free energy as
$$\log Z_{\mathcal{L}^{sN^{2/3}}_0, \mathcal{L}_{N}\setminus\mathcal{L}_{N}^{(s+t)N^{2/3}}} \leq \log Z^\textup{max}_{\mathcal{L}^{sN^{2/3}}_0, \mathcal{L}_{N}\setminus\mathcal{L}_{N}^{(s+t)N^{2/3}}} + 100 \log N.$$

By translation invariance and Proposition \ref{trans_fluc_loss}, 
\begin{align*}
&\mathbb{P}(\log Z_{I^{-k}, J^h}  - 2N{f_d} \geq -\oldc{trans_fluc_loss_c1}(h+k)^2N^{1/3})\\ &= \mathbb{P}(\log Z_{I^{0}, J^{h+k}}  - 2N{f_d} \geq -\oldc{trans_fluc_loss_c1}(h+k)^2N^{1/3})\\
& \leq e^{-\oldc{trans_fluc_loss_c2}(h+k)^3}.
\end{align*}
Then, the following union bound finishes the proof
\begin{align*}
&\mathbb{P}\Big(\log\Big(Z^{\textup{max}}_{\mathcal{L}^{sN^{2/3}}_0, \mathcal{L}_{N}\setminus\mathcal{L}_{N}^{(s+t)N^{2/3}}}\Big)  - 2N{f_d} \geq -\tfrac{\oldc{trans_fluc_loss_c1}}{10}t^2N^{1/3}\Big)\\
& \leq \mathbb{P}\Big(\bigcup_{\substack{h\geq s+t\\k \leq s}}\{\log(Z^\textup{max}_{I^{k}, J^{h}}) - 2N{f_d} \geq -\tfrac{\oldc{trans_fluc_loss_c1}}{5}(h+k)^2N^{1/3}\}\Big)\\
& \qquad \qquad \qquad + \mathbb{P}\Big(\bigcup_{\substack{h\geq s+t\\k \leq s}}\{\log(Z^\textup{max}_{I^{-k}, J^{-h}}) - 2N{f_d} \geq -\tfrac{\oldc{trans_fluc_loss_c1}}{5}(h+k)^2N^{1/3}\}\Big)\\
&\leq \sum_{\substack{h, k \geq t}} e^{-C(h+k)^3} \leq e^{-Ct^3}.
\end{align*}
\end{proof}

\subsubsection{Proof of Theorem \ref{trans_fluc_loss3}}

We start with the following proposition which states that the restricted partition function obtained by summing over paths with high fluctuation at the halfway time will be much smaller than typical.
Fix any $s\in \mathbb{Z}_{\geq 0}$. Let $Z^{\textup{mid}, (s+t)N^{2/3}}_{\mathcal{L}_0^{sN^{2/3}}, \mathcal{L}_N^{sN^{2/3}}}$ denote the partition function which sums over all paths between $\mathcal{L}_0^{sN^{2/3}}$ and $\mathcal{L}_N^{sN^{2/3}}$, and they all avoid the segment $\mathcal{L}^{(s+t)N^{2/3}}_{N/2}$.
\begin{proposition}\label{trans_mid}
There exist positive constants $\newc\label{trans_mid_c1}, \newc\label{trans_mid_c2},   N_0, t_0$ such that for each $N\geq N_0$, $t\geq t_0$  and $s\geq 0$ we have
$$\mathbb{P}\Big(\log Z^{\textup{mid},   (s+t)N^{2/3}}_{\mathcal{L}_0^{sN^{2/3}}, \mathcal{L}_N^{sN^{2/3}}} - 2N{f_d} \geq -\oldc{trans_mid_c1}t^2N^{1/3}\Big) \leq e^{-\oldc{trans_mid_c2}t^{3}}.$$
\end{proposition}

\begin{proof}
Let $G$ denote the segment 
$$G = \mathcal{L}_{N/2} \setminus \mathcal{L}^{(s+t)N^{2/3}}_{N/2}.$$
This follows directly from subadditivity 
$$ \log  Z^{\textup{mid},   (s+t)N^{2/3}}_{\mathcal{L}_0^{sN^{2/3}}, \mathcal{L}_N^{sN^{2/3}}} \leq  \log Z_{\mathcal{L}_0^{sN^{2/3}}, G} + \log Z_{G , \mathcal{L}_N^{sN^{2/3}}}$$
and applying Proposition \ref{trans_fluc_loss2} to $\log Z_{\mathcal{L}_0^{sN^{2/3}}, G}$ and $ \log Z_{G , \mathcal{L}_N^{sN^{2/3}}}$.
\end{proof}

\begin{figure}[t]
\begin{center}
\begin{tikzpicture}[x=0.75pt,y=0.75pt,yscale=-0.8,xscale=0.8]

\draw    (240,171) -- (340,271) ;
\draw    (320,100) -- (420,200) ;
\draw    (231,81) -- (420,279.5) ;
\draw  [dash pattern={on 4.5pt off 4.5pt}]  (320,100) -- (240,171) ;
\draw  [color={rgb, 255:red, 155; green, 155; blue, 155 }  ,draw opacity=1 ][dash pattern={on 2.53pt off 3.02pt}][line width=2.25] [line join = round][line cap = round] (221,252.5) .. controls (227.57,249.22) and (230.73,241.47) .. (235,235.5) .. controls (249.53,215.15) and (262,191.49) .. (259,165.5) .. controls (257.48,152.31) and (233.05,125.77) .. (228,113.5) .. controls (222.7,100.62) and (206.81,63.63) .. (221,53.5) .. controls (231.72,45.84) and (244.4,48.2) .. (255,53.5) .. controls (285.87,68.93) and (292.79,86.86) .. (313,113.5) .. controls (317.69,119.68) and (324.63,124.9) .. (332,127.5) .. controls (356.14,136.02) and (387.32,113.87) .. (411,107.5) .. controls (422.49,104.41) and (434.33,102.83) .. (446,100.5) ;

\draw (224,169.4) node [anchor=north west][inner sep=0.75pt]    {$a$};
\draw (326,81.4) node [anchor=north west][inner sep=0.75pt]    {$b$};
\draw (290,128.4) node [anchor=north west][inner sep=0.75pt]    {$c$};
\draw (226,61.4) node [anchor=north west][inner sep=0.75pt]    {$d$};
\draw (216,271.4) node [anchor=north west][inner sep=0.75pt]    {$\mathcal{L}^{(s+t\sum_{i=0}^{k-1}{2^{-i/5}})N^{2/3}}_{(l-1)N/2^{k+1}}$};
\draw (410,204.4) node [anchor=north west][inner sep=0.75pt]    {$\mathcal{L}^{(s+t\sum_{i=0}^{k-1}{2^{-i/5}})N^{2/3}}_{(l+1)N/2^{k+1}}$};
\draw (437,272.4) node [anchor=north west][inner sep=0.75pt]    {$\mathcal{L}^{(s+ t\sum_{i=0}^{k}{2^{-i/5}})N^{2/3}}_{lN/2^{k+1}}$};

\end{tikzpicture}
\end{center}
\captionsetup{width=0.8\textwidth}
\caption{The $\ell^\infty$-distance between $a$ and $b$ is $N/2^{k+1}$ and the $\ell^\infty$-distance between $c$ and $d$ is $2^{-k/5}tN^{2/3}$. If $|a-b|_\infty < |c-d|_\infty$, then there would not exist a directed path (shown in gray) which goes through $\mathcal{L}^{(s+t\sum_{i=0}^{k-1}{2^{-i/5}})N^{2/3}}_{(l-1)N/2^{k+1}}$ and $\mathcal{L}^{(s+t\sum_{i=0}^{k-1}{2^{-i/5}})N^{2/3}}_{(l+1)N/2^{k+1}}$ while avoiding $\mathcal{L}^{(s+t\sum_{i=0}^{k}{2^{-i/5}})N^{2/3}}_{lN/2^{k+1}}$.} 
\label{high_fluc}
\end{figure}
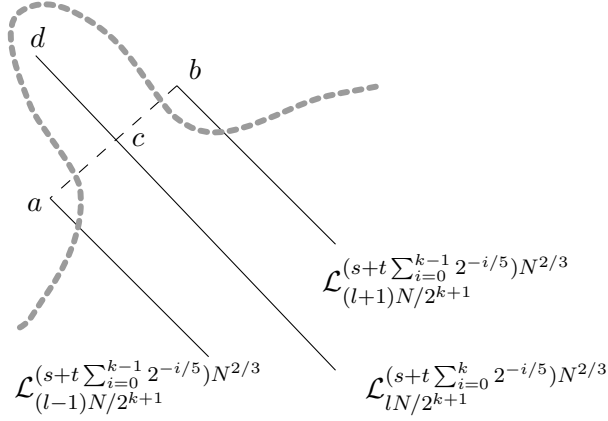

Next, we are ready to prove  Theorem \ref{trans_fluc_loss3}.
\begin{proof}[Proof of Theorem \ref{trans_fluc_loss3}]
First, let us rewrite the collection of directed paths that exit the rectangle $R_{0, N}^{(s+t\sum_{i=0}^\infty {2^{-i/5}})N^{2/3}} $
as the following disjoint union. 
For each $j\in \mathbb{Z}_{\geq 0}$, let $T_j$ denote the collection of paths between $\mathcal{L}_0^{sN^{2/3}}$ and $\mathcal{L}_{N}^{sN^{2/3}}$ that avoids at least one of the segments  $\mathcal{L}^{(s+t\sum_{i=0}^j{2^{-i/5}})N^{2/3}}_{lN/2^{j+1}}$ where  $l \in \{ 1 \dots, 2^{j+1}-1\}$. The set complement in \eqref{define_A1} below is taken over the collection of paths between $\mathcal{L}_0^{sN^{2/3}}$ and $\mathcal{L}_{N}^{sN^{2/3}}$. Our disjoint union will be
\begin{equation}\label{define_A1}
A_0 = T_0, \quad  A_1 = T_0^c \cap T_1, \quad A_2 = T_1^c \cap T_2, \quad  \dots, \quad  A_{k_0} = T_{k_0-1}^c \cap T_{k_0}
\end{equation}
where $k_0$ is the smallest index such that if $k > k_0$, $A_k$ will be an empty set. This is demonstrated in Figure \ref{high_fluc}. As seen from the figure, the following upper bound holds for $k_0$,
\begin{equation}\label{asm_k0}
2^{4k_0/5} \leq N^{1/3}/t.
\end{equation}

Using this decomposition of the paths, we have
\begin{align*}
\log Z^{\textup{exit}, { (s+t\sum_{i=0}^\infty {2^{-i/5}})N^{2/3}}}_{\mathcal{L}_0^{sN^{2/3}}, \mathcal{L}_N^{sN^{2/3}}}
&= \log\Big(\sum_{k=0}^{k_0}Z_{\mathcal{L}_0^{sN^{2/3}}, \mathcal{L}_N^{sN^{2/3}}}(A_k)\Big)\\
&\leq \log(k_0) + \max_{0\leq k \leq k_0} \{ \log Z_{\mathcal{L}_0^{sN^{2/3}}, \mathcal{L}_N^{sN^{2/3}}}(A_k)\}\\
&\leq \log N  + \max_{0\leq k \leq k_0} \{ \log Z_{\mathcal{L}_0^{sN^{2/3}}, \mathcal{L}_N^{sN^{2/3}}}(A_k) \}.
\end{align*}
Since our estimate is on the scale $N^{1/3}$, we may ignore the $\log N$ term above. Now, it suffices for us to upper bound bound
\begin{align}
&\mathbb{P}\Big(\max_{0\leq k \leq k_0} \big\{ \log Z_{\mathcal{L}_0^{sN^{2/3}}, \mathcal{L}_N^{sN^{2/3}}}(A_k) \big\} - 2N{f_d} \geq -\wt Ct^2N^{1/3}\Big)\label{goal} \\
&\leq \sum_{k=0}^{k_0} \mathbb{P}\Big(\log Z_{\mathcal{L}_0^{sN^{2/3}}, \mathcal{L}_N^{sN^{2/3}}}(A_k)  - 2N{f_d} \geq -\wt  Ct^2N^{1/3}\Big)\nonumber
\end{align}

Next, let us upper bound each term inside the sum above,
\begin{equation} \label{2_decay}
\mathbb{P}\Big(\log Z_{\mathcal{L}_0^{sN^{2/3}}, \mathcal{L}_N^{sN^{2/3}}}(A_k) - 2N{f_d} \geq -\wt Ct^2N^{1/3}\Big).
\end{equation}
Define $U_k = \{2m-1: m = 1, \dots 2^k\}$. For $l\in U_k$, we can write $A_k$ as a (non-disjoint) union of $A_k^l$ where $A_k^l$ contains the collection of paths between $\mathcal{L}_0^{sN^{2/3}}$ and $\mathcal{L}_N^{sN^{2/3}}$ that
go through the segments $\mathcal{L}^{(s+t\sum_{i=0}^{k-1}{2^{-i/5}})N^{2/3}}_{(l-1)N/2^{k+1}}$ and  $\mathcal{L}^{(s+t\sum_{i=0}^{k-1}{2^{-i/5}})N^{2/3}}_{(l+1)N/2^{k+1}}$ while avoiding the segment $\mathcal{L}^{(s+t\sum_{i=0}^{k}{2^{-i/5}})N^{2/3}}_{ lN/2^{k+1}}$ in between.
Then, we have
\begin{align}
\eqref{2_decay} &\leq \mathbb{P}\Big(\log \Big (\sum_{l\in U_k} Z_{\mathcal{L}_0^{sN^{2/3}}, \mathcal{L}_N^{sN^{2/3}}}(A_k^l)\Big) - 2N{f_d} \geq -\wt Ct^2N^{1/3}\Big) \nonumber\\
&\leq \mathbb{P}\Big(\log(2^{k_0}) + \max_{l\in U_k} \Big\{ \log Z_{\mathcal{L}_0^{sN^{2/3}}, \mathcal{L}_N^{sN^{2/3}}}(A_k^l)\Big\} - 2N{f_d} \geq -\wt Ct^2N^{1/3}\Big)\nonumber \\
&\leq \mathbb{P}\Big( \max_{l\in U_k} \Big\{ \log Z_{\mathcal{L}_0^{sN^{2/3}}, \mathcal{L}_N^{sN^{2/3}}}(A_k^l)\Big\} - 2N{f_d} \geq -2\wt Ct^2N^{1/3}\Big) \quad \textup{by \eqref{asm_k0}}  \nonumber \\
&\leq \sum_{l\in U_k} \mathbb{P}\Big( \log Z_{\mathcal{L}_0^{sN^{2/3}}, \mathcal{L}_N^{sN^{2/3}}}(A_k^l)  - 2N{f_d} \geq -2\wt Ct^2N^{1/3}\Big).\label{the_sum}
\end{align}
Again, let us look at the probability inside the sum \eqref{the_sum}. First, note we have the following upper bound 
\begin{align}
\log Z_{\mathcal{L}_0^{sN^{2/3}}, \mathcal{L}_N^{sN^{2/3}}}(A_k^l) &\leq \log Z_{\mathcal{L}_0^{sN^{2/3}}, \mathcal{L}_{(l-1)N/2^{k+1}}}  \label{ptl1} \\
&\qquad + \log Z^{\textup{ mid}, (s+t\sum_{i=0}^{k}{2^{-i/5}})N^{2/3}}_{\mathcal{L}^{(s+t\sum_{i=0}^{k-1}{2^{-i/5}})N^{2/3}}_{(l-1)N/2^{k+1}}, \mathcal{L}^{(s+t\sum_{i=0}^{k-1}{2^{-i/5}})N^{2/3}}_{(l+1)N/2^{k+1}}} \label{high1}\\
&\qquad \qquad +  \log Z_{\mathcal{L}_{(l+1)N/2^{k+1}}, \mathcal{L}_N^{sN^{2/3}}} . \label{ptl2}
\end{align}
Now, note for \eqref{high1}, the transversal fluctuation of the paths between $\mathcal{L}^{(s+t\sum_{i=0}^{k-1}{2^{-i/5}})N^{2/3}}_{(l-1)N/2^{k+1}}$ and $\mathcal{L}^{(s+t\sum_{i=0}^{k-1}{2^{-i/5}})N^{2/3}}_{(l+1)N/2^{k+1}}$ is more than 
$$2^{-k/5}tN^{2/3} = \frac{2^{2k/3}}{2^{k/5}}t(N/2^{k})^{2/3}.$$ 
Thus, by Proposition \ref{trans_mid}, for some positive constants $C'$ and $C''$,
\begin{equation}\label{high_bound}
\mathbb{P}\Big(\eqref{high1} - 2\frac{N}{2^k}{f_d}> -C' \Big(\frac{2^{2k/3}}{2^{k/5}}t\Big)^2(N/2^{k})^{1/3}\Big) \leq e^{-C''\big(\frac{2^{2k/3}}{2^{k/5}}t\big)^3}.
\end{equation}
And note that 
$$\Big(\frac{2^{2k/3}}{2^{k/5}}t\Big)^2(N/2^{k})^{1/3} = 2^{3k/5}t^2N^{1/3}$$
With this, we may upper bound the probability inside the sum \eqref{the_sum}  as 
\begin{align}
&\mathbb{P}\Big( \log Z_{\mathcal{L}_0^{sN^{2/3}}, \mathcal{L}_N^{sN^{2/3}}}(A_k^l) - 2N{f_d} \geq -{2}\wt{C}t^2N^{1/3}\Big)\nonumber\\
& \qquad \leq  \mathbb{P}\Big(\log Z^{\textup{ mid}, (s+t\sum_{i=0}^{k}{2^{-i/5}})N^{2/3}}_{\mathcal{L}^{(s+t\sum_{i=0}^{k-1}{2^{-i/5}})N^{2/3}}_{(l-1)N/2^{k+1}}, \mathcal{L}^{(s+t\sum_{i=0}^{k-1}{2^{-i/5}})N^{2/3}}_{(l+1)N/2^{k+1}}}  - 2\frac{N}{2^{k}}{f_d} \geq - C' 2^{3k/5}t^2N^{1/3} \Big)\label{high_bound2}\\
& \;\;\quad \qquad + \mathbb{P}\Big(\log Z_{\mathcal{L}_0^{sN^{2/3}}, \mathcal{L}_{{(l-1)N}/{2^{k+1}}}}  + \log\ Z_{\mathcal{L}_{{(l+1)N}/{2^{k+1}}}, \mathcal{L}_N^{sN^{2/3}}}  \nonumber\\ 
& \qquad \qquad \qquad \qquad\qquad \qquad - \Big(2N - 2\frac{N}{2^{k}}\Big){f_d} \geq  {C'}2^{3k/5}t^2N^{1/3} - 2\wt Ct^2N^{1/3}\Big)\label{two_ptl}
\end{align}
Note that we have seen in \eqref{high_bound} that $\eqref{high_bound2} \leq e^{-C''2^{k/100}t^3}$. To bound \eqref{two_ptl}, note that by lowering the value of $\wt C$ if needed, 
$${C'}2^{3k/5}t^2N^{1/3} - 2\wt Ct^2N^{1/3} \geq \frac{1}{2} {C'}2^{3k/5} t^2 N^{1/3},$$
then the event in \eqref{two_ptl} should be rare because the free energy is unusually large. By a union bound
\begin{align}
\eqref{two_ptl}
& \leq \mathbb{P}\Big(\log Z_{\mathcal{L}_0^{sN^{2/3}}, \mathcal{L}_{{(l-1)N}/{2^{k+1}}}}   - 2\tfrac{(l-1)N}{2^{k+1}}{f_d} \geq  \tfrac{1}{4}C'2^{3k/5}t^2N^{1/3}\Big)\label{we_bd}\\
& \qquad \qquad+\mathbb{P}\Big( \log Z_{\mathcal{L}_{{(l+1)N}/{2^{k+1}}}, \mathcal{L}_N^{sN^{2/3}}}   - (2N - 2\tfrac{(l+1)N}{2^{k+1}}){f_d} \geq  \tfrac{1}{4}C'2^{3k/5}t^2N^{1/3}\Big).\nonumber
\end{align}
By symmetry, let us bound \eqref{we_bd} above. For simplicity of the notation, let $M =  \frac{(l-1)N}{2^{k+1}}$, then 
\begin{align*}
\eqref{we_bd} = \mathbb{P}\Big(\log Z_{\mathcal{L}_0^{s(N/M)^{2/3}M^{2/3}}, \mathcal{L}_{M}}   - 2M{f_d} \geq  \tfrac{C'2^{3k/5}t^2N^{1/3}}{4M^{1/3}}M^{1/3}\Big)
\end{align*}
To upper bound this term, we would like to apply the interval-to-line bound from Theorem \ref{ptl_upper}. The only assumption from Theorem \ref{ptl_upper} that we need to verify here is the width of the interval can not be too wide. A sufficient bound that guarantees the assumption is
$$s(N/M)^{2/3} \leq e^{\tfrac{C'2^{3k/5}t^2N^{1/3}}{4M^{1/3}}}.$$
The inequality above holds by our assumption that $s \leq e^{t}.$
Thus, we obtain
$$\eqref{we_bd} \leq e^{-C\min\big\{\big(\tfrac{2^{3k/5}t^2N^{1/3}}{M^{1/3}}\big)^{3/2}, \tfrac{2^{3k/5}t^2N^{1/3}}{M^{1/3}} M^{1/3}\big\}} \leq e^{-C2^{k/100}t^3}.$$ 

To summarize this last part, we have shown that 
$$\eqref{2_decay} \leq  \sum_{l\in U_k} e^{-C2^{k/100}t^{3}} + e^{-C2^{k/100}t^{3}} \leq 2^k\cdot e^{-{C}2^{t/100}t^{3}}.$$
And going back to our goal \eqref{goal}, we have shown 
$$\mathbb{P}\Big(\max_{0\leq k \leq k_0}  \log Z_{0, N}(A_k)) - 2N{f_d} \geq -Ct^2N^{1/3}\Big) \leq \sum_{k=0}^{k_0} 2^k\cdot e^{-{C}2^{k/100}t^{3}} \leq e^{-Ct^{3}}.$$
With this, we have finished the proof of this theorem.
\end{proof}

\subsubsection{Proof of Corollary \ref{trans_fluc_loss4}}
\begin{proof}
Because of the choice $s<t/10$, we see that 
$$R_{0, N}^{(s+\tfrac{t}{2})N^{2/3}} \subset R_{(-sN^{2/3}, sN^{2/3}), N}^{tN^{2/3}}.$$
Then, we have the following bound for the free energy 
$$\log Z^{\textup{exit},{tN^{2/3}}}_{(-sN^{2/3}, sN^{2/3}), N}  \leq \log Z^{\textup{exit}, {(s+\frac{t}{2})N^{2/3}}}_{\mathcal{L}_{0}^{sN^{2/3}}, \mathcal{L}_N^{sN^{2/3}}} .$$
Our corollary follows directly from Theorem \ref{trans_fluc_loss4} when applied to the right side above.
\end{proof}

\subsection{Interval-to-line free energy}

We start with a point-to-line bound.
\begin{proposition}\label{lem_ptl}
There exist positive constants $\newc\label{lem_ptl_c1}, N_0$ such that for each $N \geq N_0$ and each $t\geq 1$, we have 
$$\mathbb{P} \Big( \log Z_{0,\mathcal{L}_N}  - 2N{f_d} \geq tN^{1/3}\Big) \leq 
e^{-\oldc{lem_ptl_c1} \min\{t^{3/2}, tN^{1/3}\}}.  
$$
\end{proposition}
\begin{proof}
Note it suffices to prove the same estimate for 
\begin{equation}\label{replace_max}
\mathbb{P}(\log Z^{\textup{max}}_{0,\mathcal{L}_N}  - 2N{f_d} \geq tN^{1/3})
\end{equation}
since 
$\log Z_{0,\mathcal{L}_N}   \leq \log(Z^\textup{ max}_{0,\mathcal{L}_N}) + 100\log N.$
Let  $J^h = \mathcal{L}_{(N-2hN^{2/3}, N+ 2hN^{2/3})}^{N^{2/3}}$. By a union bound and Proposition \ref{trans_fluc_loss}, we have
$$\eqref{replace_max}\leq \sum_{h\in \mathbb{Z}} \mathbb{P}(\log Z^{\textup{max}}_{0,J^h}  - 2N{f_d} \geq tN^{1/3}) \leq \sum_{h\in \mathbb{Z}} e^{-C(|h|^3 + \min\{t^{3/2}, tN^{1/3}\})} \leq e^{-C \min\{t^{3/2}, tN^{1/3}\}}.$$
\end{proof}

Next, we use a step-back argument to upgrade the point-to-line bound of Proposition \ref{lem_ptl} to Theorem \ref{ptl_upper}.

\begin{proof}[Proof of Theorem \ref{ptl_upper}]
First, we prove the case when $h=1$. Since 
$$\log Z_{\mathcal{L}_0^{N^{2/3}},\mathcal{L}_N}  \leq \max_{{\bf p}\in \mathcal{L}_{0}^{N^{2/3}}} \log Z_{{\bf p }, \mathcal{L}_N}  + 100\log N,$$
it suffices to work with the maximum above. Let ${\bf p}^*$ denote the random maximizer that  
$$\max_{{\bf p}\in \mathcal{L}_{0}^{N^{2/3}}} \log Z_{{\bf p }, \mathcal{L}_N}= \log Z_{{\bf p }^*, \mathcal{L}_N}.$$ 
Then, we have
$$\log Z_{-N, \mathcal{L}_N}  \geq \log Z_{-N, {\bf p}^*}  + \log Z_{{\bf p}^*, \mathcal{L}_N} .$$
With this, we see that 
\begin{align}
&\mathbb{P}\Big( \max_{{\bf p}\in \mathcal{L}_{0}^{N^{2/3}}}\log Z_{{\bf p}, \mathcal{L}_N}  - 2N{f_d} \geq tN^{1/3}\Big)\nonumber\\
&\leq \mathbb{P}\Big( [\log Z_{-N, \mathcal{L}_N}  - 4N{f_d}] - [ \log Z_{-N, {\bf p}^*}  - 2N{f_d}] \geq tN^{1/3}\Big)\nonumber\\
&\leq \mathbb{P}\Big( \log Z_{-N, \mathcal{L}_N}  - 4N{f_d} \geq \tfrac{t}{2}N^{1/3}\Big) \label{eq_ptl} \\
&\qquad \qquad\qquad \qquad + \mathbb{P}\Big(  \log Z_{-N, {\bf p}^*}  - 2N{f_d} \leq -\tfrac{t}{2}N^{1/3}\Big)\label{eq_lowtail}.
\end{align}
From Proposition \ref{lem_ptl}, we obtain $\eqref{eq_ptl} \leq e^{-C\min\{t^{3/2}, tN^{1/3}\}}$. Because ${\bf p}^*$ only depends on weights between $\mathcal{L}_0$ and $\mathcal{L}_N$, and by Proposition \ref{low_ub}, we have 
\begin{equation}\label{h_one}
\eqref{eq_lowtail} \leq \max_{{\bf p} \in \mathcal{L}_{0}^{N^{2/3}}}  \mathbb{P}\Big(  \log Z_{-aN, {\bf p}} - aN{f_d} \leq -\tfrac{t}{2}N^{1/3}\Big) \leq e^{-\wt{C}\min\{t^{3/2}, tN^{1/3}\}}.
\end{equation}
This finishes the case when $h=1$.

Next, let us define
$I^j = \mathcal{L}_{(-2jN^{2/3}, + 2jN^{2/3})}^{N^{2/3}}$ where $j$ is the collection of integers in $[\![-h,h]\!]$
Then, it holds that
$$\log Z_{\mathcal{L}_0^{hN^{2/3}},\mathcal{L}_N}  \leq \max_{j\in [\![-h,h]\!]} \log Z_{I_j,\mathcal{L}_N}  + \log h.$$
Using this and a union bound, 
\begin{align*}
\mathbb{P}\Big(\log Z_{\mathcal{L}_0^{hN^{2/3}},\mathcal{L}_N} - 2N{f_d} \geq tN^{1/3}\Big) 
& \leq \mathbb{P}\Big( \max_{j\in [\![-h,h]\!]} \log Z_{I_j,\mathcal{L}_N}  - 2N{f_d} \geq \tfrac{t}{3}N^{1/3}\Big) \\
& \leq 10h \mathbb{P}\Big(\log Z_{I^0,\mathcal{L}_N}  - 2N{f_d} \geq \tfrac{t}{3}N^{1/3}\Big) \\
& \leq 10 e^{\oldc{ptl_upper_c1}\min \{t^{3/2}, tN^{1/3}\}} e^{-\wt C\min \{t^{3/2}, tN^{1/3}\}}\\
& \leq e^{-C\min\{t^{3/2}, tN^{1/3}\}}.
\end{align*}
where the last inequality holds if we fix $\oldc{ptl_upper_c1} \leq \tfrac{1}{2} \wt C$ where $\wt C$ is the constant appearing in \eqref{h_one}. With this, we have finished the proof of the theorem.
\end{proof}




\subsection{Estimates for the constrained free energy}

\subsubsection{Proof of Theorem \ref{wide_similar}}
\begin{proof}
First, we prove the estimate when 
$$t_0 \leq t \leq N^{2/3}.$$ 
To do this, we break the line segment from $(0,0)$ to ${\bf p}$ into equal pieces with $\ell^1$ length $2N\theta/\sqrt{t}$. And let us denote the endpoints in between as $\{{\bf p}_i\}$. 

Let $0< C' \leq 1/2$ which we will fix later. By a union bound, we have  
\begin{align}
&\mathbb{P}\Big(\log Z^{\textup{in}, {\theta N^{2/3}}}_{0, {\bf p}}  - 2N{f_d} \leq -C't^2N^{1/3}\Big)  \nonumber\\
& \leq \frac{\sqrt{t}}{\theta} \mathbb{P}\Big(\log Z^{\textup{in}, {\theta N^{2/3}}}_{0, {\bf p}_1} - 2(N \theta /\sqrt{t}){f_d} \leq -C't^{2/3}\theta^{2/3}(N\theta/\sqrt{t})^{1/3}\Big)\label{union_bd}
\end{align}
Using the fact that 
$$\log Z_{0, {\bf p}_1}  \leq \log 2 + \max\Big \{\log Z^{\textup{in}, {\theta N^{2/3}}}_{0, {\bf p}_1} , \log Z^{\textup{exit}, {\theta N^{2/3}}}_{0, {\bf p}_1}\Big\},$$
we may continue the bound 
\begin{align}
\eqref{union_bd} &\leq \frac{\sqrt{t}}{\theta} \Big[\mathbb{P}\Big(\log Z_{0, {\bf p}_1}  - 2(N \theta /\sqrt{t}){f_d} \leq -C' t^{2/3}\theta^{2/3}(N\theta/\sqrt{t})^{1/3} + \log 2\Big) \label{constrain_p1}\\
& \quad \quad + \mathbb{P}\Big(\log Z^{\textup{exit},{t^{1/3} \theta^{1/3} (\theta N/\sqrt{t}) ^{2/3}}}_{0, {\bf p}_1}  - 2(N \theta/\sqrt{t}){f_d} \geq -C'(t^{1/3}\theta^{1/3})^2(N\theta/\sqrt{t})^{1/3}\Big)\Big]. \label{constrain_p2}
\end{align}
It remains to upper bound each of the probabilities above and this would finish the proof of this theorem.

First, we show that the probability in \eqref{constrain_p1} is bounded by $e^{-C\theta t }$. There exists an absolute constant $a_0$ such that   
$$\Big|(N\theta/\sqrt{t}, N\theta/\sqrt{t}) - {\bf p}_1\Big|_\infty \leq \frac{a_0 \theta^{4/3}}{t^{1/6}}(N\theta/\sqrt{t})^{2/3}.$$
 Then,  by Proposition \ref{reg_shape}
$$\Big|\Lambda ({\bf p}_1) - 2(N\theta/\sqrt{t}) f_d\Big| \leq C\Big(\frac{a_0 \theta^{4/3}}{t^{1/6}}\Big)^2 (N\theta/\sqrt{t})^{1/3},$$
and the fraction $\frac{a_0 \theta^{4/3}}{t^{1/6}}$ is bounded for $a_0$ and $0<\theta\leq 100$. Hence, we may replace the $2(N\theta/\sqrt{t}) f_d$ in \eqref{constrain_p1} by $\Lambda({\bf p}_1)$ and Proposition \ref{low_ub} can be applied.

For the probability in \eqref{constrain_p2},  we may apply Theorem \ref{trans_fluc_loss3} and obtain
$$\mathbb{P}\Big(\log Z^{\textup{exit}, {t^{1/3}\theta^{1/3} (N\theta/\sqrt{t}) ^{2/3}}}_{0, {\bf p}_1}  - 2(N \theta /\sqrt{t}){f_d} \geq -C'(t^{1/3}\theta^{1/3})^2(N\theta/\sqrt{t})^{1/3}\Big) \leq e^{-C\theta t},$$
provided that $C'$ is fixed sufficiently small. Note here the assumption $t^{1/3}\theta^{1/3} \leq (N\theta/\sqrt{t})^{1/3}$ in Theorem \ref{trans_fluc_loss3} is satisfied because of our current assumption $t\leq N^{2/3}$. This finishes the proof of the estimate when $t_0 \leq t \leq  N^{2/3}$, as we have shown that the probability appearing in our theorem is upper bounded by $\tfrac{\sqrt{t}}{\theta}e^{-C\theta t}$. 

Finally, to generalize the range of $t$, the steps are exactly the same as how we generalized the range of $t$ in the proof of Proposition \ref{low_ub}. First, we can trivially generalize to the range $t_0 \leq t \leq \alpha N^{2/3}$ for any large positive $\alpha$. This only changed the $C$ from our upper bound $\tfrac{\sqrt{t}}{\theta}e^{-C\theta t}$. For $t\geq \alpha N^{2/3}$, we replace the constrained free energy $\log Z^{\textup{in}, {\theta N^{2/3}}}_{0, {\bf p}}$ by a sum of weights from a single deterministic path inside our parallelogram $R_{0, {\bf p}}^{\theta N^{2/3}}$.  Then, our estimate follows from Theorem \ref{max_sub_exp}, as shown in \eqref{LDP}.
\end{proof}

\subsubsection{Proof of Theorem \ref{c_up_lb}}
\begin{proof}
We may lower bound the constrained free energy by an i.i.d~sum
\begin{equation}\label{iidsum}
\log Z^{\textup{in}, {s N^{2/3}}}_{0, N}\geq  \sum_{i=1}^{ks^{-3/2}} \log Z^{\textup{in}, {s N^{2/3}}}_{(i-1)\tfrac{s^{3/2}N}{k}, i\tfrac{s^{3/2}N}{k}}.
\end{equation}
Note that 
\begin{align}
&\mathbb{P}\Big(\log Z^{\textup{in}, {s N^{2/3}}}_{0, \tfrac{s^{3/2}N}{k}}  - 2(s^{3/2}N/k){f_d} \geq (s^{3/2}N/k)^{1/3}\Big) \label{lb0}\\
& \geq \mathbb{P}\Big(\log Z_{0, \tfrac{s^{3/2}N}{k}}  - 2(s^{3/2}N/k){f_d} \geq  2 (s^{3/2}N/k)^{1/3}\Big) \label{lb1}\\
&\qquad \qquad \qquad  - \mathbb{P}\Big(\log Z^{\textup{exit},  {s N^{2/3}}}_{0, \tfrac{s^{3/2}N}{k}}  - 2(s^{3/2}N/k){f_d} \geq (s^{3/2}N/k)^{1/3}\Big).\label{lb2}
\end{align}
The probability $\eqref{lb1}$ is lower bounded by an absolute constant $c_0 \in (0,1)$ by Proposition \ref{up_lb}, provided  $s^{3/2}N/k \geq N_0^*$ where $N_0^*$ is the $N_0$ from Theorem \ref{c_up_lb}. 
And the probability \eqref{lb2} is upper bounded by $e^{-Ck}$ from Theorem \ref{trans_fluc_loss3} when $k \leq \sqrt{s} N^{1/3}$ and \eqref{lb2} is zero when $k >\sqrt{s} N^{1/3}$. Thus, $$\eqref{lb0} \geq c_0 - e^{-Cs^{3/2}t^{3/2}} > c_0/10$$
when $t$ is large.

Finally, let $k = \tfrac{1}{N_0^*}s^{3/2}t^{3/2}$. On the intersection of $k$ independent events that each term of the sum in $\eqref{iidsum}$ is large like in \eqref{lb0},
$$\mathbb{P}\Big(\log Z^{\textup{in},  {s N^{2/3}}}_{0, N}  - 2N{f_d} \geq \tfrac{1}{N_0^*} tN^{1/3}\Big)  \geq \big(c_0 - e^{-Cs^{3/2}t^{3/2}}\big)^{Ct^{3/2}} \geq e^{-Ct^{3/2}}$$
where the last constant $C$ depends on $s$.
With this, we have finished the proof of this theorem.
\end{proof}

\subsection{Minimum and Maximum of the constrained free energy in a box}

\subsubsection{Proof of Theorem \ref{high_inf}}
\begin{proof}
First, we will prove the following estimate, 
\begin{equation}\label{start_est}
\mathbb{P}\Big( \min_{{\bf p} \in R_{0, N/16}^{N^{2/3}}} \log Z^{\textup{in}, R_{0,N}^{N^{2/3}}}_{{\bf p}, N}  - (2N - |{\bf p}|_1){f_d} \leq - tN^{1/3}\Big)  \leq e^{-Ct}.
\end{equation}
Then, the statement of the theorem follows from a union bound, which we will show at the end of the proof. To start, we construct a tree $\mathcal{T}$ with the base at $(N, N)$. Define $\mathcal{T}_0 = \{(N,N)\}$ and we will define the remaining part of the tree. 
Fix a positive constant $J$ such that 
\begin{equation}\label{J_range}
N^{1/4}\leq N8^{-J} \leq N^{1/3 - 0.01},
\end{equation}
such $J$ always exists provided that $N_0$ is sufficiently large.
Next, for $j = 1,2, \dots, J$, $\mathcal{T}_j$ is the collection of $32^j$ vertices which we now define. 
For $i=1, 2, \dots, 8^j$, from each segment $\mathcal{L}_{\frac{2i+1}{32}8^{-j}N}^{N^{2/3}}$, we collect $4^j$ vertices which split the  the segment $\mathcal{L}_{\frac{2i+1}{32}8^{-j}N}^{N^{2/3}}$ into $4^j+1$ equal pieces. 
We define the vertices of our tree $\mathcal{T}$ as the union $\bigcup_{j=0}^J \mathcal{T}_j$. 

Now, we form the edges between the vertices. Let us label the vertices in $\mathcal{T}_j$ as
$$\{x_{(i,k)}^j: 1\leq i \leq 8^j, 1\leq k \leq 4^j\}.$$ A fixed index $i$ records the anti-diagonal segment  $\mathcal{L}_{\frac{2i+1}{32}8^{-j}N}^{N^{2/3}}$. And along this segment, we label the $4^j$ chosen vertices by their ${e}_2$-coordinate values with index $k = 1, \dots, 4^j$. For $k=1$, we choose $x^j_{(i,1)}$ to be the vertex with the smallest ${e}_2$-coordinate (which could be negative). 
Next, for $j=1, 2, \dots, J-1$, we connect the vertex ${\bf x}^{j}_{(i,k)}\in \mathcal{T}_j$ with the 32 vertices inside $\mathcal{T}_{j+1}$ which are of the form $x^{j+1}_{8(i-1) + i', 4(k-1) + j'}$ where $1\leq i' \leq 8$ and $1\leq j' \leq 4.$ This completes the construction of the tree $\mathcal{T}$.

For fixed $j, i$ and $k$, let us denote the collection of 32 points in $T_{j+1}$ which are connected to ${\bf x}^{j}_{(i,k)}$ as $V_{j,i,k}$. 
Now, for each ${\bf v} \in V_{j,i,k}$, the diagonal distance between the ${\bf v}$ and ${\bf x}^{j}_{(i,k)}$ satisfies 
\begin{equation}\label{dia_grid} |{\bf x}^{j}_{(i,k)}|_1 -  |{\bf v}|_1 \in \Big[2\frac{8^{-j}}{32}N, 6\frac{8^{-j}}{32}N\Big],
\end{equation}
and their anti-diagonal distance is upper bounded as
\begin{equation}\label{adia_grid}
\Big| {\bf x}^{j}_{(i,k)}\cdot ({\bf e}_1 - {\bf e}_2) -  {\bf v} \cdot ({\bf e}_1 - {\bf e}_2)\Big| \leq 2\cdot 4^{-j} N^{2/3}.
\end{equation}

Similarly, we look at the vertices inside $\mathcal{T}_J$, they form a grid inside the rectangle $R_{0, \frac{2+8^{-J}}{32}N}^{N^{2/3}}$ which contains $R_{0, N/16}^{N^{2/3}}$. Then, for each ${\bf p} \in R_{0, N/16}^{N^{2/3}}$, there exists an $x^J_{(i[{\bf p}], k[{\bf p}])} = x^J_{\bf p} \in \mathcal{T}^J$ such that 
\begin{equation}\label{dia_end}
|x^J_{\bf p}|_1 -  |{\bf p}|_1 \in \Big[2 \frac{8^{-J}}{32}N, 6\frac{8^{-J}}{32}N\Big]
\end{equation}
and 
\begin{equation}\label{adia_end}
\Big| x^J_{\bf p}\cdot ({\bf e}_1 - {\bf e}_2) -  {\bf p}\cdot ({\bf e}_1 - {\bf e}_2) \Big|  \leq 2\cdot 4^{-J} N^{2/3}.
\end{equation}
Provided that $N_0$ is sufficiently large, the collection of up-right paths from ${\bf p}$ to $x^J_{\bf p}$ while remaining inside $R_{0, N/16}^{N^{2/3}}$ has to be non-empty.
This is because by our choice of $J$ from \eqref{J_range}, the estimates \eqref{dia_end} and \eqref{adia_end} imply the diagonal distance between ${\bf p}$ and $x^J_{\bf p}$ is lower bounded  by $N^{1/4}/100$ , while their anti-diagonal distance is upper bounded by $2(N8^{-J})^{2/3} \leq N^{2/9 }.$

Now, for each ${\bf v} \in V_{j,i,k}$, let us define the event 
\begin{equation}\label{oneterm_est}
\mathcal{R}^{{\bf v}}_{j,i,k} = \Big\{\log Z^{\textup{in},  {{h4^{-(j+1)}N^{2/3}}}}_{{\bf v}, {\bf x}^{j}_{(i,k)}}   - (|{\bf x}^{j}_{(i,k)}|_1 - |{\bf v}|_1){f_d} \geq - 2^{-j/5}tN^{1/3} \Big\}.
\end{equation}
Here note that the path constrain in the parallelogram in the definition above also satisfies the global constraint as
$$R_{{\bf v}, {\bf x}^{j}_{(i,k)}}^{{h4^{-(j+1)}N^{2/3}}} \subset R_{0,N}^{N^{2/3}}.$$ 
From \eqref{dia_grid}, \eqref{adia_grid} and the choice for the width of the rectangle $R_{v, {\bf x}^{j}_{(i,k)}/4^j}^{{h4^{-(j+1)}N^{2/3}}}$, by Theorem \ref{wide_similar}, 
$$\mathbb{P}(({\mathcal{R}^{\bf v}_{j,i,k}})^c) = \mathbb{P}\Big(\log Z^{\textup{in},  {{\tfrac{h}{4}(8^{-j}N)^{2/3}}}}_{{\bf v}, {\bf x}^{j}_{(i,k)}}   - (|{\bf x}^{j}_{(i,k)}|_1 - |{\bf v}|_1){f_d} \geq - 2^{-j/5}tN^{1/3} \Big)\leq e^{-C 2^{j/10} t}.$$
Hence, 
$$\mathbb{P}\Big(\cup_{j=0}^{J-1} \cup_{i=1}^{8^j} \cup_{k=1}^{4^j}\cup_{v\in V_{j,i,k}} ({\mathcal{R}^{\bf v}_{j,i,k}})^c\Big) \leq \sum_{j=0}^\infty 100^j e^{-C 2^{j/10} t} \leq e^{-Ct},$$
provided that $t$ is sufficiently large.  

Next, let us define the event 
$$\mathcal{R}_{\textup{start}} = \Big\{100\cdot 8^{-J}N \min_{{\bf z} \in R_{0, N}^{N^{2/3}}}\{  \log  Y_{\bf z} \} \geq - tN^{1/3}\Big\}.$$
Recall that $ N8^{-J} \leq N^{1/3-0.01}$, then 
\begin{align*}
\mathbb{P}((\mathcal{R}_{\textup{start}})^c) 
&\leq \mathbb{P}\Big(\min_{{\bf z} \in R_{0, N}^{N^{2/3}}}\{ \log  Y_{\bf z} \}\leq -  tN^{0.001}\Big)
 \leq N^2 \cdot \mathbb{P}\Big( \log Y_{\bf 0} \leq -  tN^{0.001}\Big) \\
& \leq e^{-CN^{0.001}t} \leq e^{-Ct}.
\end{align*}
Then, on the event $R_{\textup{start}} \cap \big(\cap_{j-0}^{J-1}\cap_{i=1}^{8^j}\cap_{k=1}^{4^j} \cap_{{\bf v} \in V_{j,i,k}} \mathcal{R}^{\bf v}_{j,i,k}\big)$ which has probability at least $1-e^{-Ct}$, we must have 
$$\min_{{\bf p} \in \mathcal{L}_0^{N^{2/3}}} \log Z^{\textup{in}, R_{0,N}^{N^{2/3}}}_{{\bf p}, N} - (2N - |{\bf p}|_1){f_d} \geq - \Big(1+ \sum_{j=0}^\infty 2^{-j/5} \Big)tN^{1/3}.$$
To see this, for each ${\bf p} \in R_{0, N/16}^{N^{2/3}}$, we may go to  $x^J_{{\bf p}}$. Then, from $x^J_{{\bf p}}$, we obtain a sequence of points $x^j_{{\bf p}}$ which traces back to $(N,N)$. Then, we have 
$$\log Z^{\textup{in}, R_{0,N}^{N^{2/3}}}_{{\bf p}, N}  \geq |{\bf p} - x^J_{\bf p}|_1 \min_{{\bf z} \in R_{0, N}^{N^{2/3}}}\{ 0\wedge \log Y_{\bf z} \}  +  \sum_{j=0}^{J-1} \log Z^{\textup{in}, {h4^{-(j+1)N^{2/3}}}}_{x^{j+1}_{\bf p}, x^{j}_{\bf p}} .$$
And on the event $R_{\textup{start}} \cap \big(\cap_{j-0}^{J-1}\cap_{i=1}^{8^j}\cap_{k=1}^{4^j} \cap_{{\bf v} \in V_{j,i,k}} \mathcal{R}^{\bf v}_{j,i,k}\big)$, the right side above is greater than $(2N-|{\bf p}|_1) {f_d} - \left(1+ \sum_{j=0}^\infty 2^{-j/5} \right)tN^{1/3}$. With this, we have finished the proof of \eqref{start_est}.

Finally, the estimate from our theorem simply follows from a union bound using \eqref{start_est}. We rewrite the rectangle $R_{0, 9N/10}^{N^{2/3}}$ as a union of smaller rectangles
$$R_{0, \frac{9}{10}N}^{N^{2/3}} = \bigcup_{k=0}^{143}R_{\frac{kN}{160}, \frac{(k+1)N}{160}}^{N^{2/3}}.$$
Then, \eqref{start_est} can be applied to each one of these rectangles in the form  
$$\mathbb{P} \Big(\min_{{\bf p} \in R_{\frac{kN}{160}, \frac{(k+1)N}{160}}^{N^{2/3}}} \log Z^{\textup{in}, R_{0,N}^{N^{2/3}}}_{{\bf p}, N}  - (2N - |{\bf p}|_1){f_d} \leq - tN^{1/3}\Big) $$
and a union bound finishes the proof.  
\end{proof}



\subsubsection{Proof of Theorem \ref{btl_upper}}

\begin{proof}[Proof of  Theorem \ref{btl_upper}]
This follows from a step-back argument. First, let ${\bf p}^*$ denote the random maximizer of  $$\max_{{\bf p} \in R_{0, \frac{9}{10}N}^{N^{2/3}}}\log Z_{{\bf p},\mathcal{L}_{N}}  - (2N - |{\bf p}|_1) {f_d}.$$
Then, 
$$\log Z_{-N, \mathcal{L}_N}  \geq \log Z_{-N, {\bf p}^*}  + \log Z_{{\bf p}^*, \mathcal{L}_N} .$$
By a union bound, 
\begin{align*}
&\mathbb{P}\Big(\log Z_{{\bf p}^*, \mathcal{L}_N}  - (2N - |{\bf p}^*|_1){f_d} \geq tN^{1/3}\Big)\\
& \leq \mathbb{P}\Big( \log Z_{-N, \mathcal{L}_N}  - 4N{f_d} \geq \tfrac{1}{2}tN^{1/3}\Big) \\
& \qquad \qquad + \mathbb{P}\Big(\log Z_{-N, {\bf p}^*}  - (2N + |{\bf p}^*|_1) {f_d} \leq -\tfrac{1}{2}tN^{1/3}\Big).
\end{align*}
The two probabilities are bounded by $e^{-Ct}$ by Proposition \ref{low_ub} and Theorem \ref{high_inf}. 
\end{proof}

\section{Proof of the random walk comparison in Section \ref{sec_rw}}\label{rwrw_proof}

\begin{figure}[t]
\captionsetup{width=0.8\textwidth}
\begin{center}

\begin{tikzpicture}[>=latex, scale=0.8]

\draw[line width=1mm, lightgray, ->] (-1,-1) -- (-1,7);
\draw[line width=1mm, lightgray, ->] (-1,-1) -- (9,-1);

\draw[line width=0.3mm, ->] (0,0) -- (8,0);
\draw[line width=0.3mm, ->] (0,0) -- (0,6);

\draw[line width=0.3mm, loosely dotted, ->] (6.5,6.5) -- (1.8,-1.8);
\node at (2.5,3) {\small$-{\boldsymbol\xi}[\lambda]$-directed};


\draw[ fill=white](0,0)circle(1mm);
\node at (-0.3,-0.3) {\small$0$};

\draw[line width=0.3mm, black] (6.5,6.5) -- (6.5,6.5- 0.4) -- (6.5+ 0.3,6.5- 0.4)-- (6.5+ 0.3,6.5- 0.8)-- (6.5+ 0.8,6.5- 0.8)-- (6.5+ 0.8,6.5- 1.1)-- (6.5+ 1.4,6.5- 1.1)--(6.5+ 1.4,6.5- 1.6);

\fill[color=white] (6.5,6.5)circle(1.7mm);
\draw[ fill=lightgray](6.5,6.5)circle(1mm);
\node at (6.8,6.8) {\small ${\bf z_0}$};

\fill[color=white] (6.5+ 0.8,6.5- 0.8)circle(1.7mm);
\draw[ fill=lightgray](6.5+ 0.8,6.5- 0.8)circle(1mm);
3\node at (6.5+ 0.8+ 0.3,6.5- 0.8+ 0.5) {\small ${\bf v}_N$};

\node at (7.6,4.8) {\small $\Theta_k$};

\fill[color=white] (1,-1)circle(1.7mm);
\draw[ fill=lightgray](1,-1)circle(1mm);
\node at (-0.1,-1.6) {\small $(-1, sN^{2/3})$};

\fill[color=white] (-1,-1)circle(1.7mm);
\draw[ fill=lightgray](-1,-1)circle(1mm);
\node at (-2.3,-1) {\small $(-1,-1)$};

\node at (4,-1.5) {\small $\textup{Ga}^{-1}(\mu - \lambda)$};
\node at (-2.3,3) {\small $\textup{Ga}^{-1}(\lambda)$};

\end{tikzpicture}
	\end{center}	
	\caption{\small The step up in the proof of Theorem \ref{rwrw}.}
	\label{fig1}
\end{figure}
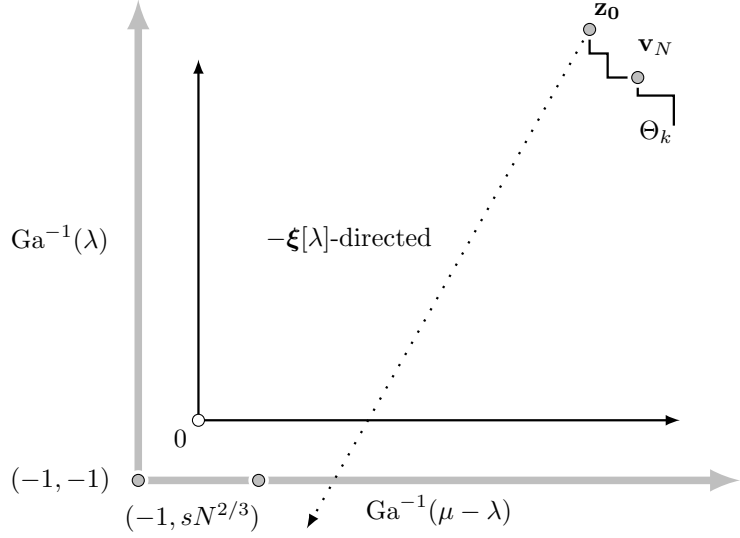

\begin{proof}[Proof of Theorem \ref{rwrw}]
We will construct the upper bound $X_i$. The construction of $Y_i$ follows from a similar argument, which we sketch at the end of the proof. In addition, we will assume that the partition functions include the weight $Y_{(0,0)}$, which does not change the profile.

To start, recall the profile that we are looking at is along $\Theta_k = \{{\bf z}_0, \dots, {\bf z}_k\}$. Let us first fix an $a_0$ sufficiently small that 
${\bf z}_0 \cdot {\bf e}_1 \geq \tfrac{1}{2} {\bf v}_N \cdot {\bf e}_1$.
Next, we will fix the constant $q_0$, and the idea is illustrated in Figure \ref{fig1}. 
Recall $\lambda =  \rho + q_0 sN^{-1/3}.$  
By Proposition \ref{slope1}, for $q_0 > 0$, 
$$\textup{the slope of the vector } \boldsymbol\xi[\lambda]  \geq m_\rho(0)+ c q_0r N^{-1/3}.$$ 
This means increasing $q_0$ will make the dotted line appearing in Figure \ref{fig1} more vertical.
Then, because the slope between $(0,0)$ and $v_N$ is $m_\rho(0)$ and
$|{\bf z}_0 - {\bf v}_N|_\infty \leq sN^{2/3}$, there exists a positive constant ${q_0}$ sufficiently large such that the $-{\boldsymbol\xi}[\lambda]$-directed ray starting at ${\bf z}_0$ (the dotted line) will cross the horizontal line $y = -1$ on the right of the vertical line $x =  sN^{2/3}$, as shown in Figure \ref{fig1}.

Once $q_0$ is fixed,  we may lower the value of $a_0$ further if necessary, so the parameters $\lambda$ and $\mu-\lambda$ are both contained inside $[\tfrac{\epsilon}{2}, \mu-\tfrac{\epsilon}{2}]$. We will place $\textup{Ga}^{-1}(\mu-\lambda)$ and $\textup{Ga}^{-1}(\lambda)$ on the ${\bf e}_1$- and ${\bf e}_2$-boundaries based at the base $(-1,-1)$.  

By Theorem \ref{exit_time}, we have 
\begin{equation}\label{exit1}
\mathbb{P}\Big(Q^{\lambda}_{-1, {\bf z}_0}\{\tau \leq -1\} \geq 1/10\Big) \leq e^{-Cs^3},
\end{equation} 
and let us define the complement of \eqref{exit1} as 
$$A = \Big\{Q^{\lambda}_{-1, {\bf z}_0}\{\tau \geq 1\} \geq 9/10\Big\}.$$
Let $Z^{\lambda, \textup{south}}_{(0,-1), \bbullet}$ be the partition function that uses the same weights as $Z^{\lambda}_{-1, \bbullet}$, except that $Z^{\lambda, \textup{south}}_{(0,-1), \bbullet}$ does not see or use any of the weights on the vertical boundary along $x= -1$.  Then, for each $j = 1, \dots, k$, we can upper bound $\log Z_{0, {\bf z}_{j}} - \log Z_{0, {\bf z}_{0}}$ as follows,
\begin{align*}
e^{\log Z_{0, {\bf z}_{j}} - \log Z_{0, {\bf z}_{0}}} &= \frac{Z_{0, {\bf z}_j}}{ Z_{0, {\bf z}_{0}}}  = \prod_{i=1}^j\frac{Z_{0, {\bf z}_i}}{ Z_{0, {\bf z}_{i-1}}} \\
\textup{by Proposition \ref{monot1}} \quad &\leq \prod_{i=1}^j\frac{Z^{\lambda, \textup{south}}_{{(0,-1), {\bf z}_i}}}{ Z^{\lambda, \textup{south}}_{(0,-1), {\bf z}_{i-1}}}
 =  \frac{Z^{\lambda, \textup{south}}_{{(0,-1), {\bf z}_j}}}{ Z^{\lambda, \textup{south}}_{(0,-1), {\bf z}_{0}}}\cdot \frac{I^{\lambda,\textup{south}}_{[\![(-1,-1), (0,-1)]\!]}}{I^{\lambda,\textup{south}}_{[\![(-1,-1), (0,-1)]\!]}}\\
 &= \frac{Z^{\lambda }_{-1, {\bf z}_{j}}(\tau \geq 1)}{ Z^{\lambda}_{-1,{\bf z}_{0}}(\tau \geq 1)}
= \frac{Q^{\lambda }_{-1, {\bf z}_{j}}(\tau\geq  1)}{Q^{\lambda }_{-1,{\bf z}_{0}}(\tau\geq  1)}\cdot \frac{Z^{\lambda }_{-1, {\bf z}_{j}}}{Z^{\lambda }_{-1,{\bf z}_{0}}}\\
\textup{ on the event $A$ } \quad & \leq \frac{10}{9}\frac{Z^{\lambda }_{-1, {\bf z}_{j}}}{Z^{\lambda }_{-1, {\bf z}_{0}}}.
\end{align*}
Choosing $X_i = \log\frac{Z^{\lambda }_{-1, {\bf z}_{i}}}{Z^{\lambda }_{-1, {\bf z}_{i-1}}}$  for $i = 1, \dots, j$ finishes the proof of the upper bound. Note the distributional proprieties of $X_i$ are guaranteed by Theorem \ref{stat}.

For the lower bound, by increasing the value $q_0$ if necessary , the $-{\boldsymbol\xi}[\eta]$ directed ray starting from ${\bf z}_k$ will hit the vertical line $x=-1$ above the horizontal line $y = sN^{2/3}$. We place $\textup{Ga}^{-1}(\mu-\eta)$ and $\textup{Ga}^{-1}(\mu)$ on the ${\bf e}_1$- and ${\bf e}_2$-boundaries based at the base $(-1,-1)$. Then by Theorem  \ref{exit_time}, we have
$$\mathbb{P}\Big(Q^{\eta}_{-1, {\bf z}_k}\{\tau \geq 1\} \geq 1/10\Big) \leq e^{-Cs^3}.$$
Let us define the complement of the event above as 
$$B = \Big\{Q^{\eta}_{-1, {\bf z}_k}\{\tau \leq -1\} \geq 9/10\Big\}.$$
By Proposition \ref{monot2},
it holds that for $j = 1, \dots, k-1$,
$$Q^{\eta}_{-1, {\bf z}_j}\{\tau \leq -1\} \geq Q^{\eta}_{-1, {\bf z}_k}\{\tau \leq -1\}.$$
Then, for each $j = 1, \dots k$, we lower bound $\log Z_{0, {\bf z}_{j}} - \log Z_{0, {\bf z}_{0}}$ as follows,
\begin{align*}
e^{\log Z_{0, {\bf z}_{j}} - \log Z_{0, {\bf z}_{0}}} &= \frac{Z_{0, {\bf z}_j}}{ Z_{0, {\bf z}_{0}}} = \prod_{i=1}^j\frac{Z_{0, {\bf z}_i}}{ Z_{0, {\bf z}_{i-1}}} \\
\textup{by Proposition \ref{monot1}} \quad &\geq \prod_{i=1}^j\frac{Z^{\eta, \textup{west}}_{{(-1,0), {\bf z}_i}}}{ Z^{\eta, \textup{west}}_{(-1,0), {\bf z}_{i-1}}}
 =  \frac{Z^{\eta, \textup{west}}_{{(-1,0), {\bf z}_j}}}{ Z^{\eta, \textup{west}}_{(-1,0), {\bf z}_{0}}}\cdot \frac{J^{\eta,\textup{west}}_{[\![(-1,-1), (-1,0)]\!]}}{J^{\eta,\textup{west}}_{[\![(-1,-1), (-1,0)]\!]}}\\
 &= \frac{Z^{\eta }_{-1, {\bf z}_{j}}(\tau \leq -1)}{ Z^{\eta}_{-1,{\bf z}_{0}}(\tau \leq -1)}
= \frac{Q^{\eta }_{-1, {\bf z}_{j}}(\tau\leq  -1)}{Q^{\eta }_{-1,{\bf z}_{0}}(\tau\leq  -1)}\cdot \frac{Z^{\eta }_{-1, {\bf z}_{j}}}{Z^{\eta }_{-1,{\bf z}_{0}}}\\
\textup{ on the event $B$ } \quad & \geq \frac{9}{10}\frac{Z^{\eta }_{-1, {\bf z}_{j}}}{Z^{\eta }_{-1, {\bf z}_{0}}}.
\end{align*}
Choosing $Y_i = \log\frac{Z^{\eta }_{-1, {\bf z}_{i}}}{Z^{\eta }_{-1, {\bf z}_{i-1}}}$ for $i = 1, \dots, k$ will give us the desired lower bound.

\end{proof}

\section{Monotonicity for the polymer model}
The following two propositions hold for arbitrary positive weights on the lattice, and there is no probability involved. The first proposition is Lemma A.2 from \cite{Bus-Sep-22}, and the second proposition is Lemma A.5 from \cite{ras-sep-she-}.

\begin{proposition} \label{monot1}
Let ${\bf x}, {\bf y}, {\bf z}\in \mathbb{Z}^2$ be such that ${\bf x}\cdot {\bf e}_1 \leq {\bf y} \cdot {\bf e}_1$, ${\bf x}\cdot {\bf e}_2 \geq {\bf y} \cdot {\bf e}_2$, and coordinate wise ${\bf x}, {\bf y} \leq {\bf z}$ , then $$
\frac{Z_{{\bf x}, {\bf z}}}{Z_{{\bf x}, {\bf z}-{\bf e}_1}} \leq \frac{Z_{{\bf y},{\bf z}}}{Z_{{\bf y}, {\bf z}-{\bf e}_1}} \qquad \text{ and } \qquad \frac{Z_{{\bf x}, {\bf z}}}{Z_{{\bf x}, {\bf z}-{\bf e}_2}} \geq \frac{Z_{{\bf y},{\bf z}}}{Z_{{\bf y}, {\bf z}-{\bf e}_2}}.
$$
\end{proposition}


\begin{proposition} \label{monot2}
For any $k, l, m\in \mathbb{Z}_{\geq 0}$ and ${\bf z}\in \mathbb{Z}^2_{\geq 0}$,
$$Q_{0, {\bf z}} \{\tau \geq k \} \leq Q_{0, {\bf z}+l{\bf e}_1-m{\bf e}_2} \{\tau \geq k \}.$$
\end{proposition}


\section{Sub-exponential random variables}\label{sec_sub_exp}

First, we state a general result for the running maximum of sub-exponential random variables.
Recall that a random variable $X_1$ is sub-exponential if there exist two positive constants $K_0$ and $\lambda_0$ such that 
\be\label{sub_exp}
\log(\mathbb{E}[e^{\lambda (X_1-\mathbb{E}[X_1])}]) \leq K_0 \lambda^2 \quad \textup{ for $\lambda \in [0, \lambda_0]$}.
\ee
Let $\{X_i\}$ be a sequence of i.i.d.~sub-exponential random variables with the parameters $K_0$ and $\lambda_0$. Define $S_0 = 0$ and $S_k = X_1 + \dots + X_k - k\mathbb{E}[X_1]$ for $k\geq 1$. The following theorem captures the right tail behavior of the running maximum.
\begin{theorem}\label{max_sub_exp}
Let the random walk $S_k$ be defined as above. Then,
$$\mathbb{P} \Big(\max_{0\leq k \leq n} S_k \geq t\sqrt{n}\Big) \leq 
\begin{cases}
e^{-t^2/(4K_0)} \quad  & \textup{if $t \leq 2\lambda_0 K_0 \sqrt n$} \\
e^{-\frac{1}{2}\lambda_0 t\sqrt{n}} \quad  & \textup{if $t \geq 2\lambda_0 K_0 \sqrt n$}
\end{cases}.
$$ 
\end{theorem}
\begin{proof}
Since $S_k$ is a mean zero random walk, then $e^{\lambda S_k}$ is a non-negative sub-martingale for $\lambda \geq 0$. By Doob's maximal inequality, 
\begin{align*}
\mathbb{P} \Big(\max_{0\leq k \leq n} S_k \geq t\sqrt{n}\Big)  &= \mathbb{P} \Big(\max_{0\leq k \leq n} e^{\lambda S_k} \geq e^{\lambda t\sqrt{n}}\Big) 
\leq \frac{\mathbb{E}[e^{\lambda S_n}]}{e^{\lambda t\sqrt{n}}}
= \frac{\mathbb{E}[e^{\lambda X_1}]^n}{e^{\lambda t\sqrt{n}}}. 
\end{align*}
Taking the logarithm of the expression above, and using our assumption \eqref{sub_exp} for $X_1$, we obtain 
$$\log\Big( \frac{\mathbb{E}[e^{\lambda X_1}]^n}{e^{\lambda t\sqrt{n}}}\Big) = n\log(\mathbb{E}[e^{\lambda X_1}]) - \lambda t\sqrt{n} \leq nK_0 \lambda^2 -  \lambda t\sqrt{n} \textup{ \quad for $\lambda \in [0, \lambda_0]$}.$$
Let us denote the quadratic quadratic function in $\lambda \in [0, \lambda_0]$ as 
$$h(\lambda) = nK_0 \lambda^2 -  \lambda t\sqrt{n} .$$
And note the minimizer of $h$ is given by 
$\lambda^\textup{min}_t = \min\{\lambda_0, \frac{t}{2K_0 \sqrt{n}}\}$, and 
$$h(\lambda^\textup{min}_t) =  
\begin{cases}
-\frac{t^2}{4K_0} \quad  & \textup{if $t \leq 2\lambda_0K_0 \sqrt{n}$}\\
nK_0\lambda_0^2 - \lambda_0 t\sqrt{n}  \leq -\frac{1}{2}\lambda_0 t\sqrt{n} \quad  & \textup{if $t \geq 2\lambda_0K_0 \sqrt{n}$}
\end{cases}.
$$ 
With this, we have finished the proof of this proposition. 
\end{proof}

Our next proposition shows that both $\log(\textup{Ga})$ and $\log(\textup{Ga}^{-1}) = - \log(\textup{Ga})$ are sub-exponential random variables. 
\begin{proposition}\label{Ga_sub_exp}
Fix $\epsilon \in (0, \mu/2)$. There exists positive constants $K_0, \lambda_0$ depending on $\epsilon$ such that for each $\alpha \in [\epsilon, \mu-\epsilon]$, let $X\sim \textup{Ga}(\alpha)$ and we have
$$
\log(\mathbb{E}[e^{\pm\lambda(\log(X) - \Psi_1(\alpha))}]) \leq K_0 \lambda^2 \qquad \textup{ for $\lambda \in [0, \lambda_0]$}.
$$
\end{proposition}
\begin{proof}
First, note that $\mathbb{E}[X^{\pm \lambda}] = \frac{\Gamma(\alpha\pm\lambda)}{\Gamma(\alpha)}$, provided that $\alpha \pm \lambda > 0$. 
Then, the proof essentially follows from Taylor's theorem,
\begin{align*}
\log(\mathbb{E}[e^{\pm\lambda(\log(X) - \Psi_1(\alpha))}]) & = \log(\mathbb{E} [X^{\pm\lambda}]e^{\mp\lambda \Psi_1(\alpha)})\\
 \qquad & = \log(\Gamma(\alpha \pm \lambda)) -[\log(\Gamma(\alpha)) \pm \lambda \Psi_1(\alpha)]\\
\textup{(recall $\log(\Gamma(\alpha))' = \Psi_1(\alpha)$) } \quad & =\frac{\Psi_1'(\alpha)}{2}\lambda^2 + o(\lambda^2)\\
& \leq K_0 \lambda^2
\end{align*}
provided $\lambda_0$ is fixed sufficiently small. The constant $K_0$ can be chosen uniformly for all $\alpha$ from the compact interval $[\epsilon, \mu-\epsilon]$ because $\Phi_1$ is a smooth function on $\mathbb{R}_{\geq 0}$.
\end{proof}

\begin{proof}[Proof of \eqref{rw_before_E}]

First, let us normalize the $\wt{S}_k$ by its expectation, and let us denote the new walk by $\overline{S}_k$. The expectation of the step of $\wt{S}_k$ is 
\begin{align*}
&-\Psi_0(\eta) + \Psi_0(\mu-\eta) \\
& = -\Psi_0(\mu/2 - q_0 t^{2/3}N^{-1/3}) + \Psi_0(\mu/2 + q_0 t^{2/3}N^{-1/3}) \\
& \leq c_1 t^{2/3}N^{-1/3}
\end{align*}
provided $N_0$ is sufficiently large and $c_0$ is sufficiently small. Then, 
$$\mathbb{E}\big[\wt{S}_k\big] \leq a c_1 t^{2/3}N^{-1/3} \leq c_1 t \sqrt{a} .$$
By fixing $C' = 2c_1$ in \eqref{rw_before_E}, we see that 
$$\eqref{rw_before_E} \leq \mathbb{P} \Big(\max_{0\leq k \leq a} \overline{S}_k  \geq c_1t \sqrt{a} \Big).$$

Since the sum of two independent sub-exponential random variables is still a sub-exponential random variable, this fact together with Proposition \ref{Ga_sub_exp} shows that the steps of $\overline S_k$ are sub-exponential. Now, we may apply the right tail bound on the running maximum from Theorem \ref{max_sub_exp} to the term 
$$\mathbb{P} \Big(\max_{0\leq k \leq a} \overline{S}_k  \geq c_1 t \sqrt{a}\Big),$$ 
and this finishes the proof. 
\end{proof}

\section{Random walk estimate}
First, let us recall two results from \cite{Nag-70}.
Let $\{X_i\}_{i\in \mathbb{Z}_{>0}}$ be an i.i.d.~sequence of random variables with
$$\mathbb{E}[X_i] = \mu, \quad \Var[X_i] = 1 \quad \text{and} \quad c_3 = \mathbb{E}|X-\mu|^3 < \infty.$$ Define $S_k = \sum_{i=1}^k X_i$ with $S_0 = 0$.  
\begin{lemma}[{\cite[Lemma 5]{Nag-70}}]
There exists an absolute constant $C$ such that for any $l>0$
\beq\mathbb{P}\Big(\max_{1\leq k \leq N} S_k < l\Big) - \mathbb{P}\Big(\max_{1\leq k \leq N} S_k < 0\Big) \leq C(c_3l+c_3^2)(|\mu| + 1/\sqrt N).\eeq
\end{lemma}

\begin{lemma}[{\cite[Lemma 7]{Nag-70}}]
There exists an absolute constant $C$ such that 
\beq\mathbb{P}\Big(\max_{1\leq k \leq N} S_k < 0\Big) \leq Cc^2_3(|\mu| + 1/\sqrt N).\eeq
\end{lemma}

Combining them, we obtain the following proposition.

\begin{proposition}\label{rwest} There exists an absolute constant $C$ such that for any $l \geq 0$,
\beq\mathbb{P}\Big(\max_{1\leq k \leq N} S_k < l\Big) \leq C(c_3l+c_3^2)(|\mu| + 1/\sqrt N).\eeq
\end{proposition}

\bibliographystyle{amsplain}
\bibliography{time}

\end{document}